\documentclass{aart}

\usepackage{amsmath} 
\usepackage{amssymb}
\usepackage{amsfonts}
\usepackage{latexsym}
\usepackage{amsthm}
\usepackage{amsxtra}
\usepackage{amscd}
\usepackage{bbm}
\usepackage{mathrsfs}
\usepackage{bm}

\usepackage{cite}

\usepackage{graphicx}

\usepackage{color}

\numberwithin{equation}{section}
\numberwithin{figure}{section}

\newcommand{\where}{\qquad \text{where} \quad}

\newcommand{\hp}[1]{with $#1$-high probability}

\renewcommand{\b}[1]{\boldsymbol{\mathrm{#1}}} 
\newcommand{\bb}{\mathbb} 
\renewcommand{\cal}{\mathcal} 
 
\newcommand{\fra}{\mathfrak} 
\newcommand{\ol}[1]{\overline{#1} \!\,} 
\newcommand{\wh}{\widehat}
\newcommand{\wt}{\widetilde}

\newcommand{\me}{\mathrm{e}}
\newcommand{\ii}{\mathrm{i}}
\newcommand{\dd}{\mathrm{d}}
\newcommand{\col}{\mathrel{\vcenter{\baselineskip0.75ex \lineskiplimit0pt \hbox{.}\hbox{.}}}}
\newcommand*{\deq}{\mathrel{\vcenter{\baselineskip0.65ex \lineskiplimit0pt \hbox{.}\hbox{.}}}=}
\newcommand*{\eqd}{=\mathrel{\vcenter{\baselineskip0.65ex \lineskiplimit0pt \hbox{.}\hbox{.}}}}
\newcommand{\eqdist}{\overset{d}{=}}
\newcommand{\simd}{\overset{d}{\sim}}
\newcommand{\umat}{\mathbbmss{1}} 
\renewcommand{\leq}{\leqslant}
\renewcommand{\geq}{\geqslant}
\renewcommand{\le}{\leqslant}
\renewcommand{\ge}{\geqslant}

\newcommand{\ind}[1]{\b 1 (#1)}
\newcommand{\indb}[1]{\b 1 \pb{#1}}

\renewcommand{\epsilon}{\varepsilon}

\renewcommand{\P}{\mathbb{P}}
\newcommand{\E}{\mathbb{E}}
\newcommand{\R}{\mathbb{R}}
\newcommand{\C}{\mathbb{C}}
\newcommand{\N}{\mathbb{N}}

\newcommand{\p}[1]{({#1})}
\newcommand{\pb}[1]{\bigl({#1}\bigr)}
\newcommand{\pB}[1]{\Bigl({#1}\Bigr)}
\newcommand{\pbb}[1]{\biggl({#1}\biggr)}
\newcommand{\pBB}[1]{\Biggl({#1}\Biggr)}

\newcommand{\q}[1]{[{#1}]}
\newcommand{\qb}[1]{\bigl[{#1}\bigr]}
\newcommand{\qB}[1]{\Bigl[{#1}\Bigr]}
\newcommand{\qbb}[1]{\biggl[{#1}\biggr]}
\newcommand{\qBB}[1]{\Biggl[{#1}\Biggr]}

\newcommand{\h}[1]{\{{#1}\}}
\newcommand{\hb}[1]{\bigl\{{#1}\bigr\}}
\newcommand{\hB}[1]{\Bigl\{{#1}\Bigr\}}
\newcommand{\hbb}[1]{\biggl\{{#1}\biggr\}}
\newcommand{\hBB}[1]{\Biggl\{{#1}\Biggr\}}

\newcommand{\abs}[1]{\lvert #1 \rvert}
\newcommand{\absb}[1]{\bigl\lvert #1 \bigr\rvert}
\newcommand{\absB}[1]{\Bigl\lvert #1 \Bigr\rvert}
\newcommand{\absbb}[1]{\biggl\lvert #1 \biggr\rvert}
\newcommand{\absBB}[1]{\Biggl\lvert #1 \Biggr\rvert}

\newcommand{\norm}[1]{\lVert #1 \rVert}
\newcommand{\normb}[1]{\bigl\lVert #1 \bigr\rVert}

\newcommand{\scalar}[2]{\langle{#1} \mspace{2mu}, {#2}\rangle}
\newcommand{\scalarb}[2]{\bigl\langle{#1} \mspace{2mu}, {#2}\bigr\rangle}

\DeclareMathOperator{\diag}{diag}
\DeclareMathOperator{\tr}{Tr}

\DeclareMathOperator{\re}{Re}
\DeclareMathOperator{\im}{Im}

\DeclareMathOperator{\dist}{dist}
\DeclareMathOperator{\diam}{diam}

\newcommand{\msc}{m}

\newcommand{\beqa}{\begin{eqnarray}}
\newcommand{\eeqa}{\end{eqnarray}}

\newcommand{\be}{\begin{equation}}
\newcommand{\ee}{\end{equation}}

\theoremstyle{plain} 
\newtheorem{theorem}{Theorem}[section]
\newtheorem*{theorem*}{Theorem}
\newtheorem{lemma}[theorem]{Lemma}
\newtheorem*{lemma*}{Lemma}

\newtheorem*{corollary*}{Corollary}
\newtheorem{proposition}[theorem]{Proposition}
\newtheorem*{proposition*}{Proposition}
\newtheorem{definition}[theorem]{Definition}
\newtheorem*{definition*}{Definition}

\theoremstyle{definition} 

\newtheorem*{example*}{Example}
\newtheorem{remark}[theorem]{Remark}

\newtheorem*{remark*}{Remark}
\newtheorem*{remarks*}{Remarks}

\title{The Isotropic Semicircle Law and Deformation of Wigner Matrices}

\author{
Antti Knowles${}^1$\thanks{Partially supported by NSF grant DMS-0757425} \; and \; Jun Yin${}^2$\thanks{Partially 
supported by NSF grant DMS-1001655} \\\\\\
Department of Mathematics, Harvard University \\
Cambridge MA 02138, USA \\ knowles@math.harvard.edu${}^1$ \\\\
Department of Mathematics, University of Wisconsin \\
Madison WI 53706, USA \\
jyin@math.wisc.edu${}^2$ \\ \\}
\begin{document}


\maketitle


\begin{abstract}
We analyse the spectrum of additive finite-rank deformations of $N \times N$ Wigner matrices $H$. The spectrum of the deformed matrix undergoes a transition, associated with the creation or annihilation of an outlier, when an eigenvalue $d_i$ of the deformation crosses a critical value $\pm 1$. This transition happens on the scale $\abs{d_i} - 1 \sim N^{-1/3}$. We allow the eigenvalues $d_i$ of the deformation to depend on $N$ under the condition $\absb{\abs{d_i} - 1} \geq (\log N)^{C \log \log N} N^{-1/3}$.
We make no assumptions on the eigenvectors of the deformation. In the limit $N \to \infty$, we identify the law of the outliers and prove that the non-outliers close to the spectral edge have a universal distribution coinciding with that of the extremal eigenvalues of a Gaussian matrix ensemble.

A key ingredient in our proof is the \emph{isotropic local semicircle law}, which establishes optimal high-probability bounds on the quantity $\scalarb{\b v}{\pb{(H - z)^{-1} - m(z) \umat} \b w}$, where $m(z)$ is the Stieltjes transform of Wigner's semicircle law and $\b v, \b w$ are arbitrary deterministic vectors.
\end{abstract}

\vspace{1cm}


\medskip

\medskip


\newpage

\section{Introduction}

Random matrices were introduced by Wigner \cite{Wig} in the 1950s to model the excitation spectra of large atomic nuclei, and have since been the subject of intense mathematical investigation. In this paper we study Wigner matrices -- random matrices whose entries are independent up to symmetry constraints -- that have been deformed by a finite-rank perturbation. By Weyl's eigenvalue interlacing inequalities, such a deformation does not influence the global statistics of the eigenvalues. Thus, the empirical eigenvalue densities of deformed and undeformed Wigner matrices have the same large-scale asymptotics, and are governed by Wigner's famous semicircle law. However, the behaviour of individual eigenvalues may change dramatically under a deformation. In particular, deformed Wigner matrices may exhibit \emph{outliers}, eigenvalues located away from the bulk spectrum. Such models were first investigated by F\"uredi and Koml\'os \cite{FKoml}. Subsequently, much progress \cite{SoshPert, FP, CDMF1, CDMF2, CDMF3, BGN, BGGM1, BGGM2} has been made in the analysis of the spectrum of such deformed matrix models. See e.g.\ \cite{SoshPert} for a review of recent developments.  Analogous deformations of covariance matrices, so-called \emph{spiked population models}, as well as generalizations thereof, were studied in \cite{BY1,BY2,BS}.

In a seminal work \cite{BBP}, Baik, Ben Arous, and P\'ech\'e investigated the spectrum of deformed (spiked) complex Gaussian sample covariance matrices. They established a phase transition, sometimes referred to as the \emph{BBP transition}, in the distribution of the extremal eigenvalues. In \cite{Pec}, P\'ech\'e proved a similar result for additive deformations of GUE (the Gaussian Unitary Ensemble). Subsequently, the results of \cite{BBP} and \cite{Pec} were extended to the other Gaussian ensembles, such as GOE (the Gaussian Orthogonal Ensemble), by Bloemendal and Vir\'ag \cite{BV1,BV2}. We sketch the results of \cite{BBP, Pec, BV1,BV2} in the case of additive deformations of GUE. For simplicity, we consider rank-one deformations, although the results of \cite{BBP, Pec, BV1,BV2} cover arbitrary rank-$k$ deformations. Thus, let $H$ be an $N \times N$ GUE matrix, normalized so that its entries have variance $N^{-1}$. Let $\wt H(d) \deq H + d \b v \b v^*$, where $\b v$ is a normalized vector and $d$ is independent of $N$. If $d>1$ then the spectrum of $\wt H(d)$ consists of a \emph{bulk spectrum} asymptotically contained in $[-2,2]$, and an \emph{outlier}, located at $d + d^{-1}$ and having a normal law with variance of order $N^{-1}$. If $d < 1$ then there is no such outlier, and the statistics of the extremal eigenvalues of $\wt H(d)$ coincide with those of $H$. Thus, as $d$ increases from $1 - \epsilon$ to $1 + \epsilon$ for some small $\epsilon > 0$, the largest eigenvalue of $\wt H(d)$ detaches itself from the bulk spectrum and becomes an outlier.

The phase transition takes place on the scale $d = 1 + w N^{-1/3}$ where $w$ is of order one. This may be heuristically understood  as follows. The largest eigenvalues of $H$ are known to fluctuate on the scale $N^{-2/3}$ around $2$. The critical scale for $d$, i.e.\ the scale on which the outlier is separated from $2$ by a gap of order $N^{-2/3}$, is therefore $d = 1 + w N^{-1/3}$ (since in that case $d + d^{-1} = 2 + w^2 N^{-2/3} + O(w^3 N^{-1})$). In \cite{BBP, Pec, BV1,BV2}, the authors established the weak convergence as $N \to \infty$
\begin{equation*}
N^{2/3} \pB{\lambda_N \pb{\wt H(1 + w N^{-1/3})} - 2} \;\Longrightarrow\; \Lambda_w\,,
\end{equation*}
where $\lambda_N(A)$ denotes the largest eigenvalue of $A$. Moreover, the asymptotics in $w$ of the law $\Lambda_w$ was analysed in \cite{BBP, Pec, BV1, BV2, Bthesis}: as $w \to + \infty$, the law $\Lambda_w$ converges to a Gaussian; as $w \to - \infty$, the law $\Lambda_w$ converges to the Tracy-Widom-$\beta$ distribution (where $\beta = 1$ for GOE and $\beta = 2$ for GUE). As mentioned above, the results of \cite{BBP, Pec, BV1,BV2} also apply to rank-$k$ deformations, where the picture is similar; each eigenvalue $d_i \in [-1,1]^c$ gives rise to an outlier located around $d_i + d_i^{-1}$, while eigenvalues $d_i \in (-1,1)$ do not change the statistics of the extremal eigenvalues of $\wt H$.

The proofs of \cite{BBP, Pec} use an asymptotic analysis of Fredholm determinants, while those of \cite{BV1,BV2} use an explicit tridiagonal representation of $H$; both of these approaches rely heavily on the Gaussian nature of $H$. In order to study the phase transition for non-Gaussian matrix ensembles, and in particular address the question of spectral universality, a different approach is needed. Interestingly, it was observed in \cite{CDMF1, CDMF2, CDMF3} that the distribution of the outliers is not universal, and may depend on the geometry of the eigenvectors of $A$. The non-universality of the outliers was further investigated in \cite{SoshPert}.

In the present paper we take $H$ to be a real symmetric or complex Hermitian Wigner matrix, and $A$ to be a rank-$k$ deterministic matrix whose symmetry class (real symmetric or complex Hermitian) coincides with that of $H$. We make the following assumptions on the perturbation $A$.
\begin{itemize}
\item[(A1)]
The eigenvalues $d_1, \dots, d_k$ of $A$ may depend on $N$; they satisfy $\absb{\abs{d_i} - 1} \geq (\log N)^{C \log \log N} N^{-1/3}$, i.e., on the scale of the phase transition, the eigenvalues of $A$ are separated from the transition points by at least a logarithmic factor.
\item[(A2)]
The eigenvectors of $A$ are arbitrary orthonormal vectors.
\end{itemize}
Our main results on the spectrum of $H + A$ may be informally summarized as follows.
\begin{itemize}
\item[(R1)]
The non-outliers ``stick'' to eigenvalues of the undeformed matrix $H$ (Theorem \ref{theorem: rank-k lde}). In particular, the extremal bulk eigenvalues of $H + A$ are universal.
\item[(R2)] We identify the distribution of the outliers of $H+A$ (Theorem \ref{theorem: outlier distributions}).
\end{itemize}

A key ingredient in our proof is a generalization of the \emph{local semicircle law}. The study of the local semicircle law was initiated in \cite{ESY1, ESY3}; it provides a key step towards establishing universality for Wigner matrices \cite{ESY4, ESY6, EYY2, EYY3,  TV1,TV2}.   The strongest versions of the local semicircle law, proved  in \cite{EYY3, EKYY1, EKYY2}, give precise estimates on the local eigenvalue density, down to scales containing $N^\epsilon$ eigenvalues.
In fact, as formulated in \cite{EYY3}, the local semicircle law gives optimal high-probability estimates on the quantity
\begin{equation} \label{old lsc}
G_{ij}(z) - \delta_{ij} m(z)\,,
\end{equation}
where $m(z)$ denotes the Stieltjes transform of Wigner's semicircle law and $G(z) = (H - z)^{-1}$ is the resolvent of $H$. Starting from such estimates on \eqref{old lsc}, the two following facts are established in \cite{EYY3}.
\begin{enumerate}
\item
The eigenvalue density is governed by Wigner's semicircle law down to scales containing $N^\epsilon$ eigenvalues.
\item
\emph{Eigenvalue rigidity}: optimal high-probability bounds on the eigenvalue locations. 
\end{enumerate}

Another key ingredient in the proof of universality of random matrices is the \emph{Green function comparison method} introduced in \cite{EYY2}. It  uses a Lindeberg replacement strategy, which previously appeared in the context of random matrix theory in  \cite{Chat, TV1,TV2}. A fundamental input in the Green function comparison method is a precise control on the matrix entries of $G$, which is provided by the local semicircle law. The Green function comparison method has subsequently been applied to proving the spectral universality of adjacency matrices of random graphs \cite{EKYY1,EKYY2} as well as the universality of eigenvectors of Wigner matrices \cite{KY1}.

In this paper, we extend the local semicircle law to the \emph{isotropic local semicircle law}, which gives optimal high-probability estimates on the quantity
\begin{equation} \label{quantity in islc}
\scalarb{\b v}{\pb{G(z) - m(z) \umat} \b w}\,,
\end{equation}
where $\b v$ and $\b w$ are arbitrary deterministic vectors. Note that \eqref{old lsc} is a special case obtained from \eqref{quantity in islc} by setting $\b v = \b e_i$ and $\b w = \b e_j$, where $\b e_i$ denotes $i$-th standard basis vector of $\C^N$.

\subsection{Outline and sketch of proofs}
In Section \ref{sec: results}, we introduce basic definitions and state our results. In a first part, we state the isotropic semicircle law (Theorem \ref{thm: ILSC}) and some important corollaries, such as the isotropic delocalization estimate (Theorem \ref{theorem: delocalization}). The second part of Section \ref{sec: results} is devoted to the spectra of deformed Wigner matrices. Our main results are deviation estimates on the eigenvalue locations (Theorem \ref{theorem: rank-k lde}) and the distribution of the outliers (Theorem \ref{theorem: outlier distributions}). In subsequent remarks we discuss some special cases of interest, in particular making the link to the previous results of \cite{CDMF1, CDMF2, CDMF3, SoshPert}.

The remainder of this paper is devoted to proofs. As it turns out, the proof of the isotropic local semicircle law is considerably simpler if the third moments of the matrix entries of $H$ vanish. This case is dealt with in Section \ref{section: three moments}. The proof is based on the Green function comparison 
method and the local semicircle law of \cite{EYY3}.
In Section \ref{section: two moments}, we give the additional arguments needed to extend the isotropic local semicircle law to arbitrary matrix entries.
We remark that the Green function comparison method has been traditionally \cite{EYY2, EKYY2, KY1} used to obtain limiting distributions of smooth, bounded, observables that depend on the resolvent $G$. In this paper we use it in a novel setting: to obtain high-probability bounds on a fluctuating error.

In Section \ref{section: deloc and exterior} we use the isotropic semicircle law to obtain an improved estimate outside of the classical spectrum $[-2,2]$, and prove the isotropic delocalization result which yields optimal high-probability bounds on projections of the eigenvectors of $H$ onto arbitrary deterministic vectors.

Section \ref{section: eigenvalue locations} is devoted to the proof of deviation estimates for the eigenvalues of $H + A$. Our starting point for locating the eigenvalues is a simple identity from linear algebra (Lemma \ref{lemma: det identity}) already used in the works \cite{SoshPert, BGGM1, BGGM2, BGN}. Similar identities were also used in \cite{BS, BY1, BY2} for deformed covariance matrices. Using such identities, the study of the eigenvalue distribution of the deformed ensemble can be reduced to the study of the resolvent. In our case, this study of the resolvent is considerably more involved because we allow very general perturbations and also identify the distribution of non-outliers. In order to illustrate our method, we first consider the rank-one case in Theorem \ref{theorem: rank-one lde}. The general rank-$k$ case is based on a bootstrap argument -- in which the eigenvalues $\b d = (d_1, \dots, d_k)$ of $A$ are varied -- which may be summarized in the following three steps.
\begin{enumerate}
\item
For arbitrary $\b d$, we establish a ``permissible region'' $\Gamma(\b d) \subset \R$ whose complement cannot contain eigenvalues of $H + A$. The region $\Gamma(\b d)$ consists essentially of small neighbourhoods of the extremal eigenvalues of $H$ as well as of small neighbourhoods of the classical outlier locations $d_i + d_i^{-1}$ for $i$ satisfying $\abs{d_i} > 1$. 
\item
We fix $\b d$ to be \emph{independent} of $N$. In this simple case, we prove that each permissible neighbourhood of a classical outlier location $d_i + d_i^{-1}$ contains exactly one eigenvalue of $H + A$. Moreover, we prove that the non-outliers of $H + A$ stick to eigenvalues of $H$.
\item
In order to allow arbitrary $N$-dependent $\b d$'s, we construct a continuous path $(\b d(t))_{t \in [0,1]}$ that takes an $N$-independent initial configuration $\b d(0)$ to the desired $N$-dependent configuration $\b d \equiv \b d(1)$. Using (i), (ii), and the continuity of the eigenvalues of $H + A(t)$ as functions of $t$, we infer that the conclusions of (ii) remain valid for all $\b d(t)$ where $t \in [0,1]$, and in particular for $\b d(1)$.
(Here $A(t)$ denotes the perturbation with eigenvalues $\b d(t)$.)
\end{enumerate}

Finally, Section \ref{section: ditribution of outl} contains the proof  of Theorem \ref{theorem: outlier distributions}, the distribution of the outliers. The proof consists of four main steps.
\begin{enumerate}
\item
We reduce the problem of identifying the distribution of an outlier to that of analysing the distribution of random variables of the form $\scalar{\b v}{G(\theta) \b v}$, where $\theta \deq d + d^{-1}$ and $d$ is an eigenvalue of $A$ with associated eigenvector $\b v$. The argument is based on a precise control of the derivative of $G(z)$ and second-order perturbation theory.
\item
We consider the case where $H$ is Gaussian. Using the unitary invariance of the law of $H$, we prove that $\scalar{\b v}{G(\theta) \b v}$, when appropriately rescaled, converges to a normal random variable.
\end{enumerate}
The remainder of the proof consists in analysing the difference between the general Wigner case and the Gaussian case. Ultimately, we shall apply the Green function comparison method to expressions of the form $\scalar{\b v}{G(\theta) \b v}$ (Step (iv) below). However, this method is only applicable if $\norm{\b v}_\infty$ is sufficiently small (in fact, our result shows that the Green function comparison method must fail if $\norm{\b v}_\infty$ is not small). We therefore have to perform a two-step comparison.
\begin{enumerate}
\setcounter{enumi}{2}
\item
Let $H$ be the Wigner matrix we are interested in. We introduce a cutoff $\epsilon_N$ (equal to $\varphi^{-D}$ in the notation of Section \ref{section: almost Gaussian}). We define $\wh H$ as the Wigner matrix obtained from $H$ by replacing the $(i,j)$-th entry of $H$ with a Gaussian whenever $\abs{v_i} \leq \epsilon_N$ and $\abs{v_j} \leq \epsilon_N$. We choose $\epsilon_N$ large enough that most entries of $\wh H$ are Gaussian. We shall compare $H$ with a Gaussian matrix $V$ via the intermediate matrix $\wh H$. In this step, (iii), we compare $\wh H$ with $V$.

Our proof relies on a block expansion of $\wh H$, which expresses the distribution of the difference 
\begin{equation*}
\scalarb{\b v}{\p{\wh H - \theta}^{-1} \b v} - \scalarb{\b v}{\p{V - \theta}^{-1} \b v}
\end{equation*}
in terms of a sum of independent random variables ($\Gamma_1, \dots, \Gamma_6$ in the notation of Section \ref{section: almost Gaussian}) whose laws may be explicitly computed.
\item
In the final step, we use the Green function comparison method to analyse the difference
\begin{equation*}
\scalarb{\b v}{\p{H - \theta}^{-1} \b v} - \scalarb{\b v}{\p{\wh H - \theta}^{-1} \b v}\,.
\end{equation*}
By definition of $\wh H$, whenever the entry $(i,j)$ of $H$ differs from that of $\wh H$, we have $\abs{v_i} \leq \epsilon_N$ and $\abs{v_j} \leq \epsilon_N$. As a consequence, as it turns out, the Green function comparison method is applicable. Of special note in this comparison argument is a shift in the mean of the outlier (arising from the second term on the right-hand side of \eqref{comparison for T}), depending on the third moments of the entries of $H$.
\end{enumerate}

\subsection*{Acknowledgements}
We are grateful to Alex Bloemendal, Paul Bourgade, L\'aszl\'o Erd\H{o}s, and Horng-Tzer Yau for helpful comments.

\section{Results} \label{sec: results}

\subsection{The setup}
Let $H^\omega \equiv H = (h_{ij})$ be an $N \times N$ matrix; here $\omega$ denotes the running element in probability space, which we shall almost always drop from the notation. We assume that the upper-triangular entries $(h_{ij} \col i \leq j)$ are independent complex-valued random variables. The remaining entries of $H$ are given by imposing $H = H^*$. Here $H^*$ denotes the Hermitian conjugate of $H$. We assume that all entries are centred, $\E h_{ij} = 0$. In addition, we assume that one of the two following conditions holds.
\begin{description}
\item[{\rm (i) {\it Real symmetric Wigner matrix}:}] $h_{ij} \in \R$ for all $i,j$ and
\begin{equation*}
\E h_{ii}^2 \;=\; \frac{2}{N} \,, \qquad \E h_{ij}^2 \;=\; \frac{1}{N} \qquad (i \neq j)\,.
\end{equation*}
\item[{\rm (ii) {\it Complex Hermitian Wigner matrix}:}]
\begin{equation*}
\E h_{ii}^2 \;=\; \frac{1}{N} \,, \qquad \E \abs{h_{ij}}^2 \;=\; \frac{1}{N}\,, \qquad \E h_{ij}^2 \;=\; 0 \qquad (i \neq j)\,.
\end{equation*}
\end{description}
We use the abbreviation GOE/GUE to mean GOE if $H$ is a real symmetric Wigner matrix with Gaussian entries and GUE if $H$ is a complex Hermitian Wigner matrix with Gaussian entries.
We assume that the entries of $H$ have uniformly subexponential decay, i.e.\ that there exists a constant $\vartheta > 0$ such that
\begin{equation} \label{subexp for h}
\P\pb{\sqrt{N} \abs{h_{ij}} \geq x} \;\leq\; \vartheta^{-1} \exp(- x^\vartheta)
\end{equation}
for all $i,j$. Note that we do not assume the entries of $H$ to be identically distributed.

The following quantities will appear throughout this paper.  We choose a fixed but arbitrary constant $\Sigma \geq 3$. We define the logarithmic control parameter
\be
\varphi_N \;\equiv\; \varphi \;\deq\; (\log N)^{\log\log N}\,.
\ee
The parameter $\zeta$ will play the role of a fixed positive constant, which simultaneously dictates the power of $\varphi$ in large deviations estimates and characterizes the decay of probability of exceptional events, according to the following definition.
\begin{definition}[High probability events]
Let $\zeta> 0$.
We say that an $N$-dependent event $\Xi$ holds with \emph{$\zeta$-high probability} if there is some constant $C$ such that
\begin{equation} \label{high prob}
\P(\Xi^c) \;\leq\; N^C \exp(-\varphi^\zeta)
\end{equation}
for large enough $N$.
\end{definition}

Introduce the spectral parameter
\begin{equation*}
z \;=\; E + \ii \eta\,,
\end{equation*}
which will be used as the argument of Stieltjes transforms and resolvents. In the following we shall often use the notation $E = \re z$ and $\eta = \im z$ without further comment.
Let
\begin{equation*}
\varrho(\xi) \;\deq\; \frac{1}{2 \pi} \sqrt{[4 - \xi^2]_+} \qquad (\xi \in \R)
\end{equation*}
denote the density of the local semicircle law, and
\begin{equation} \label{definition of msc}
\msc(z) \;\deq\; \int \frac{\varrho(\xi)}{\xi - z} \, \dd \xi \qquad (z \notin [-2,2])
\end{equation}
its Stieltjes transform. To avoid confusion, we remark that the Stieltjes transform $m$ was denoted by $m_{sc}$ in the papers \cite{ESY1, ESY2, ESY3, ESY4,ESY5,ESY6, ESY7, ESYY, EYY1, EYY2, EYY3, EKYY1, EKYY2}, in which $m$ had a different meaning from \eqref{definition of msc}. It is well known that the Stieltjes transform $\msc$ satisfies the identity
\begin{equation} \label{identity for msc}
\msc(z) + \frac{1}{\msc(z)} + z \;=\; 0\,.
\end{equation}
For $\eta > 0$ we define the resolvent of $H$ through
\begin{equation*}
G(z) \;\deq\; (H - z)^{-1}\,.
\end{equation*}

We use the notation $\b v = (v_i)_{i = 1}^N \in \C^N$ for the components of a vector. We introduce the standard scalar product $\scalar{\b v}{\b w} \deq \sum_i \ol v_i w_i$, which induces the Euclidean norm $\norm{\b v} \deq \sqrt{\scalar{\b v}{\b v}}$. By definition, $\b v$ is normalized if $\norm{\b v} = 1$.

We denote by $C$ a generic positive large constant, whose value may change from one expression to the next. If this constant depends on some parameters $\alpha$, we indicate this by writing $C_\alpha$.
Finally, for two positive quantities $A_N$ and $B_N$ we use the notation $A_N \asymp B_N$ to mean $C^{-1} A_N \leq B_N \leq C A_N$ for some positive constant $C$.

\subsection{The isotropic local semicircle law}
For $\zeta > 0$ let
\begin{equation} \label{def D}
\b S(\zeta) \;\deq\; \hb{z \in \C \col \abs{E} \leq \Sigma \,,\, \varphi^\zeta N^{-1} \leq \eta \leq \Sigma}\,.
\end{equation}
For $z \in \b S(\zeta)$ define the control parameter
\begin{equation*}
\Psi(z) \;\deq\; \sqrt{\frac{\im m(z)}{N \eta}} + \frac{1}{N \eta}\,.
\end{equation*}
Our first main result is on the convergence of $G(z)$ to $m(z) \umat$.

\begin{theorem}[Isotropic local semicircle law] \label{thm: ILSC}
Fix $\zeta > 0$. Then there exists a constant $C_\zeta$ such that
\begin{align}\label{ILSC for three moments}
\absb{\scalar{\b v}{G(z) \b w} - \msc(z) \scalar{\b v}{\b w}} \;\leq\; \varphi^{C_\zeta} \Psi(z) \norm{\b v} \norm{\b w}
\end{align}
holds \hp{\zeta} for all deterministic $\b v, \b w \in \C^N$ under either of the two following conditions.
\begin{itemize}
\item[{\rm \bf A.}]
The spectral parameter $z \in \b S(C_\zeta)$ is arbitrary, and the third moments of the entries of $H$ vanish in the sense that
\begin{equation} \label{3rd moment vanishes}
\E h_{ij}^3 \;=\; \E h_{ij}^2 \ol h_{ij} \;=\; 0 \qquad (i,j = 1, \dots, N)\,.
\end{equation}
\item[{\rm \bf B.}]
The spectral parameter $z \in \b S(C_\zeta)$ satisfies
\begin{equation} \label{main assumption on Psi}
\Psi(z)^3 \;\leq\; \varphi^{-C_0} N^{-1/2}
\end{equation}
for some large enough constant $C_0$ depending on $\zeta$.
\end{itemize}
\end{theorem}

Away from the asymptotic spectrum $[-2,2]$, Theorem \ref{thm: ILSC} can be strengthened as follows.

\begin{theorem}[Isotropic local semicircle law outside of the spectrum] \label{theorem: strong estimate}
Fix $\zeta > 0$ and $\Sigma \geq 3$. Then there exist constants $C_1$ and $C_\zeta$ such that for any
\begin{equation*}
E \;\in\; \qb{-\Sigma, -2 - \varphi^{C_1} N^{-2/3}} \cup \qb{2 + \varphi^{C_1} N^{-2/3}, \Sigma}\,,
\end{equation*}
any $\eta \in (0, \Sigma]$, and any deterministic $\b v, \b w \in \C^N$ we have
\begin{equation} \label{strong estimate}
\absb{\scalar{\b v}{G(z) \b w} - \msc(z) \scalar{\b v }{\b w}} \;\leq\; \varphi^{C_\zeta} \sqrt{\frac{\im \msc(z)}{N \eta}} \, \norm{\b v} \norm{\b w}\,.
\end{equation}
\hp{\zeta}.
\end{theorem}

\begin{remark} \label{remark: lattice}
Using a simple lattice argument combined with the Lipschitz continuity of $z \mapsto G(z)$, one can easily strengthen the statement \eqref{ILSC for three moments} of Theorem \ref{thm: ILSC} to a simultaneous high probability statement for all $z$, as in \eqref{Gij estimate} below. For more details, see e.g.\ Corollary 3.19 in \cite{EKYY1}.

Similarly, mimicking the proof of Lemma \ref{lemma: derivative of Gvv} below, we find
\begin{equation} \label{sup of derivative of G}
\sup \hB{\abs{\partial_z \scalar{\b v}{G(z) \b w}} \col 2 + \varphi^{C_1} N^{-2/3} \leq \abs{E} \leq \Sigma \,,\, 0 < \abs{\eta} \leq \Sigma} \;\leq\; N
\end{equation}
\hp{\zeta}, from which we infer that the statement \eqref{strong estimate} of Theorem \ref{theorem: strong estimate} holds \hp{\zeta} simultaneously for all $z = E + \ii \eta$ satisfying the conditions in \eqref{sup of derivative of G}.
\end{remark}

For an $N \times N$ matrix $A$ we denote by $\lambda_1(A) \leq \lambda_2(A) \leq \cdots \leq \lambda_N(A)$ the nondecreasing sequence of eigenvalues of $A$. Moreover, we denote by $\sigma(A)$ the spectrum of $A$.
It is convenient to abbreviate the (random) eigenvalues of $H$ by
\begin{equation*}
\lambda_\alpha \;\deq\; \lambda_\alpha(H)\,.
\end{equation*}
Denote by $\b u^{(1)}, \b u^{(2)}, \dots, \b u^{(N)} \in \C^N$ the normalized eigenvectors of $H$ associated with the eigenvalues $\lambda_1 \leq \lambda_2 \leq \cdots \leq \lambda_N$. Our next result provides a bound on $\scalar{\b u^{(\alpha)}}{\b v}$ for arbitrary deterministic $\b v$.

\begin{theorem}[Isotropic delocalization] \label{theorem: delocalization}
Fix $\zeta > 0$. Then there is a constant $C_\zeta$ such that the following holds for any deterministic and normalized $\b v \in \C^N$.
\begin{enumerate}
\item
For any integers $a$ and $b$ satisfying  $1 \leq a < b \leq N/2$ and
\begin{equation} \label{main condition for delocalization}
b - a \;\geq\; 2 \varphi^{C_0} \pB{b^{1/3} N^{-1/6} + (ab)^{1/3} N^{-1/3}}
\end{equation}
we have
\begin{equation} \label{two moment delocalization}
\frac{1}{b - a} \sum_{\alpha = a}^{b} \abs{\scalar{\b u^{(\alpha)}}{\b v}}^2 \;\leq\; \varphi^{C_\zeta} N^{-1}
\end{equation}
\hp{\zeta}. Here $C_0$ is the constant from Theorem \ref{thm: ILSC}.
By symmetry, a similar result holds for the eigenvectors $\alpha \geq N/2$.
\item
If the third moments of the entries of $H$ vanish in the sense of \eqref{3rd moment vanishes}, then we have the stronger statement
\begin{equation} \label{three moment delocalization}
\sup_{\alpha} \abs{\scalar{\b u^{(\alpha)}}{\b v}}^2 \;\leq\; \varphi^{C_\zeta} N^{-1}
\end{equation}
\hp{\zeta}.
\end{enumerate}
\end{theorem}

\begin{remark}
Theorem \ref{theorem: delocalization} implies that the coefficients of the eigenvectors of $H$ are strongly oscillating. In order to see this, let $\alpha = 1, \dots, N$. If the third moments of the entries of $H$ do not vanish, we require that $\alpha \notin [\varphi^{-4C_0} N^{1/2}, N - \varphi^{-4C_0}N^{1/2}]$. Then choosing $\b v = N^{-1/2}(1, \dots, 1)$ and $\b v = \b e_i$ for $i = 1, \dots, N$ in Theorem \ref{theorem: delocalization} yields
\begin{equation} \label{sum of v alpha}
\absBB{\sum_{i = 1}^N u^{(\alpha)}_i} \;\leq\; \varphi^{C_\zeta}\,, \qquad \max_{1 \leq i \leq N} \abs{u^{(\alpha)}_i} \;\leq\; \varphi^{C_\zeta} N^{-1/2}
\end{equation}
\hp{\zeta}. The second inequality implies
\begin{equation*}
\sum_{i = 1}^N \abs{u^{(\alpha)}_i} \;\geq\; \varphi^{-C_\zeta} N^{1/2} \sum_{i = 1}^N \abs{u^{(\alpha)}_i}^2 \;=\; \varphi^{-C_\zeta} N^{1/2}
\end{equation*}
\hp{\zeta}. Compare this with the first inequality of \eqref{sum of v alpha}.

This behaviour is not surprising. In the GOE/GUE case, it is well known that each eigenvector $\b u^{(\alpha)}$ is uniformly distributed on the unit sphere, so that its entries asymptotically behave like i.i.d.\ Gaussians.
\end{remark}

\subsection{Finite-rank deformation of Wigner matrices}
Let $k \in \N$ be fixed, $V$ be a deterministic $N \times k$ matrix satisfying $V^* V = \umat$, and $d_1, \dots, d_k \in \R \setminus \{0\}$ be deterministic. We allow $d_1 \equiv d_1(N), \dots, d_k \equiv d_k(N)$ to depend on $N$. We also use the notation $V = [\b v^{(1)}, \dots, \b v^{(k)}]$, where $\b v^{(1)}, \dots, \b v^{(k)} \in \C^N$ are orthonormal.
Define the rank-$k$ perturbation
\begin{equation*}
V D V^* \;=\; \sum_{i = 1}^k d_i \b v^{(i)} (\b v^{(i)})^*\,, \qquad D \;=\; \diag(d_1, \dots, d_k)\,.
\end{equation*}
We shall study the spectrum of the deformed matrix
\begin{equation*}
\wt H \;\deq\; H + V D V^*\,.
\end{equation*}
We abbreviate the eigenvalues of $\wt H$ by
\begin{equation*}
\mu_\alpha \;\deq\; \lambda_\alpha(\wt H)\,.
\end{equation*}

In order to state our results, we order the eigenvalues of $D$, i.e.\ we assume that $d_1 \leq \dots \leq d_k$. Define the numbers
\begin{equation*}
k^\pm \;\deq\; \# \h{i \col \pm d_i > 1}\,.
\end{equation*}
As we shall see, $k^-$ is the number of outliers to the left of the bulk and $k^+$ the number of outliers to the right of the bulk.
We shall always assume that $k^-$ and $k^+$ are independent of $N$.

Let
\begin{equation} \label{set of outliers}
O \;\deq\; \hb{i \in \{1, \dots, k\} \col \abs{d_i} > 1} \;=\; \{1, \dots, k^-, k - k^+ + 1, \dots, k\}
\end{equation}
denote the $k^- + k^+$ indices associated with the outliers. 
For $i \in O$ abbreviate the associated eigenvalue index by
\begin{equation} \label{indices of outliers}
\alpha(i) \;\deq\;
\begin{cases}
N - k + i & \text{if } i \geq k - k^+ + 1
\\
i & \text{if } i \leq k^-\,.
\end{cases}
\end{equation}
Finally, for $d \in \R \setminus (-1,1)$ we define
\begin{equation} \label{def theta}
\theta(d) \;\deq\; d + \frac{1}{d}\,.
\end{equation}

\begin{theorem}[Locations of the deformed eigenvalues]\label{theorem: rank-k lde}
Fix $\zeta > 0$, $K > 0$, $k \in \N$,  and $0 < \fra b < 1/3$. Then there exist positive constants $C_2$ and $C_3$ such that the following holds.

Choose a sequence $\psi \equiv \psi_N$ satisfying $1 \leq \psi \leq N^{\fra b}$. Suppose that
\begin{equation} \label{d away from transition}
\abs{d_i} \;\leq\; \Sigma - 1 \,, \qquad \absb{\abs{d_i} - 1} \;\geq\; \varphi^{C_2} \psi N^{-1/3}
\end{equation}
for all $i = 1, \dots, k$. Then for $i \in O$ we have
\begin{equation} \label{main result: outliers}
\absb{\mu_{\alpha(i)} - \theta(d_i)} \;\leq\; \varphi^{C_3} N^{-1/2} (\abs{d_i} - 1)^{1/2}
\end{equation}
\hp{\zeta}.
Moreover,
\begin{subequations} \label{main result: bulk evs}
\begin{align}
\abs{\mu_\alpha - \lambda_{\alpha - k^-}} \;\leq\; \psi^{-1} N^{-2/3} \qquad \text{for} \qquad k^- + 1 \;\leq\; \alpha \;\leq\; \varphi^K\,,
\\
\abs{\mu_\alpha - \lambda_{\alpha + k^+}} \;\leq\; \psi^{-1} N^{-2/3} \qquad \text{for} \qquad N - \varphi^K \;\leq\; \alpha \;\leq\; N - k^+\,,
\end{align}
\end{subequations}
\hp{\zeta}.
\end{theorem}

\begin{remark}
In \cite{CDMF1}, Capitaine, Donati-Martin, and F\'eral proved that $\mu_{\alpha(i)} \to \theta(d_i)$ almost surely for all $i \in O$, under the assumptions that (i) $D$ does not depend on $N$ and (ii) the law of the entries of $H$ is symmetric and satisfies a Poincar\'e inequality.  Subsequently, the assumption (ii) was relaxed by Pizzo, Renfrew, and Soshnikov \cite{SoshPert}. In fact, in \cite{SoshPert} the authors proved, assuming (i), that the sequence $\sqrt{N} (\mu_{\alpha(i)} - \theta(d_i))$ is bounded in probability for all $i \in O$.

In \cite{BGGM1,BGGM2}, Benaych-Georges, Guionnet, and Ma\"ida  considered deformations of Wigner matrices by finite-rank random matrices whose eigenvalues are independent of $N$ and whose eigenvectors are either independent copies of a random vector with i.i.d.\ centred components satisfying a log-Sobolev inequality or are obtained by Gram-Schmidt orthonormalization of such independent copies. For these random perturbation models, they established eigenvalue sticking estimates similar to \eqref{main result: bulk evs}.
\end{remark}

\begin{remark}
Provided one is only interested in the locations of the outliers, i.e.\ \eqref{main result: outliers}, one can set $\psi = 1$ in Theorem \ref{theorem: rank-k lde}.
\end{remark}

We shall refer to the eigenvalues in \eqref{main result: outliers}, i.e.\ $\mu_1, \dots \mu_{k^{-}}, \mu_{N - k^+ + 1}, \dots, \mu_N$, as the \emph{outliers}, and to the eigenvalues in \eqref{main result: bulk evs}, i.e.\ $\mu_{k^- + 1}, \dots, \mu_{\varphi^K}, \mu_{N - \varphi^K}, \dots, \mu_{N - k^+}$ , as the \emph{extremal bulk eigenvalues}.

\begin{remark}
The phase transition associated with $d_i$ happens on the scale $d_i = 1 + a_i N^{-1/3}$ where $a_i$ is of order one. The condition \eqref{d away from transition} is optimal (up to powers of $\varphi$) in the sense that the power of $N$ in \eqref{d away from transition} cannot be reduced. Indeed, in \cite{BBP, Pec, BV1,BV2} it is established that, for rank-one\footnote{For simplicity of presentation, we consider rank-one deformations, although the results of \cite{BBP, Pec, BV1,BV2} hold for rank-$k$ deformations.} deformations of GOE/GUE with $d = 1 + a N^{-1/3}$ and $a$ of order one, $\mu_N$ fluctuates on the scale $N^{-2/3}$ and its distribution differs from that of $\lambda_N$. Hence in that case \eqref{main result: bulk evs} cannot hold for $\psi \gg 1$. See also Remark \ref{remark: unsticking} below for a more detailed discussion of the qualitative behaviour of eigenvalues of $\wt H$ as $d_i$ crosses a transition point.

Note that the location $\theta(d_i)$ of the outlier associated with $d_i = 1 + a_i N^{-1/3}$ satisfies $\theta(d_i) = 2 + N^{-2/3} a_i^2 + O(a_i^3 N^{-1})$. In comparison, the largest eigenvalue of $H$ fluctuates on a scale $N^{-2/3}$ around $2$.
\end{remark}
\begin{remark}
An immediate corollary of Theorem \ref{theorem: rank-k lde} is the universality of the extremal bulk eigenvalues of $\wt H$. In other words, under the assumption $\abs{\abs{d_i} - 1} \geq \varphi^{C_2 + 1} N^{-1/3}$ for all $i$, the statistics of the extremal bulk eigenvalues of $\wt H$ coincide with those of GOE/GUE.

Indeed, choosing $\psi = \varphi$ in Theorem \ref{theorem: rank-one lde} and invoking the edge universality for the Wigner matrix $H$ proved in Theorem 1.1 of \cite{KY1} (for similar results, see also \cite{EYY3,EKYY2}), we find for all $\ell \in \N$ and all bounded and continuous $f$ that
\begin{equation*}
\lim_{N \to \infty} \qbb{\E f\pB{N^{2/3}(\mu_{k^- + 1} + 2), \dots, N^{2/3} (\mu_{k^- + \ell} + 2)} - \E^G f \pb{N^{2/3}(\lambda_1 + 2), \dots, N^{2/3}(\lambda_\ell + 2)}} \;=\; 0\,,
\end{equation*}
where $\E^G$ denotes expectation with respect to the $N \times N$ GOE/GUE matrices. A similar result holds at the other end of the spectrum.
\end{remark}

\begin{remark}
Theorem \ref{theorem: rank-k lde} was formulated for deterministic perturbations. However, it extends trivially to the case where $V$ is random, independent of $H$, with arbitrary law satisfying $V^* V = \umat$.
\end{remark}

\begin{remark} \label{remark: unsticking}
The parameter $\psi$ describes how strongly the extremal bulk eigenvalues of $\wt H$ stick to extremal eigenvalues of $H$. If $d_i$ is within distance $C N^{-1/3}$ of a transition point $\pm 1$, one does not expect the eigenvalues of $\wt H$ to stick to the eigenvalues of $H$. For very weak sticking on the scale $N^{-2/3} \varphi^{-1}$, corresponding to $\psi = \varphi$, the eigenvalues $d_i$ have to satisfy $\absb{\abs{d_i} - 1} \geq \varphi^{C_2 + 1} N^{-1/3}$. In particular, we may allow outliers at a distance $\varphi^{2 C_2 + 2} N^{-2/3}$ from the spectral edge.

On the other hand, in order to obtain strong sticking on the scale $N^{-1 + \epsilon}$, corresponding to $\psi = N^{1/3 - \epsilon}$, the eigenvalues $d_i$ have to satisfy $\absb{\abs{d_i} - 1} \geq \varphi^{C_2} N^{-\epsilon}$. Now the outliers have to lie at a distance of at least $N^{2 C_2 - 2 \epsilon}$ from the spectral edge.

Thus, Theorem \ref{theorem: rank-k lde} gives a clear picture of what happens to the extremal bulk eigenvalues as $d_i$ passes a transition point $\pm1$. For definiteness, consider the case where $d_i$ is varied from $1 - c$ to $1 + c$ for some small $c > 0$, and all other eigenvalues of $D$ are kept constant. 
Consider an extremal bulk eigenvalue near $+2$, say $\mu_\alpha$. By Theorem \ref{theorem: rank-k lde}, for $d_i \leq 1 - \varphi^{C_2 + 1} N^{-1/3}$, $\mu_\alpha$ sticks to $\lambda_\beta$ where $\beta \deq \alpha + k^+$. As $d_i$ approaches $1$, the eigenvalue $\mu_\alpha$ progressively detaches itself from $\lambda_\beta$. Theorem \ref{theorem: rank-k lde} allows one to follow this behaviour down to $\abs{d_i - 1} = \varphi^{C_2 + 1} N^{-1/3}$. Below this scale, as $d_i$ passes $1$, the eigenvalue $\mu_\alpha$ ``jumps'' from from the vicinity of $\lambda_\beta$ to the vicinity of $\lambda_{\beta + 1}$. This jump happens in the range $d_i \in [1 - \varphi^{C_2 + 1} N^{-1/3}, 1+\varphi^{C_2 + 1} N^{-1/3}]$. After the jump, i.e.\ for $d_i \geq 1 + \varphi^{C_2 + 1} N^{-1/3}$, the eigenvalue $\mu_\alpha$ sticks to $\lambda_{\beta + 1}$ instead of $\lambda_\beta$, provided that $\beta < N$. If $\beta = N$, then $\mu_\alpha$ escapes from the bulk spectrum and becomes an outlier. This jump happens simultaneously for all extremal bulk eigenvalues near $+2$, and is accompanied by the creation of an outlier. This may be expressed as $(k^0, k^+) \mapsto (k^0 - 1, k^+ + 1)$. Meanwhile, the extremal bulk eigenvalues on the other side of the spectrum, i.e.\ near $-2$, remain unaffected by the transition, and continue sticking to the same eigenvalues of $H$ they stuck to before the transition.
\end{remark}

Next, we identify the distribution of the outliers. We introduce the customary symmetry index $\beta$, by definition equal to $1$ if $H$ is real symmetric and $2$ if $H$ is complex Hermitian. In order to state our result, we define the moment matrices $M^{(3)} = (M^{(3)}_{ij})$ and $M^{(4)} = (M^{(4)}_{ij})$ of $H$ through
\begin{equation*}
M^{(3)}_{ij} \;\deq\; N^{3/2} \E \pb{\abs{h_{ij}}^2 h_{ij}}\,, \qquad M^{(4)}_{ij} \;\deq\; N^2 \E \abs{h_{ij}}^4\,.
\end{equation*}
By definition of $H$, the matrices $M^{(3)}$ and $M^{(4)}$ are Hermitian. Moreover, by \eqref{subexp for h} they have uniformly bounded entries. For $\b v = (v_i) \in \C^N$ define
\begin{align}
Q(\b v) &\;\deq\; \frac{1}{2 \sqrt{N}} \sum_{i,j} \ol v_i M^{(3)}_{ij} \pb{\abs{v_i}^2 + \abs{v_j}^2} v_j\,,
\notag \\
R(\b v) &\;\deq\; \frac{1}{N} \sum_{i,j} \pb{M^{(4)}_{ij} - 4 + \beta} \abs{v_j}^4\,,
\notag \\ \label{definition of QRS}
S(\b v) &\;\deq\; \frac{1}{N} \sum_{i,j} \ol v_i M^{(3)}_{ij}  v_j\,.
\end{align}
The functions $Q$, $R$, and $S$ are bounded on the unit ball in $\C^N$, uniformly in $N$.

\begin{theorem}[Distribution of the outliers] \label{theorem: outlier distributions}
There is a constant $C_2$ such that the following holds.
Suppose that
\begin{equation} \label{d assumptions for distribution}
\abs{d_i} \;\leq\; \Sigma - 1 \,, \qquad \absb{\abs{d_i} - 1} \;\geq\; \varphi^{C_2} N^{-1/3}
\end{equation}
for all $i = 1, \dots, k$. Suppose moreover that for all $i \in O$ we have
\begin{equation} \label{non-degeneracy condition}
\min_{j \neq i} \abs{d_i - d_j} \;\geq\; \varphi^{C_2} N^{-1/2} (\abs{d_i} - 1)^{-1/2}\,.
\end{equation}
For $i \in O$ define the random variable
\begin{equation*}
\Pi_i \;\deq\; (\abs{d_i} + 1) (\abs{d_i} - 1)^{1/2} \pBB{\frac{N^{1/2} \scalar{\b v^{(i)}}{H \b v^{(i)}}} {d_i^2} + \frac{S(\b v^{(i)})} {d_i^4}}
\end{equation*}
and $\Upsilon_i$, a random variable independent of $\Pi_i$ with law
\begin{equation*}
\Upsilon_i \;\eqdist\; \cal N \pBB{0 \,,\, \frac{2 (\abs{d_i}+1)}{\beta d_i^4} + (\abs{d_i} +1)^2 (\abs{d_i} - 1) \pbb{\frac{4 Q(\b v^{(i)})}{d_i^5}
+ \frac{R(\b v^{(i)})}{d_i^6}}}\,.
\end{equation*}
Then we have, for all $i \in O$ and all bounded and continuous $f$,
\begin{equation} \label{distribution of outliers}
\lim_{N \to \infty} \qbb{\E f \pB{N^{1/2} (\abs{d_i} - 1)^{-1/2} \pb{\mu_{\alpha(i)} - \theta(d_i)}} - \E f (\Pi_i + \Upsilon_i)} \;=\; 0\,.
\end{equation}
\end{theorem}

Note that, by a standard approximation argument, \eqref{distribution of outliers} also holds for $f(x) = \ind{x \leq a}$ where $a \in \R$; hence the convergence \eqref{distribution of outliers} may also be stated in terms of distribution functions. 

\begin{remark}
In \cite{CDMF2}, Capitaine, Donati-Martin, and F\'eral identified the law of the outliers of deformed Wigner matrices subject to the following conditions: (i) $D$ is independent of $N$ but may have degenerate eigenvalues; (ii) the law of the matrix entries of $H$ is symmetric and satisfies a Poincar\'e inequality; (iii) the eigenvectors of the deformation belong to one of two classes, corresponding roughly to either partially delocalized eigenvectors or strongly localized eigenvectors.
Subsequently, the assumption (ii) was relaxed by Pizzo, Renfrew, and Soshnikov in \cite{SoshPert}. (But assumption (iii) imposes that $S(\b v^{(i)}) = Q(\b v^{(i)}) = 0$ still holds for the results of \cite{SoshPert}.)
\end{remark}

\begin{remark}
The condition \eqref{non-degeneracy condition} has the following interpretation.
Let $i \in O$ and assume for definiteness that $d_i > 1$. If $j$ is not associated with an outlier on the right-hand side of the bulk, i.e.\ if $d_j < 1$, then $d_i - d_j$ is bounded from below by the right-hand side of \eqref{non-degeneracy condition}, as follows from \eqref{d assumptions for distribution}. Hence the condition \eqref{non-degeneracy condition} is only needed to ensure that the outliers are not to close too each other; in fact, this condition is optimal (up to the factor $\varphi^{C_2}$) in guaranteeing that the distributions of the outliers have essentially no overlap. Indeed, by Theorem \ref{theorem: rank-k lde} we know that $\mu_{\alpha(i)}$ lies \hp{\zeta} in an interval of length $2 \varphi^{C_3} N^{-1/2} (d_i - 1)^{1/2}$ centred around $\theta(d_i)$. Moreover, differentiating \eqref{def theta} yields
\begin{equation*}
\theta(d_j) - \theta(d_i) \;\asymp\; (d_i - 1) (d_j - d_i)\,.
\end{equation*}
Imposing the condition $\abs{\theta(d_j) - \theta(d_i)} \geq \varphi^{C_3} N^{-1/2} (d_i - 1)^{1/2}$ leads to \eqref{non-degeneracy condition} (with $C_2$ increased if necessary so that $C_2 \geq C_3$). In fact, in \cite{BBP, Pec, SoshPert} it was proved (for $D$ independent of $N$) that the distribution associated with degenerate outliers is not Gaussian.
\end{remark}

The following remarks discuss some special cases of interest. In order to simplify notations, we set $k = 1$ and write $d \equiv d_1$, $\b v \equiv \b v^{(1)}$, $\Pi \equiv \Pi_1$, and $\Upsilon \equiv \Upsilon_1$.

\begin{remark}
In the GOE/GUE case, we have $M^{(3)} = 0$ and $M^{(4)}_{ij} = (4 - \beta) + \delta_{ij} (17 - 8 \beta)$. Thus we get that $Q(\b v) = S(\b v) = 0$ and $R(\b v) = O(N^{-1})$. Since $N^{1/2} \scalar{\b v}{H \b v}$ is a centred Gaussian with variance $2 \beta^{-1}$, we therefore find that $\Pi + \Upsilon$ has asymptotically\footnote{See Section \ref{section: GOE/GUE} for precise definitions and more details.} the distribution of a centred Gaussian with variance
\begin{equation*}
\frac{2 (\abs{d} + 1)^2 (\abs{d} - 1)}{\beta d^4} + \frac{2 (\abs{d} + 1)}{\beta d^4} \;=\; \frac{2 (\abs{d} + 1)}{\beta d^2}\,.
\end{equation*}
\end{remark}

\begin{remark}
If $\varphi^{C_2} N^{-1/3} \leq \absb{\abs{d} - 1} = o(1)$ then $\Pi + \Upsilon$ converges weakly to a centred Gaussian with variance $4 \beta^{-1}$. As an outlier approaches the bulk spectrum, the dependence of its distribution on the details of $H$ and $\b v$ is washed out. Therefore, unlike outliers located at a distance of order one from the bulk spectrum, outliers close to $\pm 2$ exhibit universality. Moreover, as an outlier approaches the bulk, its variance shrinks from $N^{-1}$ (for $d - 1 \asymp 1$) to $N^{-4/3}$ (for $d - 1 \asymp N^{-1/3}$).
\end{remark}

\begin{remark}
If $\max_i \abs{v_i} \to 0$ as $N \to \infty$, we find that $Q(\b v) \to 0$ and $R(\b v) \to 0$ as $N \to \infty$. Moreover, the Central Limit Theorem implies in this case that $N^{1/2} \scalar{\b v}{H \b v}$ converges in distribution to a centred Gaussian with variance $2 \beta^{-1}$. Therefore $\Pi + \Upsilon$ has asymptotically the distribution of
\begin{equation*}
\cal N \pbb{\frac{(\abs{d} + 1) (\abs{d} - 1)^{1/2} S(\b v)}{d^4}\,,\, \frac{2 (\abs{d} + 1)}{\beta d^2}}\,.
\end{equation*}
Thus, the only difference to the GOE/GUE case is a shift caused by the nonvanishing third moments of $H$. For example, if $M^{(3)}_{ij} = m^{(3)} \in \R$ is independent of $i$ and $j$, and $\b v = N^{-1/2}(1, \dots, 1)$, we find $S(\b v) = m^{(3)} + O(N^{-1})$.
\end{remark}

\begin{remark}
Typically, $R(\b v)$ is nonzero if $\b v$ has entries which do not converge to zero. An example for which $Q(\b v)$ is nonzero is $M^{(3)}_{ij} = m^{(3)} \in \R$ independent of $N$ and $\b v = (2^{-1/2}, (2N - 2)^{-1/2}, \dots, (2N - 2)^{-1/2})$, in which case we have $Q(\b v) = 2^{-3/2} m^{(3)} + O(N^{-1/2})$.
\end{remark}

\begin{remark}
Consider now the case where $\max_i \abs{v_i}$ does not tend to zero as $N \to \infty$. For definiteness, let $\b v = (\b u, \b w)$, where the dimension of $\b u$ is constant and $\max_i \abs{w_i} \to 0$ as $N \to \infty$. By the Central Limit Theorem and a short variance calculation, $N^{1/2} \scalar{\b v}{H \b v}$ has asymptotically the same distribution as $N^{1/2} \scalar{\b u}{H \b u} + 2 \beta^{-1}(1 - \norm{\b u}^2)(1 + 2 \norm{\b u}^2) Z$, where $Z$ is a standard normal random variable independent of $H$.

Let us take for example $\b v = (1, 0, \dots, 0)$. Then $\Pi + \Upsilon$ has asymptotically the same distribution as $\Pi' + \Upsilon'$, where
\begin{equation*}
\Pi' \;\deq\; (\abs{d} + 1) (\abs{d} - 1)^{1/2} d^{-2} N^{1/2} h_{11}\,,
\end{equation*}
and $\Upsilon'$ is a centred Gaussian, independent of $\Pi'$, with variance
\begin{equation*}
\frac{2 (\abs{d} + 1)}{\beta d^4} + \frac{(\abs{d} + 1)^2 (\abs{d} - 1)}{N d^6} \sum_i \pb{N^2 \E \abs{h_{1i}}^4  - 4 + \beta}\,.
\end{equation*}
\end{remark}

\section{Proof of Theorem \ref{thm: ILSC}, Case {\bf A}} \label{section: three moments}

In this section we prove Theorem \ref{thm: ILSC} in the case {\bf A}, i.e.\  where the first three moments of the entries of $H$ coincide with those of GOE/GUE.

We start by introducing the following notations we shall use throughout the rest of the paper. For an $N \times N$ matrix $A$ and $\b v, \b w \in \C^N$ we abbreviate
\begin{equation*}
A_{\b v \b w} \;\deq\; \scalar{\b v}{A \b w}\,.
\end{equation*}
We also write
\begin{equation*}
A_{\b v \b e_i} \;\equiv\; A_{\b v i}\,,
\qquad
A_{\b e_i \b v} \;\equiv\; A_{i \b v}\,,
\qquad
A_{\b e_i \b e_j} \;\equiv\; A_{i j}\,,
\end{equation*}
where $\b e_i \in \C^N$ denotes the $i$-th standard basis vector.

For definiteness, we consider the case where $H$ is a complex Hermitian Wigner matrix; the proof for real symmetric Wigner matrices is the same. By Markov's inequality, in order to prove Theorem \ref{thm: ILSC} it suffices to prove the following result.

\begin{proposition} \label{prop: three moment case}
Assume \eqref{3rd moment vanishes} and let $\zeta > 0$ be fixed. Then there exists a constant $C_\zeta$ such that, for all $n \leq \varphi^{\zeta}$, all deterministic $\b v, \b w \in \C^N$, and all $z \in \b S(C_\zeta)$,
\begin{equation} \label{three moment gen claim}
\E \absb{G_{\b v \b w}(z) - \scalar{\b v}{\b w} \msc(z)}^n \;\leq\; \pb{\varphi^{C_\zeta} \Psi(z) \norm{\b v} \norm{\b w}}^n\,.
\end{equation}
\end{proposition}

The rest of this section is devoted to the proof of Proposition \ref{prop: three moment case}.

\subsection{Preliminaries}
We start with a few basic tools. For $E \in \R$ define
\begin{equation} \label{def kappa}
\kappa_E \;\deq\; \absb{\abs{E} - 2}\,,
\end{equation}
the distance from $E$ to the spectral edges $\pm 2$. In the following we use the notations
\begin{equation*}
z \;=\; E + \ii \eta \,, \qquad \kappa \;\equiv\; \kappa_E
\end{equation*}
without further comment. The following lemma collects some useful properties of $\msc$, the Stieltjes transform of the semicircle law.
\begin{lemma} \label{lemma: msc}
For $\abs{z} \leq 2 \Sigma$ we have
\begin{equation} \label{bounds on msc}
\abs{\msc(z)} \;\asymp\; 1 \,, \qquad \abs{1 - \msc(z)^2} \;\asymp\; \sqrt{\kappa + \eta}\,.
\end{equation}
Moreover,
\begin{equation*}
\im \msc(z) \;\asymp\;
\begin{cases}
\sqrt{\kappa + \eta} & \text{if $\abs{E} \leq 2$}
\\
\frac{\eta}{\sqrt{\kappa + \eta}} & \text{if $\abs{E} \geq 2$}\,.
\end{cases}
\end{equation*}
(Here the implicit constants depend on $\Sigma$.)
\end{lemma}
\begin{proof}
The proof is an elementary calculation; see Lemma 4.2 in \cite{EYY2}.
\end{proof}

In addition to $\Psi$, we shall make use of a larger control parameter $\Phi$, defined as
\be \label{definition of Phi and Psi}
\Phi(z) \;\deq\; \im \msc(z) + \frac{1}{N\eta}\,,
\qquad \Psi(z) \;=\; \sqrt{\frac{\im \msc(z)}{N\eta}}+\frac{1}{N\eta} \;\asymp\; \sqrt{\frac{\Phi(z)}{N\eta}}\,.
\ee
From Lemma \ref{lemma: msc} we find, for any $z$ satisfying $\abs{z} \leq 2 \Sigma$,
\be\label{1.5}
N^{-1/2}
\;\lesssim\; \sqrt{\frac{\im \msc(z)}{N \eta}}
\;\lesssim\; \Psi(z) \;\lesssim\; \Phi(z)\,,
\ee
where $A_N \lesssim B_N$ means $A_N \leq C B_N$ for some constant $C$.

We shall often need to consider minors of $H$, which are the content of the following definition.

\begin{definition}[Minors]
For $\bb T \subset \{1, \dots, N\}$ we define $H^{(\bb T)}$ by
\begin{equation*}
(H^{(\bb T)})_{ij} \;\deq\; \ind{i \notin \bb T} \ind{j \notin \bb T} h_{ij}\,.
\end{equation*}
Moreover, we define the resolvent of $H^{(\bb T)}$ through
\begin{equation*}
G^{(\bb T)}_{ij}(z) \;\deq\; \ind{i \notin \bb T} \ind{j \notin \bb T} (H^{(\bb T)} - z)^{-1}_{ij}\,.
\end{equation*}
We also set
\begin{equation*}
\sum_i^{(\bb T)} \;\deq\; \sum_{i \col i \notin \bb T}\,.
\end{equation*}
When $\bb T = \{a\}$, we abbreviate $(\{a\})$ by $(a)$ in the above definitions; similarly, we write $(ab)$ instead of $(\{a,b\})$.
\end{definition}

We shall also need the following resolvent identities, proved in Lemma 4.2 of \cite{EYY1} and Lemma 6.10 of \cite{EKYY2}.
\begin{lemma}[Resolvent identities]
For any $i,j,k$ we have
\begin{equation} \label{Gij Gijk}
G_{ij} \;=\; G_{ij}^{(k)} + \frac{G_{ik} G_{kj}}{G_{kk}}\,.
\end{equation}
Moreover, for $i \neq j$ we have
\begin{equation} \label{sq root formula}
G_{ij} \;=\; - G_{ii} \sum_{k}^{(i)} h_{ik} G_{kj}^{(i)} \;=\; - G_{jj} \sum_k^{(j)} G_{ik}^{(j)} h_{kj}\,.
\end{equation}
These identities also hold for minors $H^{(\bb T)}$.
\end{lemma}
It is an immediate consequence of \eqref{Gij Gijk} that
\begin{equation} \label{Gvw minor}
G_{\b v \b w} \;=\; G_{\b v \b w}^{(k)} + \frac{G_{\b v k} G_{k \b w}}{G_{kk}}\,.
\end{equation}
Moreover, we introduce the notations
\begin{equation} \label{def G prime}
\cal G_{\b v i} \;\deq\; - \sum_{k}^{(i)} G_{\b v k}^{(i)} h_{k i}\,, \qquad
\cal G_{i \b v} \;\deq\; - \sum_{k}^{(i)} h_{i k} G_{k \b v}^{(i)}\,,
\end{equation}
so that
\begin{equation} \label{G G prime}
G_{\b v i} \;=\; G_{ii} \pb{\ol v_i + \cal G_{\b v i}}\,,
\qquad
G_{i \b v} \;=\; G_{ii} \pb{v_i + \cal G_{i \b v}}
\end{equation}
by \eqref{sq root formula}.

Next, we record some basic large deviations estimates.
\begin{lemma}[Large deviations estimates] \label{lemma: LDE}
Let $a_1, \dots, a_N, b_1, \dots, b_M$ be independent random variables with zero mean and unit variance. Assume that there is a constant $\vartheta > 0$ such that
\begin{align}
\P(\abs{a_i} \geq x) &\;\leq\; \vartheta^{-1} \exp(-x^\vartheta) \quad (i = 1, \dots, N)\,,
\notag  \\ \label{subexponential decay for LDE}
\P(\abs{b_i} \geq x) &\;\leq\; \vartheta^{-1} \exp(-x^\vartheta) \quad (i = 1, \dots, M)\,.
\end{align}
Then there exists a constant $\rho \equiv \rho(\vartheta) > 1$ such that, for any $\zeta > 0$ and any deterministic complex numbers $A_i$ and $B_{ij}$, we have \hp{\zeta}
\begin{align} \label{LDE}
\absBB{\sum_{i = 1}^N A_i a_i} &\;\leq\; \varphi^{\rho \zeta} \pBB{\sum_{i = 1}^N \abs{A_i}^2}^{1/2}\,,
\\ \label{diag LDE}
\absBB{\sum_i A_i \abs{a_i}^2 - \sum_i A_i} &\;\leq\;\varphi^{\rho \zeta} \pbb{\sum_i \abs{A_i}^2}^{1/2}\,,
\\ \label{offdiag LDE}
\absBB{\sum_{i \neq j} \ol a_i B_{ij} a_j} &\;\leq\; \varphi^{\rho \zeta} \pbb{\sum_{i \neq j} \abs{B_{ij}}^2}^{1/2}\,,
\\ \label{two-set LDE}
\absBB{\sum_{i,j} a_i B_{ij} b_j} &\;\leq\; \varphi^{\rho \zeta} \pbb{\sum_{i,j} \abs{B_{ij}}^2}^{1/2}\,.
\end{align}
\end{lemma}
\begin{proof}
The estimates \eqref{LDE} -- \eqref{offdiag LDE} we proved in Appendix B of \cite{EYY1}. The estimate \eqref{two-set LDE} follows easily from \eqref{LDE} in two steps. Defining $A_i \deq \sum_j B_{ij} b_j$, \eqref{LDE} yields $\abs{A_i} \leq \varphi^{\rho \zeta} \pb{\sum_j \abs{B_{ij}}^2}^{1/2}$ \hp{\zeta}. Since the families $\{A_i\}$ and $\{a_i\}$ are independent, \eqref{two-set LDE} follows by using \eqref{LDE} again.
\end{proof}

Finally, we quote the following results which are proved in Theorems 2.1 and 2.2 of \cite{EYY3}. (Recall that we use the notation $\msc$ for the quantity denoted by $m_{sc}$ in \cite{EYY3}.)

\begin{theorem}[Local semicircle law] \label{LSCTHM}
Fix $\zeta>0$. Then there exists a constant $C_\zeta$ such that the event
 \begin{equation}\label{Gij estimate}
\bigcap_{z \in \b S(C_\zeta)} \hbb{\max_{1 \leq i,j \leq N} \absb{G_{ij}(z) - \delta_{ij} \msc(z)} \leq \varphi^{C_\zeta} \Psi(z)}
\end{equation}
holds \hp{\zeta}.
\end{theorem}

Denote by $\gamma_1 \leq \gamma_2 \leq \cdots \leq \gamma_N$ the classical locations of the eigenvalues of $H$, defined through
\begin{equation} \label{classical location}
N \int_{-\infty}^{\gamma_\alpha} \varrho(x) \, \dd x \;=\; \alpha \qquad (1 \leq \alpha \leq N)\,.
\end{equation}

\begin{theorem}[Rigidity of eigenvalues] \label{theorem: rigidity}
Fix $\zeta > 0$. Then there exists a constant $C_\zeta$ such that
\begin{equation*}
\abs{\lambda_\alpha - \gamma_\alpha} \;\leq\; \varphi^{C_\zeta} \pb{\min \h{\alpha, N + 1 - \alpha}}^{-1/3} N^{-2/3}
\end{equation*}
for all $\alpha = 1, \dots, N$ \hp{\zeta}.
\end{theorem}

\subsection{Estimate of $G_{\b v i}$}
After these preparations, we may prove the key tool behind the proof of Proposition \ref{prop: three moment case}. It will be used as input in the Green function comparison method, throughout Sections \ref{section 3.3}, \ref{section 3.4}, and \ref{section: two moments}. Let us sketch its importance in the Green function comparison method. Anticipating the notation from the proof of Lemma \ref{lemma: three moment bound}, we shall have to estimate quantities of the form
\begin{equation*}
(S - R)_{\b v \b v} \;=\; \pb{-N^{-1/2} RVR + N^{-1} RVRVR + \cdots}_{\b v \b v}\,,
\end{equation*}
where the right-hand side is a resolvent expansion of the left-hand side. The first matrix product on the right-hand side may be written as
\begin{equation*}
(RVR)_{\b v \b v} \;=\; R_{\b v a} V_{ab} R_{b \b v} + R_{\b v b} V_{ba} R_{a \b v}
\end{equation*}
(again anticipating the notation from the proof of Lemma \ref{lemma: three moment bound}). Lemma \ref{lemma: Gvi Gvv} will be used to estimate the resolvent entries of the form $R_{\b v a}$ in such error estimates. These resolvent entries arise whenever the Green function comparison method is applied to the component $(\cdot)_{\b v \b v}$ of a resolvent.

\begin{lemma} \label{lemma: Gvi Gvv}
For any $\zeta > 0$ there exists a constant $C_\zeta$ such that
\begin{equation}
\abs{\cal G_{\b v i}(z)} + \abs{\cal G_{i \b v}(z)} + \abs{G_{\b v i}(z)} + \abs{G_{i \b v}(z)} \;\leq\; \varphi^{C_\zeta} \sqrt{\frac{\im G_{\b v \b v}(z)}{N \eta}} + C \abs{v_i}
\end{equation}
holds \hp{\zeta} for all $z \in \b S(C_\zeta)$.
\end{lemma}

\begin{proof}
Since the families $(h_{ki})_k$ and $(G^{(i)}_{\b v k})_k$ are independent, \eqref{def G prime}, \eqref{LDE}, and \eqref{subexp for h} yield
\begin{equation*}
\abs{\cal G_{\b v i}} \;\leq\; \varphi^{C_\zeta} \pBB{\frac{1}{N}\sum_{k}^{(i)} \absb{G_{\b v k}^{(i)}}^2}^{1/2}
\end{equation*}
\hp{\zeta} for some constant $C_\zeta$. By spectral decomposition one easily finds that
\begin{equation*}
\frac{1}{N} \sum_{k}^{(i)} \absb{G_{\b v k}^{(i)}}^2 = \frac{1}{N \eta} \im G_{\b v \b v}^{(i)}\,.
\end{equation*}
From \eqref{bounds on msc} and \eqref{Gij estimate} we find that
\begin{equation} \label{Gii bound}
\abs{G_{ii}} \;\leq\; C
\end{equation}
\hp{\zeta} provided that $\eta > \varphi^{C_\zeta}$ for some large enough $C_\zeta$.
Setting
\begin{equation*}
X \;\deq\; \abs{\cal G_{\b v i}} + \abs{\cal G_{i \b v}}\,,
\end{equation*}
we therefore conclude, using first \eqref{Gvw minor} and then \eqref{G G prime}, that
\begin{equation*}
X \;\leq\; \varphi^{C_\zeta} \pBB{\frac{\im G_{\b v \b v} + \abs{G_{ii}} \absb{G_{\b v i} / G_{ii}} \absb{G_{i \b v} / G_{ii}}}{N \eta}}^{1/2}
\;\leq\;
\varphi^{C_\zeta} \sqrt{\frac{\im G_{\b v \b v}}{N \eta}} + \varphi^{C_\zeta} \frac{X}{\sqrt{N \eta}} + \varphi^{C_\zeta} \frac{\abs{v_i}}{\sqrt{N \eta}}
\end{equation*}
\hp{\zeta}.
Thus we find for $\eta \geq \varphi^{2 C_\zeta} N^{-1}$
\begin{equation*}
X \;\leq\; \varphi^{C_\zeta} \sqrt{\frac{\im G_{\b v \b v}}{N \eta}} + \abs{v_i}
\end{equation*}
\hp{\zeta}, and the claim for $\abs{\cal G_{\b v i}} + \abs{\cal G_{i \b v}}$ follows. The claim for $\abs{G_{\b v i}} + \abs{G_{i \b v}}$ follows using \eqref{G G prime} and \eqref{Gii bound}.
\end{proof}

\subsection{Estimate of $\im G_{\b v \b v}$} \label{section 3.3}
The first step in the proof of Proposition \ref{prop: three moment case} is the following estimate of $\im G_{\b v \b v}$. Note that $\im G_{\b v \b v}$ is a nonnegative quantity, as may be easily seen by spectral decomposition of $G$.

\begin{lemma} \label{lemma: three moment bound}
Let $\zeta > 0$ be fixed. Then there exists a constant $C_\zeta$ such that, for all $n \leq \varphi^{\zeta}$, all deterministic and normalized $\b v \in \C^N$, and all $z \in \b S(C_\zeta)$, we have
\begin{equation} \label{rough bound}
\E \pb{\im G_{\b v \b v}(z)}^n \;\leq\; \pb{\varphi^{C_\zeta} \Phi(z)}^n\,.
\end{equation}
\end{lemma}
\begin{proof}
We shall prove \eqref{rough bound} using Green function comparison to GOE/GUE. First we claim that \eqref{rough bound} holds if $H$ is a GOE/GUE matrix. Indeed, in that case, using unitary invariance, \eqref{1.5}, and \eqref{Gij estimate}, we find for $z \in \b S(C_\zeta)$ that
\begin{equation*}
\E \pb{\im G_{\b v \b v}(z)}^n \;=\; \E \pb{\im G_{11}(z)}^n \;\leq\; \pb{\varphi^{C_\zeta} \Phi(z)}^n + N^n N^C \exp(-\varphi^{2 \zeta})\,,
\end{equation*}
where in the last inequality we used the rough bound $\abs{G_{11}(z)} \leq \eta^{-1} \leq N$.
Thus \eqref{rough bound} for GOE/GUE follows from \eqref{1.5} and the estimate
\begin{equation*}
N^{C n} \exp(-\varphi^{2 \zeta}) \;\leq\; C\,,
\end{equation*}
valid for $n \leq \varphi^\zeta$.

From now on we work on the product space generated by the Wigner matrix $H = (N^{-1/2} W_{ij})_{i,j}$ and the GOE/GUE matrix $(N^{-1/2} V_{ij})_{i,j}$. We fix a bijective ordering map on the index set of the independent matrix elements,
\begin{equation} \label{ordering map}
\phi \col \{(i, j) \col 1 \le i\le  j \le N \} \;\to\; \Big\{1, \ldots, \gamma_{\rm max}\Big\} \where \gamma_{\rm max} 
\;\deq\; \frac{N(N+1)}{2}\,,
\end{equation}
and denote by $H_\gamma = (h^{\gamma}_{ij})$, $\gamma = 0, \dots, \gamma_{\rm max}$, the Wigner matrix whose upper-triangular entries are defined by
\begin{equation*}
h_{ij}^\gamma \;\deq\;
\begin{cases}
N^{-1/2} W_{ij} & \text{if } \phi(i,j)\le \gamma
\\
N^{-1/2} V_{ij} & \text{otherwise}\,.
\end{cases}
\end{equation*}
In particular, $H_0$ is a GOE/GUE matrix and $H_{\gamma_{\rm max}} = H$.

Let $E^{(ij)}$ denote the matrix whose matrix elements are given by $E^{(ij)}_{kl} \deq \delta_{ik}\delta_{jl}$.
Fix $\gamma\ge 1$ and let $(a,b)$ be determined by  $\phi (a, b) = \gamma$. We shall compare $H_{\gamma-1}$ with $H_\gamma$ for each $\gamma$
and then sum up the differences. Note that the matrices $H_{\gamma - 1}$ and $H_\gamma$ differ only in the entries $(a,b)$ and $(b,a)$, and they can be written as
 \be\label{defHg1}
H_{\gamma-1} \;=\; Q + N^{-1/2} V \where V \;\deq\; V_{ab} E^{(ab)} + \ind{a \neq b} V_{ba} E^{(ba)}\,,
\ee
and
$$
H_\gamma \;=\; Q + N^{-1/2} W \where W \;\deq\; W_{ab} E^{(ab)}+ \ind{a \neq b} W_{ba} E^{(ba)}\,;
$$
here the matrix $Q$ satisfies $Q_{ab} = Q_{ba} = 0$.

Next, we introduce the Green functions
\be\label{defG}
		R \;\deq\; \frac{1}{Q-z}\,, \qquad S \;\deq\; \frac{1}{H_{\gamma-1}-z}\,, \qquad T \;\deq\; 
\frac{1}{H_{\gamma}-z}\,,
\ee
which are well-defined for $\eta > 0$ since $Q$ and $H_\gamma$ are self-adjoint.
Using the notation $G^\gamma \deq (H_\gamma - z)^{-1}$, we have the telescopic sum
\begin{equation} \label{telescopic sum}
\E \pb{\im G^{\gamma_{\rm max}}_{\b v \b v}}^n - \E \pb{\im G^{0}_{\b v \b v}}^n \;=\; \sum_{\gamma = 1}^{\gamma_{\rm max}} \pB{\E \pb{\im G^{\gamma}_{\b v \b v}}^n - \E \pb{\im G^{\gamma - 1}_{\b v \b v}}^n}\,.
\end{equation}

For any $K \in \N$ we have the resolvent expansions
\begin{equation} \label{resolvent exp}
S \;=\; \sum_{k = 0}^{K - 1} N^{-k/2} (-RV)^k R + N^{-K/2} (-RV)^K S
\;=\; \sum_{k = 0}^{K - 1} N^{-k/2} R (-V R)^k + N^{-K/2} S (-VR)^K
\end{equation}
and
\begin{equation} \label{resolvent exp 2}
R \;=\; \sum_{k = 0}^{K - 1} N^{-k/2} (SV)^k S + N^{-K/2} (SV)^K R
\;=\; \sum_{k = 0}^{K - 1} N^{-k/2} S (V S)^k + N^{-K/2} R (VS)^K\,.
\end{equation}
Now we choose $K = 10$ in \eqref{resolvent exp 2}. Applying Theorem \ref{LSCTHM} to the Wigner matrix $S$, using the rough bound $\norm{R} \leq \eta^{-1} \leq N$ to estimate the rest term in \eqref{resolvent exp 2}, and recalling \eqref{subexp for h}, we find
 \begin{equation}\label{Rij estimate}
\absb{R_{ij} - \delta_{ij} \msc} \;\leq\; \absb{S_{ij} - \delta_{ij} \msc} + \varphi^{C_\zeta} N^{-1/2} \;\leq\; \varphi^{C_\zeta} \Psi
\end{equation}
\hp{2 \zeta}. Here we also used \eqref{1.5}. Throughout the proof we shall tacitly make use of the bound $\abs{R_{ij}} \leq C$ \hp{2 \zeta}, as follows from \eqref{Rij estimate}.

Next, setting $K = 1$ in \eqref{resolvent exp}, recalling \eqref{subexp for h}, and using Lemma \ref{lemma: Gvi Gvv}, we find
\begin{equation} \label{Sva Rva}
\abs{S_{\b v a} - R_{\b v a}} \;\leq\; N^{-1/2} \varphi^{C_\zeta} \pB{\abs{S_{\b v a} R_{b a}} + \abs{S_{\b v b} R_{a a}}}
\;\leq\; N^{-1/2} \varphi^{C_\zeta} \pBB{\sqrt{\frac{\im S_{\b v \b v}}{N \eta}} + \abs{v_a} + \abs{v_b}}
\end{equation}
\hp{2\zeta}. Now \eqref{Sva Rva}, \eqref{1.5}, and Lemma \ref{lemma: Gvi Gvv} yield
\begin{equation} \label{Rva estimate}
\abs{R_{\b v a}} \;\leq\; \varphi^{C_\zeta} \sqrt{\frac{\im S_{\b v \b v}}{N \eta}} + C \abs{v_a} + \varphi^{C_\zeta} N^{-1/2}
\;\leq\;
\varphi^{C_\zeta} \sqrt{\frac{\im S_{\b v \b v}}{N \eta}} + \varphi^{C_\zeta} \Psi + C \abs{v_a}
\end{equation}
\hp{2\zeta}. The same bound holds for $R_{a \b v}$. Similarly, choosing $K = 1$ in \eqref{resolvent exp} yields, using \eqref{Rva estimate}, that
\begin{equation} \label{Svv Rvv}
\abs{S_{\b v \b v} - R_{\b v \b v}} \;\leq\; N^{-1/2} \varphi^{C_\zeta} \pb{\abs{S_{\b v a} R_{b \b v}} + \abs{S_{\b v b} R_{a \b v}}}
\;\leq\; N^{-1/2} \varphi^{C_\zeta} \pBB{\frac{\im S_{\b v \b v}}{N \eta} + \abs{v_a}^2 + \abs{v_b}^2}
\end{equation}
\hp{2\zeta}.

After these preparations, we may start to estimate
\begin{equation*}
\pb{\im S_{\b v \b v}}^n - \pb{\im R_{\b v \b v}}^n \;=\; \sum_{m = 1}^n A_m \, (\im R_{\b v \b v})^{n - m}\,,
\end{equation*}
where we defined
\begin{equation*}
A_m \;\deq\; \binom{n}{m} \pb{\im S_{\b v \b v} - \im R_{\b v \b v}}^m\,.
\end{equation*}
We choose $K = 4$ in \eqref{resolvent exp} and introduce the notation $S - R = \sum_{k = 1}^4 Y_k$, whereby $Y_k$ has $k$ factors $V$. We write
\begin{equation} \label{Amk def}
A_m \;=\; \sum_{k = m}^{4m}A_{m,k} \where
A_{m,k} \;\deq\; \binom{n}{m} \sum_{k_1, \dots, k_m = 1}^4 \indb{k_1 + \cdots + k_m = k} \prod_{i = 1}^m \im (Y_{k_i})_{\b v \b v}\,.
\end{equation}
Thus we have
\begin{equation} \label{S - R 1}
\E \pb{\im S_{\b v \b v}}^n - \E \pb{\im R_{\b v \b v}}^n \;=\; \cal A + \sum_{m = 1}^n \sum_{k = \max \{4,m\}}^{4m} \E A_{m,k} \pb{\im R_{\b v \b v}}^{n - m}\,,
\end{equation}
where $\cal A$ depends on the randomness only through $Q$ and the first three moments of $V_{ab}$.

We shall prove that
\begin{equation} \label{main bound for rough estimate}
\sum_{m = 1}^n \sum_{k = 4}^{4m} \E \, \abs{A_{m,k}} \pb{\im R_{\b v \b v}}^{n - m} \;\leq\; \frac{\cal E_{ab}}{\log N} \pB{\E \pb{\im S_{\b v \b v}}^n + (\varphi^{C_\zeta} \Phi)^n}\,,
\end{equation}
where we defined
\begin{equation} \label{def Eab}
\cal E_{ab} \;\deq\; \sum_{\sigma, \tau = 0}^2 N^{-2 + \sigma/2 + \tau/2} \abs{v_a}^\sigma \abs{v_b}^\tau\,.
\end{equation}
For future use, we note that the proof of \eqref{main bound for rough estimate} does not require the vanishing of the third moments of $H$ as in \eqref{3rd moment vanishes}.
Before proving \eqref{main bound for rough estimate}, we show how it implies \eqref{rough bound}. Let us abbreviate $X_\gamma \deq \E \pb{\im G^{\gamma}_{\b v \b v}}^n$ and $\cal E_\gamma \deq (\log N)^{-1} \cal E_{\phi^{-1}(\gamma)}$. Note that, since $\im G^\gamma_{\b v \b v} \geq 0$, we have $X_\gamma \geq 0$ for all $\gamma$. Repeating the derivation of \eqref{S - R 1} for $T$ instead of $S$, using that the first three moments of $V_{ab}$ and $W_{ab}$ are the same, and using the estimate \eqref{main bound for rough estimate} and its analogue with $S$ replaced by $T$, we find
\begin{equation*}
X_\gamma - X_{\gamma - 1} \;\leq\; \cal E_\gamma \pB{X_\gamma + X_{\gamma - 1} + \pb{\varphi^{C_\zeta} \Phi}^n}\,.
\end{equation*}
Abbreviating $r_\gamma \deq (1 - \cal E_\gamma)^{-1} (1 + \cal E_\gamma) \geq 1$ we therefore find
\begin{equation*}
X_\gamma \;\leq\; r_\gamma \, X_{\gamma - 1} + r_\gamma \, \cal E_\gamma \pb{\varphi^{C_\zeta} \Phi}^n\,.
\end{equation*}
Since \eqref{rough bound} holds for GOE/GUE, we have the initial estimate $X_0 \leq \pb{\varphi^{C_\zeta} \Phi}^n$. Iteration therefore yields
\begin{equation*}
X_\gamma \;\leq\; \pBB{\prod_{j = 1}^\gamma r_\gamma} \pBB{1 + \sum_{j = 1}^\gamma \cal E_\gamma} \pb{\varphi^{C_\zeta} \Phi}^n\,.
\end{equation*}
Next, we observe that $\sum_\gamma \cal E_\gamma \leq 1$. Since $0 \leq \cal E_\gamma \leq 1/2$, we find $\prod_\gamma r_\gamma \leq C$.
This implies
\begin{equation*}
\E \pb{\im G_{\b v \b v}}^n \;=\; X_{\gamma_{\rm max}} \;\leq\; \pb{\varphi^{C_\zeta} \Phi}^n\,,
\end{equation*}
which is \eqref{rough bound}.

What remains is to prove \eqref{main bound for rough estimate}. Recall that in \eqref{Amk def}, $(Y_k)_{\b v \b v} = N^{-k/2} \qb{(-RV)^k R}_{\b v \b v}$ if $k < 4$ and $(Y_4)_{\b v \b v} = N^{-2} \qb{(-RV)^k S}_{\b v \b v}$. For each $Y_{k_i}$ in \eqref{Amk def}, we write out the matrix multiplication in terms of matrix elements of $S$, $R$, and $V$. Then we multiply everything out. We classify the resulting terms using two additional parameters $s,t \geq 0$. Here $s$ is the total number of matrix elements $R_{\b v a}$, $R_{a \b v}$, $S_{\b v a}$, and $S_{a \b v}$; $t$ is defined similarly with $a$ replaced by $b$. If $a = b$, we use the symmetric convention $s = t$.

We have the conditions
\begin{equation} \label{conditions on s and t}
s + t \;=\; 2m \,, \qquad k \;\geq\; \max\{s,t\}\,.
\end{equation}
The first one is immediate. The second one is clearly true if $a = b$. In order to prove it in the case $a \neq b$, assume for definiteness that $s \geq t$. Then each factor $R_{\b v a}$, $R_{a \b v}$, $S_{\b v a}$, and $S_{a \b v}$ is associated with a unique factor $V_{ab}$ or $V_{ba}$ (the one standing next to it in the matrix product); this proves the second condition of \eqref{conditions on s and t}. Thus we have the decomposition
\begin{equation} \label{def Amkst}
A_{m,k} \;=\; \sum_{s,t = 0}^{k} \ind{s + t = 2m} A_{m,k,s,t}\,,
\end{equation}
in self-explanatory notation.

Using Lemma \ref{lemma: Gvi Gvv} and \eqref{Rva estimate}, we get the bound
\begin{align}
&\mspace{-40mu} \pB{\abs{R_{\b v a}} + \abs{R_{a \b v}} + \abs{S_{\b v a}} + \abs{S_{a \b v}}}^s
\pB{\abs{R_{\b v b}} + \abs{R_{b \b v}} + \abs{S_{\b v b}} + \abs{S_{b \b v}}}^t
\notag \\
&\;\leq\; \pBB{\varphi^{C_\zeta} \sqrt{\frac{\im S_{\b v \b v}}{N \eta}} + \varphi^{C_\zeta} \Psi + C \abs{v_a}}^{s}
\pBB{\varphi^{C_\zeta} \sqrt{\frac{\im S_{\b v \b v}}{N \eta}} + \varphi^{C_\zeta} \Psi + C \abs{v_b}}^{t}
\notag \\
&\;\leq\; \pBB{\varphi^{C_\zeta} \frac{\im S_{\b v \b v}}{N \eta} + \varphi^{C_\zeta} \Psi^2}^m
+ \pBB{\varphi^{C_\zeta} \frac{\im S_{\b v \b v}}{N \eta} + \varphi^{C_\zeta} \Psi^2}^{m - s/2} (C \abs{v_a})^s
\notag \\
&\mspace{40mu} {}+{} \pBB{\varphi^{C_\zeta} \frac{\im S_{\b v \b v}}{N \eta} + \varphi^{C_\zeta} \Psi^2}^{m - t/2} (C \abs{v_b})^t + (C \abs{v_a})^s (C \abs{v_b})^t
\notag \\ \label{estimate of powers of s and t}
&\;\leq\; \varphi^{- D m} \pBB{\varphi^{C_{\zeta, D}} \pbb{\frac{\im S_{\b v \b v}}{N \eta} + \Psi^2 + N^{-1/2}}}^m
\pB{1 + N^{s/4} \abs{v_a}^s + N^{t/4} \abs{v_b}^t + N^{s/4 + t/4} \abs{v_a}^s \abs{v_b}^t}
\end{align}
\hp{2\zeta}, where in the second step we used Lemma \ref{lemma: expanding estimate} below and $s + t \leq \varphi^\zeta$, and in the third step the inequality $x^{m - a} y^a \leq (x + y)^m$. Here $D > 0$ is some constant to be chosen later, and $C_{\zeta, D}$ denotes a constant depending on $\zeta$ and $D$. For the following it will be convenient to abbreviate
\begin{equation*}
\cal F_{ab}(s,t) \;\deq\; 1 + N^{s/4} \abs{v_a}^s + N^{t/4} \abs{v_b}^t + N^{s/4 + t/4} \abs{v_a}^s \abs{v_b}^t\,.
\end{equation*}
Using \eqref{definition of Phi and Psi}, \eqref{1.5}, and Lemma \ref{lemma: expanding estimate} below, we find that there is a constant $C_{\zeta, D}$ such that for $z \in \b S(C_{\zeta, D})$ we have
\begin{multline} \label{estimate on Eva Rvb power}
\pB{\abs{R_{\b v a}} + \abs{R_{a \b v}} + \abs{S_{\b v a}} + \abs{S_{a \b v}}}^s
\pB{\abs{R_{\b v b}} + \abs{R_{b \b v}} + \abs{S_{\b v b}} + \abs{S_{b \b v}}}^t
\\
\leq\; \varphi^{- D m} \pB{\pb{\im S_{\b v \b v}}^m + \pb{\varphi^{C_{\zeta,D}} \Phi}^m} \cal F_{ab}(s,t)
\end{multline}
\hp{2\zeta}.

Next, we observe that \eqref{Svv Rvv} and \eqref{1.5} imply
\begin{equation} \label{estimate on im R vv}
\im R_{\b v \b v} \;\leq\; \pb{1 + \varphi^{C_\zeta} N^{-1/2}} \im S_{\b v \b v} + \varphi^{C_\zeta} \Phi
\end{equation}
\hp{2\zeta}. Recall that, be definition, $A_{m,k,s,t}$ contains $k$ factors $V$, $s$ factors in the set $\{R_{\b v a}, R_{a \b v}, S_{\b v a}, S_{a \b v}\}$, and $t$ factors in the set $\{R_{\b v b}, R_{b \b v}, S_{\b v b}, S_{b \b v}\}$. Therefore the definitions \eqref{Amk def} and \eqref{def Amkst}, as well as the estimates \eqref{subexp for h}, \eqref{estimate on Eva Rvb power}, and \eqref{estimate on im R vv}, yield
\begin{align}
&\mspace{40mu} \abs{A_{m,k,s,t}} (\im R_{\b v \b v})^{n - m}
\notag \\
&\;\leq\; (4 n)^m \varphi^{k C_\zeta} N^{-k/2} \varphi^{- D m} \pB{\pb{\im S_{\b v \b v}}^m + \pb{\varphi^{C_{\zeta,D}} \Phi}^m}
\, \cal F_{ab}(s,t)\, \pB{\pb{1 + \varphi^{C_\zeta} N^{-1/2}} \im S_{\b v \b v} + \varphi^{C_\zeta} \Phi}^{n - m}
\notag \\ \label{A im R 1}
&\;\leq\; \varphi^{(C_\zeta - D) m} N^{-k/2} \pB{(\im S_{\b v \b v})^n + \pb{\varphi^{C_{\zeta, D}} \Phi}^n}
\cal F_{ab}(s,t)
\end{align}
\hp{2\zeta},
where we used that $k \leq 4 m$, that $n \leq \varphi^\zeta$, $\binom{n}{m} \leq n^m$, and Lemma \ref{lemma: expanding estimate} below. Denote by $\Xi$ the event on which the estimate \eqref{A im R 1} holds; thus, $\P(\Xi^c) \leq N^C \exp(- \varphi^{2 \zeta})$. Using \eqref{subexp for h} and the deterministic bound $\norm{R} + \norm{S} \leq N$, it is easy to see that on $\Xi^c$ we have the rough estimate
\begin{equation*}
\E \abs{A_{m,k,s,t}} (\im R_{\b v \b v})^{n - m} \ind{\Xi^c} \;\leq\; N^n \pB{\E \abs{A_{m,k,s,t}}^2}^{1/2} \P(\Xi^c)^{1/2} \;\leq\; (N \varphi^\zeta)^{C \varphi^\zeta} \exp (- c \varphi^{2 \zeta}) \;\leq\; \Phi^{n} N^{-10n}
\end{equation*}
for all $n \leq \varphi^{\zeta}$ and $N$ large enough. Therefore choosing $D \equiv D_\zeta$ large enough we get from \eqref{A im R 1}
\begin{equation*}
\E \abs{A_{m,k,s,t}} (\im R_{\b v \b v})^{n - m}
\;\leq\;
\varphi^{- m} \pB{(\im S_{\b v \b v})^n + \pb{\varphi^{C_{\zeta}} \Phi}^n}
N^{-k/2} \, \cal F_{ab}(s,t)\,.
\end{equation*}
Therefore \eqref{main bound for rough estimate} follows using \eqref{conditions on s and t} if we can prove that
\begin{equation} \label{bound using cal E}
N^{-\max\{4,s,t\}/2}\pB{1 + N^{s/4} \abs{v_a}^s + N^{t/4} \abs{v_b}^t + N^{s / 4 + t / 4} \abs{v_a}^s \abs{v_b}^t}
\;\leq\; 
C \cal E_{ab} \;=\; C \sum_{\sigma,\tau = 0}^2 N^{-2 + \sigma/2 + \tau/2} \abs{v_a}^\sigma \abs{v_b}^\tau\,.
\end{equation}
for all $s,t$. We check that all terms on the left-hand side of \eqref{bound using cal E} are bounded, for all $s,t \geq 0$, by the right-hand side of \eqref{bound using cal E}. The first term is trivial: $N^{-\max\{4,s,t\}/2} \leq N^{-2}$. The second term is bounded by
\begin{equation*}
N^{-\max\{4,s,t\}/2} N^{s/4} \abs{v_a}^s \;\leq\; N^{-2} + N^{-2 + 1/4} \abs{v_a} + N^{-2 + 1} \abs{v_a}^2\,.
\end{equation*}
The third term is bounded similarly. Finally, the last term is bounded by
\begin{equation*}
N^{-\max\{4,s,t\}/2} N^{s/4 + t / 4} \abs{v_a}^s \abs{v_b}^t \;\leq\; E + N^{-2 + 1/2} \abs{v_a} \abs{v_b} + N^{-2 + 1 + 1/4} (\abs{v_a}^2 \abs{v_b} + \abs{v_a} \abs{v_b}^2) + \abs{v_a}^2 \abs{v_b}^2\,,
\end{equation*}
where $E$ denotes a quantity bounded by the three previous terms. This completes the proof of \eqref{bound using cal E}, and hence of \eqref{main bound for rough estimate}.
\end{proof}

What remains is to prove the following elementary result.
\begin{lemma} \label{lemma: expanding estimate}
For $x,y \geq 0$ and $m \in \N$ we have
\begin{equation*}
(x + y)^m \;\leq\; C x^m + (m y)^m\,.
\end{equation*}
\end{lemma}
\begin{proof}
By convexity of the function $x \mapsto x^m$ we have, for any $\lambda \in (0,1)$,
\begin{equation*}
(x + y)^m \;=\; \pbb{(1 - \lambda) \frac{x}{1 - \lambda} + \lambda \frac{y}{\lambda}}^m \;\leq\; \frac{1}{(1 - \lambda)^m} x^m + \frac{1}{\lambda^m} y^m\,.
\end{equation*}
Choosing $\lambda = 1/m$ yields the claim.
\end{proof}

\subsection{Estimate of $G_{\b v \b v} - \msc$} \label{section 3.4}
We now conclude the proof of Proposition \ref{prop: three moment case}. By polarization and linearity, it is enough to prove the following result.

\begin{lemma} \label{lemma: three moment error}
Let $\zeta > 0$ be fixed. Then there exists a constant $C_\zeta$ such that, for all $n \leq \varphi^{\zeta}$, all deterministic and normalized $\b v \in \C^N$, and all $z \in \b S(C_\zeta)$, we have
\begin{equation} \label{three moment claim}
\E \absb{G_{\b v \b v}(z) - \msc(z)}^n \;\leq\; \pb{\varphi^{C_\zeta} \Psi(z)}^n\,.
\end{equation}
\end{lemma}

\begin{proof}
The proof is very similar to that of Lemma \ref{lemma: three moment bound}, whose notation we take over without further comment. In order to avoid dealing with complex numbers, we estimate the real and imaginary parts of $G_{\b v \b v} - \msc$ separately. We give the argument for the real part; the imaginary part is dealt with in the same way. Throughout the following $n$ denotes an even number less than $\varphi^{\zeta}$.

For the GOE/GUE matrix $H_0$ we get from Theorem \ref{LSCTHM}, as in the proof of Lemma \ref{lemma: three moment bound}, that
\begin{equation} \label{error for GUE}
\E \pb{\re G^0_{\b v \b v} - \re \msc}^n \;\leq\; \pb{\varphi^{C_\zeta} \Psi}^n\,.
\end{equation}
In order to perform the comparison step, we write, similarly to \eqref{S - R 1},
\begin{equation*}
\E \pb{\re S_{\b v \b v} - \re \msc}^n - \E \pb{\re R_{\b v \b v} - \re \msc}^n \;=\; \cal B + \sum_{m = 1}^n \sum_{k = \max\{4,m\}}^{4m} \E B_{m,k} \pb{\re R_{\b v \b v} - \re \msc}^{n - m}\,,
\end{equation*}
where $\cal B$ depends on the randomness only through $Q$ and the first three moments of $V_{ab}$, and
\begin{equation*}
B_{m,k} \;\deq\; \binom{n}{m} \sum_{k_1, \dots, k_m = 1}^4 \indb{k_1 + \cdots + k_m = k} \prod_{i = 1}^m \re (Y_{k_i})_{\b v \b v}\,.
\end{equation*}

Similarly to \eqref{main bound for rough estimate}, we shall prove that
\begin{equation} \label{main bound for error}
\sum_{m = 1}^n \sum_{k = 4}^{4m} \E \, \pB{\abs{B_{m,k}} \absb{\re R_{\b v \b v} - \re \msc}^{n - m}} \;\leq\; \frac{\cal E_{ab}}{\log N} \, \E \qBB{\pb{\re S_{\b v \b v} - \re \msc}^n + \pbb{\varphi^{C_\zeta}\frac{\im S_{\b v \b v}}{N \eta}}^n + (\varphi^{C_\zeta} \Psi)^n}\,.
\end{equation}
Using Lemma \ref{lemma: three moment bound}, \eqref{definition of Phi and Psi}, and  \eqref{1.5} we find that the right-hand side of \eqref{main bound for error} is bounded by
\begin{equation*}
\frac{\cal E_{ab}}{\log N} \, \E \qBB{\pb{\re S_{\b v \b v} - \re \msc}^n + (\varphi^{C_\zeta} \Psi)^n}\,.
\end{equation*}
Therefore \eqref{error for GUE} and \eqref{main bound for error} yield \eqref{three moment claim}, exactly as in the paragraph following \eqref{def Eab}.

What remains therefore is to prove \eqref{main bound for error}. Using \eqref{estimate of powers of s and t}, \eqref{1.5}, and Lemma \ref{lemma: expanding estimate} we get, for arbitrary $D > 0$,
\begin{multline} \label{estimate of S S for error}
\pB{\abs{R_{\b v a}} + \abs{R_{a \b v}} + \abs{S_{\b v a}} + \abs{S_{a \b v}}}^s
\pB{\abs{R_{\b v b}} + \abs{R_{b \b v}} + \abs{S_{\b v b}} + \abs{S_{b \b v}}}^t
\\
\leq\; \varphi^{-D m} \pBB{\pbb{\varphi^{C_{\zeta, D}}\frac{\im S_{\b v \b v}}{N \eta}}^m + \pb{\varphi^{C_{\zeta, D}} \Psi}^m} \, \cal F_{ab}(s,t)
\end{multline}
\hp{2\zeta}. Therefore we get, similarly to \eqref{A im R 1},
\begin{multline*}
\abs{B_{m,k,s,t}} \absb{\re R_{\b v \b v} - \re \msc}^{n - m}
\\
\leq\; \varphi^{(C_\zeta - D) m} N^{-k/2} \pBB{\pb{\re S_{\b v \b v} - \re \msc}^n + \pbb{\varphi^{C_{\zeta, D}}\frac{\im S_{\b v \b v}}{N \eta}}^n + \pb{\varphi^{C_{\zeta, D}} \Psi}^n}
\cal F_{ab}(s,t)
\end{multline*}
\hp{2 \zeta}, where we used \eqref{Svv Rvv}, $N^{-1/2} \leq \Psi$, and Lemma \ref{lemma: expanding estimate}. Choosing $D > 0$ large enough and recalling \eqref{bound using cal E} yields \eqref{main bound for error}. (We omit the details of the analysis on the low-probability event, which are similar to those following \eqref{A im R 1}.) This concludes the proof of Lemma \ref{lemma: three moment error}.
\end{proof}

\section{Proof of Theorem \ref{thm: ILSC}, Case {\bf B}} \label{section: two moments}

In this section we prove Theorem \ref{thm: ILSC} in the case {\bf B}, i.e.\ we impose no condition on the third moments of the entries of $H$, and $\Psi(z)$ satisfies \eqref{main assumption on Psi}.  By Markov's inequality, it suffices to prove the following result.

\begin{proposition} \label{prop: two moment case}
Fix $\zeta > 0$. Then there are constants $C_0$ and $C_\zeta$, both depending on $\zeta$, such that the following holds. Assume that $z \in \b S(C_\zeta)$ satisfies \eqref{main assumption on Psi} with constant $C_0$. Then we have, for all $n \leq \varphi^{\zeta}$ and all deterministic  $\b v, \b w \in \C^N$, that
\begin{equation} \label{two moment gen claim}
\E \absb{G_{\b v \b w}(z) - \scalar{\b v}{\b w} \msc(z)}^n \;\leq\; \pb{\varphi^{C_\zeta} \Psi(z) \norm{\b v} \norm{\b w}}^n\,.
\end{equation}
\end{proposition}

The rest of this section is devoted to the proof of Proposition \ref{prop: two moment case}.
We take over the notation of Section \ref{section: three moments}, which we use throughout this section without further comment.

\subsection{Estimate of $\im G_{\b v \b v}$}
In this section we derive an apriori bound on $\im G_{\b v \b v}$ by proving the following result.

\begin{lemma} \label{lemma: final size for 2 moments}
Fix $\zeta > 0$. Then there are large enough constants $C_0$ and $C_\zeta$, both depending on $\zeta$, such that the following holds. Assume that $z \in \b S(C_\zeta)$ satisfies \eqref{main assumption on Psi} with constant $C_0$. Then we have, for all $n \leq \varphi^{\zeta}$ and all deterministic and normalized $\b v \in \C^N$, that
\begin{equation} \label{rough bound for 2 moments}
\E \pb{\im G_{\b v \b v}(z)}^n \;\leq\; \pb{\varphi^{C_\zeta} \Phi(z)}^n\,.
\end{equation}
\end{lemma}
 
The following (trivial) observation will be needed in the next section: The constant $C_0$ may be increased at will without changing $C_\zeta$ in \eqref{rough bound for 2 moments}. 

The main technical estimate behind the proof of Lemma \ref{lemma: final size for 2 moments} is the following lemma. Recall the setup \eqref{ordering map} of the Green function comparison, and in particular the definitions \eqref{defG}.

\begin{lemma} \label{lemma: size estimate for 2 moments}
Fix $\zeta > 0$. Then there are constants $C_0$ and $C_1$, both depending on $\zeta$, such that if \eqref{main assumption on Psi} holds with constant $C_0$ then we have the following.
For any $a,b$ we have
\begin{equation} \label{weak bound for 2 moments final}
\absBB{\sum_{m = 1}^n \sum_{k = \max\h{3,m}}^{4m} \E A_{m,k} (\im R_{\b v \b v})^{n - m}} \;\leq\; \frac{C}{\log N} \pbb{\wt {\cal E}_{ab} + N^{-3/2} \frac{\varphi^{C_1}}{N \eta}}
\pB{\E \pb{\im S_{\b v \b v}}^n + (\varphi^{C_1} \Phi)^n}\,,
\end{equation}
where
\begin{equation*}
\wt {\cal E}_{ab} \;\deq\; \cal E_{ab} + \delta_{ab} \pb{\abs{v_a}^2 + N^{-3/2}} \;=\; \sum_{\sigma, \tau = 0}^2 N^{-2 + \sigma/2 + \tau/2} \abs{v_a}^\sigma \abs{v_b}^\tau + \delta_{ab} \pb{\abs{v_a}^2 + N^{-3/2}}\,.
\end{equation*}
Moreover, if
\begin{equation} \label{smallness on va and vb}
\abs{v_a} + \abs{v_b} \;\leq\; N^{-1/4} \sqrt{\frac{\varphi^{C_1}}{N \eta}}
\end{equation}
then we have the stronger bound
\begin{equation} \label{strong bound for 2 moments final}
\absBB{\sum_{m = 1}^n \sum_{k = \max\h{3,m}}^{4m} \E A_{m,k} (\im R_{\b v \b v})^{n - m}} \;\leq\; \frac{C}{\log N} \, \wt{\cal E}_{ab} \,
\pB{\E \pb{\im S_{\b v \b v}}^n + (\varphi^{C_1} \Phi)^n}\,.
\end{equation}
\end{lemma}

Before proving Lemma \ref{lemma: size estimate for 2 moments}, we use it to complete the proof of Lemma \ref{lemma: final size for 2 moments}.

\begin{proof}[Proof of Lemma \ref{lemma: final size for 2 moments}]
Let $B \subset \{1, \dots, N\}^2$ denote the subset
\begin{equation*}
B \;\deq\; \hBB{(a,b) \,\col\, \abs{v_a} + \abs{v_b} > N^{-1/4} \sqrt{\frac{\varphi^{C_1}}{N \eta}}}\,.
\end{equation*}
Since $\norm{\b v} = 1$, the number of indices $a$ such that $\abs{v_a} \geq \epsilon$ is bounded by $\epsilon^{-2}$. Therefore
\begin{equation*}
\abs{B} \;\leq\; N^{3/2} \pbb{\frac{\varphi^{C_1}}{N \eta}}^{-1}\,.
\end{equation*}
Therefore we have
\begin{equation*}
\sum_{(a,b) \in B} \frac{C}{\log N} \pbb{\wt {\cal E}_{ab} + N^{-3/2} \frac{\varphi^{C_1}}{N \eta}}
+ \sum_{(a,b) \in B^c} \frac{C}{\log N} \, \wt{\cal E}_{ab} \;\leq\; \frac{C}{\log N}\,.
\end{equation*}
Now \eqref{rough bound for 2 moments} follows from \eqref{weak bound for 2 moments final} and \eqref{strong bound for 2 moments final}, by repeating the argument after \eqref{def Eab}.
\end{proof}

Before proving Lemma \ref{lemma: size estimate for 2 moments}, we record the following lower bound on $\eta$.

\begin{lemma}
Let $C_0 > 0$. If \eqref{main assumption on Psi} holds then
\begin{equation} \label{lower bound on eta}
\eta \;\geq\; \varphi^{C_0 / 3} N^{-5/6}\,.
\end{equation}
\end{lemma}
\begin{proof}
The claim follows immediately from $(N\eta)^{-1} \leq \Psi \leq \varphi^{-C_0/3}N^{-1/6}$.
\end{proof}

\begin{proof}[Proof of Lemma \ref{lemma: size estimate for 2 moments}]
Note that the proof of \eqref{main bound for rough estimate} did not use the assumption \eqref{3rd moment vanishes}. In particular, all statements in the proof of Lemma \ref{lemma: three moment bound} after \eqref{conditions on s and t} remain true in the case {\bf B}. By \eqref{main bound for rough estimate}, it is enough to prove
\begin{equation} \label{weak bound for 2 moments}
\absB{\E A_{m,3} (\im R_{\b v \b v})^{n - m}} \;\leq\; \frac{1}{\log N} \pbb{\wt {\cal E}_{ab} + N^{-3/2} \frac{\varphi^{C_\zeta}}{N \eta}}
\pB{\E \pb{\im S_{\b v \b v}}^n + (\varphi^{C_\zeta} \Phi)^n}
\end{equation}
for $m = 1,2,3$
as well as, assuming \eqref{smallness on va and vb},
\begin{equation} \label{strong bound for 2 moments}
\absB{\E A_{m,3} (\im R_{\b v \b v})^{n - m}} \;\leq\; \frac{1}{\log N} \, \wt{\cal E}_{ab} \,
\pB{\E \pb{\im S_{\b v \b v}}^n + (\varphi^{C_\zeta} \Phi)^n}
\end{equation}
for $m = 1,2,3$.
In order to prove \eqref{weak bound for 2 moments} and \eqref{strong bound for 2 moments}, we distinguish four cases depending on $m$ and whether $a = b$. Recall from \eqref{conditions on s and t} that
\begin{equation} \label{conditions for k=3}
s + t \;=\; 2m \,, \qquad s \;\leq\; 3\,, \qquad t \;\leq\; 3\,.
\end{equation}

\textbf{Case (i): $a = b$ and $m \leq 3$.} Similarly to \eqref{estimate of powers of s and t}, we find
\begin{equation*}
\pB{\abs{R_{\b v a}} + \abs{R_{a \b v}}}^{2m} \;\leq\; \varphi^{-D m} \pB{\im S_{\b v \b v} + \varphi^{C_{\zeta, D}} \Phi}^m \pb{1 + N^{m/2} \abs{v_a}^{2m}}
\end{equation*}
\hp{2 \zeta},
for any constant $D > 0$ and $z \in \b S(C_{\zeta, D})$. Therefore \eqref{estimate on im R vv} yields
\begin{align*}
\abs{A_{m,3}} (\im R_{\b v \b v})^{n - m} &\;\leq\; \varphi^{C_\zeta - D m} N^{-3/2} \pB{\im S_{\b v \b v} + \varphi^{C_{\zeta, D}} \Phi}^n \pb{1 + N^{m/2} \abs{v_a}^{2m}}
\\
&\;\leq\; \varphi^{-1} \pB{\im S_{\b v \b v} + \varphi^{C_{\zeta}} \Phi}^n \pb{N^{-3/2} + \abs{v_a}^2}
\end{align*}
\hp{2 \zeta}, where we used that $1 \leq m \leq 3$. Therefore Lemma \ref{lemma: expanding estimate} yields
\begin{equation} \label{case (i)}
\E \abs{A_{m,3}} (\im R_{\b v \b v})^{n - m} \;\leq\; C \varphi^{-1} \pB{\E \pb{\im S_{\b v \b v}}^n + (\varphi^{C_\zeta} \Phi)^n} \pb{N^{-3/2} + \abs{v_a}^2}\,,
\end{equation}
which is \eqref{strong bound for 2 moments}. In particular, we have also proved \eqref{weak bound for 2 moments}.
Here we omit the details of the estimate on the event of low probability, which are analogous to those following \eqref{A im R 1}.

\textbf{Case (ii): $a \neq b$ and $m = 3$.} By \eqref{conditions for k=3}, we have $s = t = 3$. From \eqref{estimate of powers of s and t} we get
\begin{multline} \label{estimate for mst3}
\pB{\abs{R_{\b v a}} + \abs{R_{a \b v}}}^s
\pB{\abs{R_{\b v b}} + \abs{R_{b \b v}}}^t
\;\leq\; \pBB{\varphi^{C_\zeta} \frac{\im S_{\b v \b v}}{N \eta} + \varphi^{C_\zeta} \Psi^2}^m
+ \pBB{\varphi^{C_\zeta} \frac{\im S_{\b v \b v}}{N \eta} + \varphi^{C_\zeta} \Psi^2}^{m - s/2} (C \abs{v_a})^s
\\
+ \pBB{\varphi^{C_\zeta} \frac{\im S_{\b v \b v}}{N \eta} + \varphi^{C_\zeta} \Psi^2}^{m - t/2} (C \abs{v_b})^t + (C \abs{v_a})^s (C \abs{v_b})^t
\end{multline}
\hp{2 \zeta}. Together with \eqref{definition of Phi and Psi} and \eqref{estimate on im R vv}, this yields
\begin{multline} \label{naive estimate of Rva}
\pB{\abs{R_{\b v a}} + \abs{R_{a \b v}}}^s
\pB{\abs{R_{\b v b}} + \abs{R_{b \b v}}}^t (\im R_{\b v \b v})^{n - m}
\;\leq\; \pb{\im S_{\b v \b v} + \varphi^{C_{\zeta, D}} \Phi}^n
\\
\times \qBB{\pbb{\frac{\varphi^{C_\zeta}}{N \eta}}^m + \pbb{\frac{\varphi^{C_\zeta}}{N \eta}}^{s/2} \pb{\varphi^D \Phi}^{-t/2} \abs{v_b}^t + \pbb{\frac{\varphi^{C_\zeta}}{N \eta}}^{t/2} \pb{\varphi^D \Phi}^{-s/2} \abs{v_a}^s + \pb{\varphi^D \Phi}^{-s/2 - t/2} \abs{v_a}^s \abs{v_b}^t}
\end{multline}
\hp{2 \zeta} and for any $D > 0$. Choosing $D$ and $C_0$ in \eqref{main assumption on Psi} large enough, we get from \eqref{subexp for h}, \eqref{lower bound on eta}, Lemma \ref{lemma: expanding estimate}, and $N^{-1/2} \leq \Phi$ that
\begin{equation*}
\abs{A_{3,3}} (\im R_{\b v \b v})^{n - 3} \;\leq\; \varphi^{-1} N^{-3/2} \pB{(\im S_{\b v \b v})^n + (\varphi^{C_\zeta} \Phi)^n} \pB{N^{-1/2} + N^{1/2} \abs{v_b}^2 + N^{1/2} \abs{v_a}^2 + N^{3/2} \abs{v_a}^2 \abs{v_b}^2}
\end{equation*}
\hp{2 \zeta}. Now \eqref{strong bound for 2 moments}, and hence also \eqref{weak bound for 2 moments}, follows easily (we omit the details of the analysis on the low-probability event).

\textbf{Case (iii): $a \neq b$ and $m = 2$.} Consider first the case $s = t = 2$. Then $A_{2,3,2,2}$ (see \eqref{def Amkst} and \eqref{Amk def}) is a finite sum of $O(1)$ terms of the form
\begin{equation} \label{definition X1}
X_1 \;\deq\; R_{\b v a} h_{ab} R_{b \b v} \, R_{\b v a} h_{ab} R_{ba} h_{ab} R_{b \b v}\,.
\end{equation}
(The other terms can be obtained from \eqref{definition X1} by permutation of indices and complex conjugation of factors.)
We shall estimate the contribution of $X_1$; the other terms are dealt with in exactly the same way.
Note the presence of an off-diagonal resolvent matrix element $R_{ba}$, as required by the condition $s = t = 2$. From \eqref{Rij estimate} and \eqref{naive estimate of Rva} we get, with $m = s = t = 2$, that
\begin{multline*}
\abs{X_1} \, (\im R_{\b v \b v})^{n - 2} \;\leq\; \varphi^{C_\zeta} \, \Psi\,N^{-3/2}
\, \pb{\im S_{\b v \b v} + \varphi^{C_{\zeta, D}} \Phi}^n
\\
\times \qBB{\pbb{\frac{\varphi^{C_\zeta}}{N \eta}}^2 + \frac{\varphi^{C_\zeta}}{N \eta} \pb{\varphi^D \Phi}^{-1} \abs{v_b}^2 + \frac{\varphi^{C_\zeta}}{N \eta} \pb{\varphi^D \Phi}^{-1} \abs{v_a}^2 + \pb{\varphi^D \Phi}^{-2} \abs{v_a}^2 \abs{v_b}^2}
\end{multline*}
\hp{2 \zeta}.
Note the factor $\Psi$ arising from the estimate of $R_{ba}$. Choosing $D$ and $C_0$ large enough, and
recalling \eqref{main assumption on Psi}, we find using Lemma \ref{lemma: expanding estimate} that
\begin{equation*}
\abs{X_1} \, (\im R_{\b v \b v})^{n - 2} \;\leq\; \varphi^{-1} \pB{(\im S_{\b v \b v})^n + (\varphi^{C_\zeta} \Phi)^n} \cal E_{ab}
\end{equation*}
\hp{2 \zeta}.
This yields \eqref{strong bound for 2 moments} and hence also \eqref{weak bound for 2 moments}.

Let us therefore consider the case $s = 3$ and $t = 1$. (The case $s = 1$ and $t = 3$ is estimated in the same way.) Using the bounds $\Phi \geq (N \eta)^{-1}$ and $\Phi \geq N^{-1/2}$, we find
\begin{align}
&\mspace{-40mu}\abs{A_{2,3,3,1}} \, (\im R_{\b v \b v})^{n - 2}
\\
&\;\leq\; \varphi^{C_\zeta} \, N^{-3/2}
\, \pb{\im S_{\b v \b v} + \varphi^{C_{\zeta, D}} \Phi}^n
\notag \\
&\mspace{30mu} \times \qBB{\pbb{\frac{\varphi^{C_\zeta}}{N \eta}}^2 + \pbb{\frac{\varphi^{C_\zeta}}{N \eta}}^{3/2} \pb{\varphi^D \Phi}^{-1/2} \abs{v_b} + \pbb{\frac{\varphi^{C_\zeta}}{N \eta}}^{1/2} \pb{\varphi^D \Phi}^{-3/2} \abs{v_a}^2 + \pb{\varphi^D \Phi}^{-2} \abs{v_a}^2 \abs{v_b}}
\notag \\ \label{proof of rough bound}
&\;\leq\; \varphi^{-1} \pB{(\im S_{\b v \b v})^n + (\varphi^{C_\zeta} \Phi)^n} \qBB{N^{-3/2} \frac{\varphi^{C_\zeta}}{N \eta} + N^{-3/2} \abs{v_b} + N^{-1} \abs{v_a}^2 + N^{-1/2} \abs{v_a}^2 \abs{v_b}}
\end{align}
\hp{2 \zeta}, for $D$ and $C_0$ large enough. This yields \eqref{weak bound for 2 moments} in the case $s = 3$ and $t = 1$.

In order to prove the stronger bound \eqref{strong bound for 2 moments} in the case $s = 3$ and $t = 1$, we note that \eqref{Rva estimate}, \eqref{definition of Phi and Psi}, \eqref{1.5}, and the assumption \eqref{smallness on va and vb} yield
\begin{equation} \label{stronger bound on Rva}
\abs{R_{\b v a}} \;\leq\; \varphi^{C_\zeta} \sqrt{\frac{\im S_{\b v \b v} + \Phi}{N \eta}}\,.
\end{equation}
The same bound holds for $R_{a \b v}$, $R_{\b v b}$, and $R_{b \b v}$. Now $A_{2,3,3,1}$ is a finite sum of $O(1)$ terms of the form
\begin{equation*}
X_2 \;\deq\; R_{\b v a} h_{ab} R_{b \b v} \, R_{\b v a} h_{ab} R_{bb} h_{ba} R_{a \b v}\,.
\end{equation*}
(Again, the other terms can be obtained from $X_2$ by permutation of indices and complex conjugation of factors.)
We shall show that
\begin{equation} \label{strong estimate for 2,3,3,1}
\absb{\E X_2 (\im R_{\b v \b v})^{n - 2}} \;\leq\; C \varphi^{-1} \cal E_{ab} \pB{\E (\im S_{\b v \b v})^n + (\varphi^{C_\zeta} \Phi)^n}\,.
\end{equation}
We split $R_{bb} = (R_{bb} - \msc) + \msc$ in the definition of $X_2$. The first resulting term is estimated, using \eqref{Rij estimate}, by
\begin{equation*}
\varphi^{C_\zeta} \, \Psi\, N^{-3/2} \absb{R_{\b v a} R_{b \b v} R_{\b v a} R_{a \b v}} (\im R_{\b v \b v})^{n - 2}\,.
\end{equation*}
The estimate of $\abs{X_1} (\im S_{\b v \b v})^{n - 2}$ above may now be applied verbatim. What remains is the second term resulting from the above splitting of $X_2$. Since $\abs{\msc} \leq C$ and $h_{ab}$ is independent of $R$, we therefore have to show that
\begin{equation} \label{strong estimate step 2}
C N^{-3/2} \absB{\E R_{\b v a} R_{b \b v} R_{\b v a} R_{a \b v} \, (\im R_{\b v \b v})^{n - 2}} \;\leq\; \varphi^{-1} \cal E_{ab} \pB{\E (\im S_{\b v \b v})^n + (\varphi^{C_\zeta} \Phi)^n}\,.
\end{equation}

Using \eqref{sq root formula}, we expand
\begin{equation} \label{splitting of R b v}
R_{b \b v}\;=\; \msc \cal R_{b \b v} + \cal R_{b \b v}'\,,
\end{equation}
where we defined (see also \eqref{def G prime})
\begin{equation} \label{def R prime R double prime}
\cal R_{b \b v} \;\deq\; - \sum_k^{(b)} h_{b k} R^{(b)}_{k \b v} \,, \qquad
\cal R'_{b \b v} \;\deq\; v_b R_{bb} + (R_{bb} - \msc) \cal R_{b \b v}\,.
\end{equation}
Now we observe that, using the bound \eqref{Rij estimate}, we may repeat the proof of Lemma \ref{lemma: Gvi Gvv} to the letter to find that its statement holds with $(G, \cal G)$ replaced with $(R, \cal R)$. Thus we find
\begin{equation} \label{estimate of R prime}
\abs{\cal R_{b \b v}} \;\leq\; \varphi^{C_\zeta} \sqrt{\frac{\im R_{\b v \b v}}{N \eta}} + C \abs{v_b} \;\leq\; \varphi^{C_\zeta} \sqrt{\frac{\im S_{\b v \b v} + \Phi}{N \eta}} + \varphi^{C_\zeta} N^{-1/4} (N \eta)^{-1/2} \;\leq\; \varphi^{C_\zeta} \sqrt{\frac{\im S_{\b v \b v} + \Phi}{N \eta}}
\end{equation}
\hp{2 \zeta},
where in the second step we used \eqref{estimate on im R vv} and \eqref{smallness on va and vb}, and in the last step \eqref{1.5}. Using \eqref{Rij estimate}, \eqref{smallness on va and vb}, and $\Phi \geq (N \eta)^{-1}$, we therefore find
\begin{equation} \label{estimate of R double prime}
\abs{\cal R'_{b \b v}} \;\leq\; \pbb{\varphi^{C_\zeta} \frac{\Psi}{\sqrt{N \eta}} + \varphi^{- D} N^{-1/4}} (\im S_{\b v \b v} + \varphi^{C_{\zeta,D}}\Phi)^{1/2}
\end{equation}
\hp{2 \zeta}, for any $D \geq 0$.
Therefore \eqref{estimate on im R vv} and \eqref{stronger bound on Rva} yield
\begin{equation*}
C N^{-3/2} \absB{\E R_{\b v a} \cal R'_{b \b v} R_{\b v a} R_{a \b v} \, (\im R_{\b v \b v})^{n - 2}} \;\leq\; N^{-3/2} \varphi^{C_\zeta} (N \eta)^{-3/2} \pbb{\frac{\Psi}{\sqrt{N \eta}} + N^{-1/4}} \, \pb{\im S_{\b v \b v} + \varphi^{C_\zeta} \Phi}^n
\end{equation*}
\hp{2 \zeta}. Using \eqref{main assumption on Psi}, \eqref{lower bound on eta}, and Lemma \ref{lemma: expanding estimate}, we find that the right-hand side is bounded by
\begin{equation*}
\varphi^{-1} N^{-2} \pB{(\im S_{\b v \b v})^n + (\varphi^{C_\zeta} \Phi)^n}
\end{equation*}
\hp{2 \zeta}. Combined with the usual estimate on the complementary low-probability event, this concludes the estimate of the $\cal R'_{b \b v}$-term. What remains is to prove that
\begin{equation} \label{strong estimate step 3}
C N^{-3/2} \absB{\E R_{\b v a} \cal R_{b \b v} R_{\b v a} R_{a \b v} \, (\im R_{\b v \b v})^{n - 2}} \;\leq\; \varphi^{-1} \cal E_{ab} \pB{\E (\im S_{\b v \b v})^n + (\varphi^{C_\zeta} \Phi)^n}\,,
\end{equation}
The key observation behind the estimate of \eqref{strong estimate step 3} is that $\E_b \cal R_{b \b v} = 0$, where $\E_b$ denotes partial expectation with respect to the $b$-th column of $Q$. Thus we have
\begin{equation*}
\E R_{\b v a} \cal R_{b \b v} R_{\b v a} R_{a \b v} \, (\im R_{\b v \b v})^{n - 2} \;=\; \E \qB{R_{\b v a} R_{\b v a} R_{a \b v} \, (\im R_{\b v \b v})^{n - 2} - R^{(b)}_{\b v a} R^{(b)}_{\b v a} R^{(b)}_{a \b v} \, (\im R^{(b)}_{\b v \b v})^{n - 2}} \cal R_{b \b v}\,.
\end{equation*}
In order to compare the quantities in the brackets, we use \eqref{Gij Gijk}, \eqref{Rij estimate}, and \eqref{stronger bound on Rva} to get
\begin{align}
R_{\b v a} &\;=\; R_{\b v a}^{(b)} + \frac{R_{\b v b} R_{b a}}{R_{bb}} \;=\; R_{\b v a}^{(b)} + O(\varphi^{C_\zeta} \Psi R_{\b v b})\,,
\label{Rva marginal}
\\ \label{Rvv marignal}
R_{\b v \b v} &\;=\; R_{\b v \b v}^{(b)} + \frac{R_{\b v b} R_{b \b v}}{R_{bb}} \;=\; R_{\b v \b v}^{(b)} + O \pbb{\varphi^{C_\zeta}\, \frac{\im S_{\b v \b v} + \Phi}{N \eta}}
\end{align}
\hp{2 \zeta}. In particular, we get from \eqref{estimate on im R vv} and \eqref{stronger bound on Rva} that
\begin{equation} \label{estimate on minor R v a}
\im R_{\b v \b v}^{(b)} \;\leq\; (1 + \varphi^{-\zeta}) \im S_{\b v \b v} + \varphi^{C_\zeta} \Phi \,, \qquad \abs{R_{\b v a}^{(b)}} \;\leq\; \varphi^{C_\zeta} \sqrt{\frac{\im S_{\b v \b v} + \Phi}{N \eta}}
\end{equation}
\hp{2 \zeta}, for $z \in \b S(C'_{\zeta})$ with some large enough $C'_\zeta$. A telescopic estimate of the form
\begin{equation*}
\prod_{i = 1}^k (x_i + y_i) - \prod_{i = 1}^k x_i \;=\; \sum_{j = 1}^k \pBB{\prod_{i = 1}^{j - 1} x_i} y_j \pBB{\prod_{i = j + 1}^k (x_i + y_i)}
\end{equation*}
therefore gives
\begin{align*}
&\mspace{-40mu} C N^{-3/2} \, \absB{R_{\b v a} R_{\b v a} R_{a \b v} \, (\im R_{\b v \b v})^{n - 2} - R^{(b)}_{\b v a} R^{(b)}_{\b v a} R^{(b)}_{a \b v} \, (\im R^{(b)}_{\b v \b v})^{n - 2}} \, \absb{\cal R_{b \b v}}
\\
&\leq\; \varphi^{C_\zeta} N^{-3/2} \abs{\cal R_{b \b v}} \pbb{\frac{\im S_{\b v \b v} + \Phi}{N \eta}}^{3/2} \, \Psi \, \pb{\im S_{\b v \b v} + \varphi^{C_\zeta} \Phi}^{n - 2}
\\
&\mspace{40mu} + n \varphi^{C_\zeta} N^{-3/2} \abs{\cal R_{b \b v}} \pbb{\frac{\im S_{\b v \b v} + \Phi}{N \eta}}^{5/2} \pb{\im S_{\b v \b v} + \varphi^{C_\zeta} \Phi}^{n - 3}
\\
&\leq\; \varphi^{C_\zeta} N^{-3/2} \sqrt{\frac{\im S_{\b v \b v} + \Phi}{N \eta}} \, \pBB{\frac{\Psi}{(N \eta)^{3/2}} + \frac{1}{(N \eta)^{5/2}}} \pb{\im S_{\b v \b v} + \varphi^{C_\zeta} \Phi}^{n - 1/2}
\end{align*}
\hp{2 \zeta}, where in the last step we used \eqref{estimate of R prime} and $n \leq \varphi^\zeta$. Now \eqref{strong estimate step 3} follows easily for large enough $C_0$ in \eqref{main assumption on Psi}, using \eqref{main assumption on Psi} and \eqref{lower bound on eta}.
This concludes the proof of \eqref{strong estimate step 2} and hence of \eqref{strong estimate for 2,3,3,1}.

\textbf{Case (iv): $a \neq b$ and $m = 1$.} Similarly to \eqref{proof of rough bound}, one easily finds the weak bound \eqref{weak bound for 2 moments}. Let us therefore assume \eqref{smallness on va and vb} and prove \eqref{strong bound for 2 moments}. It suffices to prove that
\begin{equation} \label{claim for case 4}
N^{-3/2} \absb{\E \, X_3  (\im R_{\b v \b v})^{n - 1}} \;\leq\; \varphi^{-1} N^{-2} \pB{\E \pb{\im S_{\b v \b v}}^n + (\varphi^{C_\zeta} \Phi)^n}\,,
\end{equation}
where $X_3$ stands for any of the following expressions:
\begin{equation*}
R_{\b v a} R_{ba} R_{ba} R_{b \b v} \,, \qquad
R_{\b v a} R_{bb} R_{ab} R_{a \b v} \,, \qquad
R_{\b v a} R_{bb} R_{aa} R_{b \b v} \,.
\end{equation*}
Here we used that $h_{ab}$ and $h_{ba}$ are independent of $R$.
(Up to an immaterial renaming of indices and complex conjugation, all terms in $A_{1,3}$ are covered by one of these three cases.) Applying the splittings $R_{aa} = \msc + (R_{aa} - \msc)$ and  $R_{bb} = \msc + (R_{bb} - \msc)$, we find that it suffices to prove \eqref{claim for case 4} for $X_3$ being any of
\begin{gather*}
R_{\b v a} R_{ba} R_{ba} R_{b \b v} \,, \qquad
R_{\b v a} (R_{bb} - \msc) R_{ab} R_{a \b v}\,, \qquad
R_{\b v a} (R_{bb} - \msc) (R_{aa} - \msc) R_{b \b v}\,,
\\
R_{\b v a} R_{ab} R_{a \b v}\,, \qquad
R_{\b v a} (R_{bb} - \msc) R_{b \b v}\,, \qquad
R_{\b v a} (R_{aa} - \msc) R_{b \b v}\,,
\\
R_{\b v a}  R_{b \b v}\,.
\end{gather*}
Next, applying the splitting \eqref{splitting of R b v} to the last line, we find that it suffices to prove \eqref{claim for case 4} for $X_3$ being any of
\begin{subequations}
\begin{gather} \label{category 1}
R_{\b v a} R_{ba} R_{ba} R_{b \b v} \,, \qquad
R_{\b v a} (R_{bb} - \msc) R_{ab} R_{a \b v}\,, \qquad
R_{\b v a} (R_{bb} - \msc) (R_{aa} - \msc) R_{b \b v}\,, \qquad
\cal R_{\b v a}' \cal R_{b \b v}'\,,
\\ \label{category 2}
R_{\b v a} R_{ab} R_{a \b v}\,, \qquad
R_{\b v a} (R_{bb} - \msc) R_{b \b v}\,, \qquad
R_{\b v a} (R_{aa} - \msc) R_{b \b v}\,, \qquad
\cal R_{\b v a}'  \cal R_{b \b v}\,, \qquad
\cal R_{\b v a}  \cal R_{b \b v}'\,,
\\ \label{category 3}
\cal R_{\b v a}  \cal R_{b \b v}\,.
\end{gather}
\end{subequations}

For $X_3$ in \eqref{category 1}, we find from \eqref{Rij estimate}, \eqref{stronger bound on Rva}, and \eqref{estimate of R double prime} that
\begin{equation*}
\abs{X_3} \;\leq\; \varphi^{C_\zeta} \pbb{\frac{\Psi^2}{N \eta} + \varphi^{-D} N^{-1/2}} \pb{\im S_{\b v \b v} + \varphi^{C_{\zeta, D}} \Phi}
\end{equation*}
\hp{2 \zeta},
from which \eqref{claim for case 4} easily follows using \eqref{main assumption on Psi}, \eqref{lower bound on eta}, \eqref{estimate on im R vv}, and Lemma \ref{lemma: expanding estimate}, having chosen $D$ and $C_0$ in \eqref{main assumption on Psi} large enough.

Let us now consider $X_3 = R_{\b v a} R_{ab} R_{a \b v}$. Using \eqref{sq root formula}, we split, similarly to \eqref{splitting of R b v},
\begin{equation*}
R_{ab} \;=\; \msc \cal R_{ab} + (R_{bb} - \msc) \cal R_{ab}\,, \qquad \cal R_{ab} \;\deq\; - \sum_{k}^{(b)} R^{(b)}_{ak} h_{kb}\,.
\end{equation*}
Using \eqref{LDE}, \eqref{definition of Phi and Psi}, \eqref{Gij Gijk}, and \eqref{Rij estimate}, we find
\begin{equation} \label{def R and cal R}
\abs{\cal R_{ab}} \;\leq\; \varphi^{C_\zeta} \pBB{\frac{1}{N} \sum_k^{(b)} \abs{R_{ak}^{(b)}}^2}^{1/2} \;=\; \varphi^{C_\zeta} \pbb{\frac{1}{N \eta} \im R_{aa}^{(b)}}^{1/2} \;\leq\; \varphi^{C_\zeta} \pbb{\frac{\im \msc + \Psi}{N \eta}}^{1/2} \;\leq\; \varphi^{C_\zeta} \Psi
\end{equation}
\hp{2 \zeta}. For the second part of $X_3$ resulting from the splitting of $R_{ab}$, we therefore get the estimate
\begin{equation*}
\absb{R_{\b v a} (R_{bb} - \msc) \cal R_{ab} R_{a \b v}} \;\leq\; \varphi^{C_\zeta} \frac{\Psi^2}{N \eta} \pb{\im S_{\b v \b v} + \varphi^{C_\zeta} \Phi} \;\leq\; \varphi^{-1} N^{-1/2} \pb{\im S_{\b v \b v} + \varphi^{C_\zeta} \Phi}
\end{equation*}
\hp{2 \zeta}. For the first part, we use $\E_b \cal R_{ab}$ to write
\begin{equation*}
\E R_{\b v a} \cal R_{ab} R_{a \b v} (\im R_{\b v \b v})^{n - 1} \;=\; \E \qB{R_{\b v a} R_{a \b v} (\im R_{\b v \b v})^{n - 1} - R_{\b v a}^{(b)} R_{a \b v}^{(b)} (\im R_{\b v \b v}^{(b)})^{n - 1}} \cal R_{ab}\,.
\end{equation*}
This may be estimated using a telescopic sum, exactly as \eqref{strong bound for 2 moments}; we omit the details.  This completes the proof of \eqref{claim for case 4} in the case $X_3 = R_{\b v a} R_{ab} R_{a \b v}$.
The second and third terms of \eqref{category 2} are estimated similarly.

For the choice $X_3 = \cal R_{\b v a} \cal R_{b \b v}'$, we use $\E_a \cal R_{\b v a} = 0$ to write
\begin{equation} \label{telescopic for second category}
\E \cal R_{\b v a}  \cal R_{b \b v}' (\im R_{\b v \b v})^{n - 1} \;=\; \E \qB{\cal R_{b \b v}' (\im R_{\b v \b v})^{n - 1} - \pb{\cal R_{b \b v}'}^{(a)} (\im R_{\b v \b v}^{(a)})^{n - 1}}  \cal R_{\b v a}\,,
\end{equation}
where we defined
\begin{equation*}
\pb{\cal R'_{b \b v}}^{(a)} \;\deq\; v_b R_{bb}^{(a)} + (R_{bb}^{(a)} - \msc) \cal R_{b \b v}^{(a)}\,, \qquad \cal R_{b \b v}^{(a)} \;\deq\; - \sum_k^{(ab)} h_{b k} R^{(ab)}_{k \b v}\,.
\end{equation*}
We find
\begin{multline} \label{cal R minor estimate}
\absb{\cal R_{b \b v}^{(a)} - \cal R_{b \b v}} \;\leq\; \abs{h_{ba}} \abs{R_{a \b v}^{(b)}} + \absBB{\sum_{k}^{(ab)} h_{bk} \pb{R^{(b)}_{k \b v} - R^{(ab)}_{k \b v}}}
\\
\leq\; \varphi^{C_\zeta} N^{-1/2} \sqrt{\frac{\im S_{\b v \b v} + \Phi}{N \eta}} + \varphi^{C_\zeta} \abs{R_{a \b v}^{(b)}} \pBB{\frac{1}{N} \sum_{k} \abs{R_{k a}^{(b)}}^2}^{1/2} \;\leq\; \varphi^{C_\zeta} \frac{\Psi}{\sqrt{N \eta}} \pb{\im S_{\b v \b v} + \Phi}^{1/2}
\end{multline}
\hp{2 \zeta}, where we used \eqref{estimate on minor R v a}, \eqref{subexp for h},  \eqref{LDE}, \eqref{Gij Gijk}, and \eqref{Rva marginal}. Together with \eqref{Gij Gijk}, \eqref{Rij estimate}, \eqref{stronger bound on Rva}, and \eqref{smallness on va and vb}, we therefore find
\begin{align}
\absb{\pb{\cal R'_{b \b v}}^{(a)} - \cal R'_{b \b v}} &\;\leq\; \pb{\abs{v_b} + \abs{\cal R_{b \b v}}} \abs{R_{bb}^{(a)} - R_{bb}} + \abs{R_{bb}^{(b)} - \msc} \varphi^{C_\zeta} \frac{\Psi}{\sqrt{N \eta}} \pb{\im S_{\b v \b v} + \Phi}^{1/2}
\notag \\ \label{estimate on R double prime}
&\;\leq\; \varphi^{C_\zeta} \frac{\Psi^2}{\sqrt{N \eta}} \pb{\im S_{\b v \b v} + \Phi}^{1/2}
\end{align}
\hp{2 \zeta}.
Recalling \eqref{estimate of R prime}, \eqref{Rvv marignal}, Lemma \ref{lemma: expanding estimate}, and the usual rough estimate on the complementary low-probability event, a telescopic estimate in \eqref{telescopic for second category} therefore gives
\begin{equation*}
\E \cal R_{\b v a}  \cal R_{b \b v}' (\im R_{\b v \b v})^{n - 1} \;\leq\; \varphi^{C_\zeta} \pbb{\frac{\Psi^2}{N \eta} + \frac{\Psi}{(N \eta)^2}} \pB{\E (\im S_{\b v \b v})^n + \pb{\varphi^{C_\zeta} \Phi}^n}\,.
\end{equation*}
Now \eqref{claim for case 4} follows.

Now we prove \eqref{claim for case 4} for $X_3$ as in \eqref{category 3}. We begin with a graded expansion of $R_{\b v \b v}$. Using \eqref{Gvw minor} we find
\begin{equation*}
R_{\b v \b v} \;=\; R_{\b v \b v}^{(a)} + \frac{R_{\b v a} R_{a \b v}}{R_{aa}} \;=\; R_{\b v \b v}^{(ab)} + \frac{R_{\b v b}^{(a)} R_{b \b v}^{(a)}}{R_{bb}^{(a)}} + \frac{R_{\b v a} R_{a \b v}}{R_{aa}}\,.
\end{equation*}
We deal with the last term by applying \eqref{Gij Gijk} twice, followed by
\begin{equation*}
\frac{1}{R_{aa}} \;=\; \frac{1}{R_{aa}^{(b)}} - \frac{R_{ab} R_{ba}}{R_{aa} R_{bb} R_{aa}^{(b)}}\,,
\end{equation*}
itself an immediate consequence of \eqref{Gij Gijk}. This gives the graded expansion
\begin{equation*}
R_{\b v \b v} \;=\; R_{\b v \b v}^{[ab]} + R_{\b v \b v}^{[a]} + R_{\b v \b v}^{[b]} + R_{\b v \b v}^{[\emptyset]}\,,
\end{equation*}
where
\begin{gather*}
R_{\b v \b v}^{[ab]} \;\deq\; R_{\b v \b v}^{(ab)}\,, \qquad
R_{\b v \b v}^{[a]} \;\deq\; \frac{R_{\b v b}^{(a)} R_{b \b v}^{(a)}}{R_{bb}^{(a)}}\,, \qquad
R_{\b v \b v}^{[b]} \;\deq\; \frac{R_{\b v a}^{(b)} R_{a \b v}^{(b)}}{R_{aa}^{(b)}}
\\
R_{\b v \b v}^{[\emptyset]} \;\deq\; \frac{R_{\b v b} R_{b a} R_{a \b v}}{R_{aa} R_{bb}} + \frac{R_{\b v a}^{(b)} R_{ab} R_{b \b v}}{R_{aa} R_{bb}} - \frac{R_{\b v a}^{(b)} R_{a \b v}^{(b)} R_{ab} R_{ba}}{R_{aa} R_{bb} R_{aa}^{(b)}}\,.
\end{gather*}
Note that $R_{\b v}^{[\bb T]}$ is independent of the columns of $H$ indexed by $\bb T$. Moreover, by \eqref{Rij estimate}, Lemma \ref{lemma: msc}, \eqref{Gij Gijk}, \eqref{stronger bound on Rva}, and \eqref{estimate on minor R v a}, we have
\begin{equation} \label{estimates on graded terms}
\absb{R_{\b v \b v}^{[a]}} + \absb{R_{\b v \b v}^{[b]}} \;\leq\; \varphi^{C_\zeta} \frac{\im S_{\b v \b v} + \Phi}{N \eta}\,, \qquad
\absb{R_{\b v \b v}^{[\emptyset]}} \;\leq\; \varphi^{C_\zeta} \Psi \frac{\im S_{\b v \b v} + \Phi}{N \eta}
\end{equation}
\hp{2 \zeta}. Thus we write
\begin{equation} \label{sum over A}
\E \, \cal R_{\b v a} \cal R_{b \b v} (\im R_{\b v \b v})^{n - 1} \;=\; \sum_{\b A} \E \, \cal R_{\b v a} \cal R_{b \b v} \prod_{i = 1}^{n - 1} \im R_{\b v \b v}^{[A_i]}
\end{equation}
where $\b A = (A_i)_{i = 1}^{n - 1}$ and $A_i \in \{\emptyset, a, b, ab\}$ for $i = 1, \dots, n - 1$. In order to keep track of the terms in the summation over $\b A$, we introduce the counting functions
\begin{equation*}
r_1(\b A) \;\deq\; \sum_{i = 1}^{n - 1} \pb{\ind{A_i = a} + \ind{A_i = b}}\,, \qquad
r_2(\b A) \;\deq\; \sum_{i = 1}^{n - 1} \ind{A_i = \emptyset}\,.
\end{equation*}
We partition the sum in \eqref{sum over A} as
\begin{equation} \label{split sum over A}
\sum_{\b A} \;=\; \sum_{\b A} \indb{r_2(\b A) = 0} \indb{r_1(\b A) = 0} + \sum_{\b A} \indb{r_2(\b A) = 0} \indb{r_1(\b A) = 1} + \sum_{\b A} \indb{r_2(\b A) \geq 1 \text{ or } r_1(\b A) \geq 2}\,.
\end{equation}
Let us concentrate on the first summand; its condition is equivalent to $A_i = ab$ for all $i$. Using $\E_a \cal R_{\b v a} = 0$ and $\E_b \pb{\cal R_{b \b v} - \cal R_{b \b v}^{(a)}} = 0$ we get
\begin{equation*}
\E \, \cal R_{\b v a} \cal R_{b \b v} (\im R^{[ab]}_{\b v \b v})^{n - 1} \;=\; \E \, \pb{\cal R_{\b v a} - \cal R_{\b v a}^{(b)}} \pb{\cal R_{b \b v} - \cal R_{b \b v}^{(a)}}(\im R^{[ab]}_{\b v \b v})^{n - 1}\,.
\end{equation*}
From \eqref{cal R minor estimate}, \eqref{estimates on graded terms}, \eqref{estimate on im R vv}, and Lemma \ref{lemma: expanding estimate} we therefore get
\begin{equation*}
\absb{\E \, \cal R_{\b v a} \cal R_{b \b v} (\im R^{[ab]}_{\b v \b v})^{n - 1}} \;\leq\; \varphi^{C_\zeta} \frac{\Psi^2}{N \eta} \pb{\E (\im S_{\b v \b v})^n + (\varphi^{C_\zeta} \Phi)^n} \;\leq\; \varphi^{-1} N^{-1/2} \pb{\E (\im S_{\b v \b v})^n + (\varphi^{C_\zeta} \Phi)^n}
\end{equation*}
for large enough $C_0$.

The second summand of \eqref{split sum over A} consists of $n$ terms of the form
\begin{equation*}
\E \, \cal R_{\b v a} \cal R_{b \b v} (\im R^{[a]}_{\b v \b v}) (\im R^{[ab]}_{\b v \b v})^{n - 2} \;=\;
\E \, \cal R_{\b v a} \pb{\cal R_{b \b v} - \cal R_{b \b v}^{(a)}} (\im R^{[a]}_{\b v \b v}) (\im R^{[ab]}_{\b v \b v})^{n - 2}\,.
\end{equation*}
Recalling \eqref{estimates on graded terms}, we estimate this as above by
\begin{equation*}
\varphi^{C_\zeta} \frac{\Psi}{(N \eta)^2} \pb{\E (\im S_{\b v \b v})^n + (\varphi^{C_\zeta} \Phi)^n} \;\leq\; \varphi^{-1} N^{-1/2} \pb{\E (\im S_{\b v \b v})^n + (\varphi^{C_\zeta} \Phi)^n}
\end{equation*}
for large enough $C_0$.

What remains is to estimate the third summand in \eqref{split sum over A}. From \eqref{estimates on graded terms} and \eqref{cal R minor estimate} we get
\begin{multline*}
\sum_{\b A} \indb{r_2(\b A) \geq 1 \text{ or } r_1(\b A) \geq 2} \absb{\cal R_{\b v a} \cal R_{b \b v}} \prod_{i = 1}^{n - 1} \absb{\im R_{\b v \b v}^{[A_i]}}
\\
\leq\; \varphi^{C_\zeta} \pBB{\frac{1}{(N \eta)^3} + \frac{\Psi}{(N \eta)^2}} \pb{\im S_{\b v \b v} + \varphi^{C_\zeta}\Phi}^n \;\leq\;
\varphi^{-1} N^{-1/2} \pb{(\im S_{\b v \b v})^n + (\varphi^{C_\zeta}\Phi)^n}
\end{multline*}
\hp{2 \zeta}. This completes the proof of \eqref{claim for case 4} for $X_3 = \cal R_{\b v a} \cal R_{b \b v}$.
\end{proof}

\subsection{Estimate of $G_{\b v \b v} - \msc$}
We now conclude the proof of Proposition \ref{prop: two moment case}. By polarization and linearity, it is enough to prove the following result.

\begin{lemma} \label{lemma: two moment error}
Fix $\zeta > 0$. Then there are constants $C_0$ and $C_\zeta$, both depending on $\zeta$, such that the following holds. Assume that $z \in \b S(C_\zeta)$ satisfies \eqref{main assumption on Psi} with constant $C_0$. Then we have, for all $n \leq \varphi^{\zeta}$ and all deterministic and normalized $\b v \in \C^N$, that
\begin{equation} \label{two moment claim}
\E \absb{G_{\b v \b v}(z) - \msc(z)}^n \;\leq\; \pb{\varphi^{C_\zeta} \Psi(z)}^n\,.
\end{equation}
\end{lemma}

\begin{proof}
As in the proof of Lemma \ref{lemma: three moment error}, we focus on $\re G_{\b v \b v} - \re \msc$.
Assume without loss of generality that $n$ is even.
We shall prove that
\begin{equation} \label{weak bound for 2 moments error}
\absB{\E \pB{B_{m,3} \pb{\re R_{\b v \b v} - \re \msc}^{n - m}}}
\;\leq\; \frac{1}{\log N} \pbb{\wt {\cal E}_{ab} + N^{-3/2} \varphi^{C_1} \Psi}
\qBB{\E \pb{\re S_{\b v \b v} - \re \msc}^n + (\varphi^{C_\zeta} \Psi)^n}
\end{equation}
for $m = 1,2,3$ as well as, assuming
\begin{equation} \label{smallness on va and vb 2}
\abs{v_a} + \abs{v_b} \;\leq\; N^{-1/4} \varphi^{C_1/2} \sqrt{\Psi}\,,
\end{equation}
that
\begin{equation} \label{strong bound for 2 moments error}
\absB{\E \pB{B_{m,3} \pb{\re R_{\b v \b v} - \re \msc}^{n - m}}}
\;\leq\; \frac{1}{\log N} \wt {\cal E}_{ab}
\qBB{\E \pb{\re S_{\b v \b v} - \re \msc}^n + (\varphi^{C_\zeta} \Psi)^n}
\end{equation}
for $m = 1,2,3$. Here $C_1$ is a large enough constant depending on $\zeta$.

Assuming that \eqref{weak bound for 2 moments error} and \eqref{strong bound for 2 moments error} have been proved, we get the claim \eqref{two moment claim} from \eqref{main bound for error} and Lemma \ref{lemma: size estimate for 2 moments} applied to $S$; the detains are identical to those of the proof of Lemma \ref{lemma: final size for 2 moments} and the argument following \eqref{def Eab}.

The proof of \eqref{weak bound for 2 moments error} and \eqref{strong bound for 2 moments error} is similar to the proof of \eqref{weak bound for 2 moments} and \eqref{strong bound for 2 moments}. The key input is the apriori bound
\begin{equation} \label{apriori bound for Phi}
\im S_{\b v \b v} \;\leq\; \varphi^{C_\zeta} \Phi
\end{equation}
\hp{2 \zeta}, which follows from \eqref{rough bound for 2 moments} and Markov's inequality. Throughout the proof, we shall consistently (and without further mention) make use of the inequality
\begin{equation*}
\Psi^m \absb{\re R_{\b v \b v} - \re \msc}^{n - m} \;\leq\; \varphi^{-D} \pB{\absb{\re S_{\b v \b v} - \re \msc}^n + (\varphi^{C_{\zeta, D}} \Psi)^n}\,,
\end{equation*}
which follows from the elementary inequality $x^{m} y^{n - m} \leq x^n + y^n$ for $x,y \geq 0$, Lemma \ref{lemma: expanding estimate}, and the estimate
\begin{equation*}
\abs{R_{\b v \b v} - S_{\b v \b v}} \;\leq\; \varphi^{C_\zeta} \Psi
\end{equation*}
\hp{2 \zeta} (as follows from \eqref{Svv Rvv}). Moreover, as in \eqref{stronger bound on Rva}, we find that \eqref{smallness on va and vb 2} implies
\begin{equation} \label{stronger bound on Rva 2}
\abs{R_{\b v a}} \;\leq\; \varphi^{C_\zeta} \Psi\,.
\end{equation}
The same bound holds for $R_{a \b v}$, $R_{\b v b}$, and $R_{b \b v}$.

As in the proof of Lemma \ref{lemma: size estimate for 2 moments}, we consider four cases.

\textbf{Case (i): $a = b$ and $m \leq 3$.} This is easily dealt with using \eqref{estimate of S S for error}; we omit further details.

\textbf{Case (ii): $a \neq b$ and $m = 3$.} Recall that in this case we have $t = s = 3$. From \eqref{estimate for mst3} we get
\begin{multline*}
\pB{\abs{R_{\b v a}} + \abs{R_{a \b v}}}^3
\pB{\abs{R_{\b v b}} + \abs{R_{b \b v}}}^3
\\
\leq\; \varphi^{C_\zeta} \pBB{\frac{\im S_{\b v \b v}}{N \eta} + \Psi^2}^{3/2} \qBB{\pBB{\frac{\im S_{\b v \b v}}{N \eta} + \Psi^2}^{3/2} + \abs{v_a}^2 + \abs{v_b}^2 + \Psi^{-3} \abs{v_a}^2 \abs{v_b}^2}
\end{multline*}
\hp{2 \zeta}.
Therefore using \eqref{apriori bound for Phi}, \eqref{definition of Phi and Psi}, and $\Psi \geq c N^{-1/2}$ we get
\begin{align*}
\abs{B_{3,3}} \absb{\re R_{\b v \b v} - \re \msc}^{n - 3} &\;\leq\; \varphi^{C_\zeta} N^{-3/2} \Psi^3 \qB{\Psi^3 + \abs{v_a}^2 + \abs{v_b}^2 + N^{3/2} \abs{v_a}^2 \abs{v_b}^2} \, \absb{\re R_{\b v \b v} - \re \msc}^{n - 3}
\\
&\;\leq\; \varphi^{C_\zeta - D} \pB{\pb{\re R_{\b v \b v} - \re \msc}^n + (\varphi^{3D} \Psi)^n} \wt {\cal E}_{ab}
\end{align*}
\hp{2 \zeta}, where in the last step we used \eqref{main assumption on Psi}. Choosing $D$ large enough yields \eqref{strong bound for 2 moments error}, and hence also \eqref{weak bound for 2 moments error}.

\textbf{Case (iii): $a \neq b$ and $m = 2$.} In the case $s = t = 2$, the estimate is similar to the estimate of $X_1$ in \eqref{definition X1}. Using \eqref{apriori bound for Phi}, \eqref{definition of Phi and Psi}, and $\Psi \geq c N^{-1/2}$ we get
\begin{equation*}
\abs{X_1} \;\leq\; \varphi^{C_\zeta} \Psi \, \Psi^2 \pb{\Psi^2 + \abs{v_a}^2 + \abs{v_b}^2 + N \abs{v_a}^2 \abs{v_b}^2}
\end{equation*}
\hp{2 \zeta},
from which \eqref{strong bound for 2 moments error}, and hence also \eqref{weak bound for 2 moments error}, easily follows.

Next, consider the case $s = 3$ and $t = 1$. In order to prove \eqref{weak bound for 2 moments error}, we estimate using \eqref{apriori bound for Phi} and \eqref{Rva estimate}, similarly to \eqref{proof of rough bound},
\begin{align*}
\abs{B_{2,3,3,1}} &\;\leq\; \varphi^{C_\zeta} N^{-3/2} \pb{\Psi + \abs{v_a}}^3 \pb{\Psi + \abs{v_b}}
\\
&\;\leq\; \varphi^{C_\zeta} \Psi^2 \qbb{N^{-3/2} \frac{\Phi}{N \eta} + N^{-3/2} \abs{v_b} + N^{-1} \abs{v_a}^2 + N^{-1/2} \abs{v_a}^2 \abs{v_b}}
\end{align*}
\hp{2 \zeta}
from which \eqref{weak bound for 2 moments error} follows. Let us therefore prove \eqref{strong bound for 2 moments error}, assuming \eqref{smallness on va and vb 2}. Using \eqref{stronger bound on Rva 2} and \eqref{apriori bound for Phi}, we find
\begin{equation}
\abs{R_{\b v a}} + \abs{R_{a \b v}} + \abs{R_{\b v b}} + \abs{R_{b \b v}}\;\leq\; \varphi^{C_\zeta} \Psi
\end{equation}
\hp{2 \zeta}. We need to prove that
\begin{equation} \label{Rva bound for error}
N^{-3/2} \absB{\E R_{\b v a} R_{bb} R_{a\b v} R_{a \b v} R_{b \b v} \pb{\re R_{\b v \b v} - \re \msc}^{n - 2}} \;\leq\; \varphi^{-1}
\wt {\cal E}_{ab}
\qBB{\E \pb{\re S_{\b v \b v} - \re \msc}^n + (\varphi^{C_\zeta} \Psi)^n}\,.
\end{equation}
As for \eqref{strong estimate step 2}, by splitting $R_{bb} = (R_{bb} - \msc) + \msc$ and using \eqref{Rij estimate}, we find that it is enough to prove
\begin{equation} \label{2 moment error step 1}
N^{-3/2} \absB{\E R_{\b v a} R_{a\b v} R_{a \b v} R_{b \b v} \pb{\re R_{\b v \b v} - \re \msc}^{n - 2}} \;\leq\; \varphi^{-1}
\wt {\cal E}_{ab}
\qBB{\E \pb{\re S_{\b v \b v} - \re \msc}^n + (\varphi^{C_\zeta} \Psi)^n}\,.
\end{equation}
As for \eqref{strong estimate step 2}, we use the splitting \eqref{splitting of R b v}. Using \eqref{Rij estimate}, \eqref{apriori bound for Phi}, and \eqref{lower bound on eta}, we find that the bounds
\begin{equation}
\abs{\cal R_{b \b v}} \;\leq\; \varphi^{C_\zeta} \Psi \;\leq\; \varphi^{C_\zeta} N^{-1/6} \,, \qquad
\abs{\cal R_{b \b v}'} \;\leq\; \varphi^{C_\zeta} \pbb{\frac{N^{-1/4}}{\sqrt{N \eta}} + \Psi^2} \;\leq\; \varphi^{C_\zeta} N^{-1/3}
\end{equation}
hold \hp{2 \zeta}. Thus we get \eqref{2 moment error step 1} with $R_{b \b v}$ replaced with $\cal R_{b \b v}'$. The remaining term with $\cal R_{b \b v}$ is estimated exactly as \eqref{strong estimate step 3}; we omit the details.

\textbf{Case (iv): $a \neq b$ and $m = 1$.} In order to prove \eqref{weak bound for 2 moments error}, we use \eqref{apriori bound for Phi} to get
\begin{equation*}
\abs{B_{1,3}}
\;\leq\; \varphi^{C_\zeta} \Psi \pB{\Psi + \abs{v_a} + \abs{v_b} + \Psi^{-1} \abs{v_a} \abs{v_b} + \Psi^{-1} \abs{v_a}^2 + \Psi^{-1} \abs{v_b}^2}
\end{equation*}
\hp{2 \zeta}, from which \eqref{weak bound for 2 moments error} easily follows using $\Psi \geq N^{-1/2}$.

As for \eqref{claim for case 4}, in order to prove \eqref{weak bound for 2 moments error} and \eqref{strong bound for 2 moments error} it suffices to prove the following claim. For $X_3$ being any expression in \eqref{category 1} -- \eqref{category 3}, we have
\begin{equation} \label{weak bound case 4}
N^{-3/2} \absb{\E \, X_3  (\re R_{\b v \b v} - \re \msc)^{n - 1}} \;\leq\; \varphi^{-1} \pbb{\wt {\cal E}_{ab} + N^{-3/2} \varphi^{C_\zeta} \Psi} \pB{\E \pb{\re S_{\b v \b v} - \re \msc}^n + (\varphi^{C_\zeta} \Psi)^n}\,,
\end{equation}
as well as, assuming \eqref{smallness on va and vb 2},
\begin{equation} \label{strong bound case 4}
N^{-3/2} \absb{\E \, X_3  (\re R_{\b v \b v} - \re \msc)^{n - 1}} \;\leq\; \varphi^{-1} \wt{\cal E}_{ab} \pB{\E \pb{\re S_{\b v \b v} - \re \msc}^n + (\varphi^{C_\zeta} \Psi)^n}\,.
\end{equation}

Note that from \eqref{def R prime R double prime} and \eqref{apriori bound for Phi} we get that
\begin{equation} \label{general bound on R double prime}
\abs{\cal R_{b \b v}'} \;\leq\; C \abs{v_b} + \varphi^{C_\zeta} \Psi^2\,.
\end{equation}
If $X_3$ is any expression in \eqref{category 1}, we get from Lemma \ref{lemma: Gvi Gvv}, \eqref{Rij estimate}, \eqref{apriori bound for Phi}, and \eqref{general bound on R double prime} that
\begin{equation*}
\abs{X_3} \;\leq\; \varphi^{C_\zeta} \Psi^2 \pb{\Psi + \abs{v_a}} \pb{\Psi + \abs{v_b}} + \varphi^{C_\zeta} \pb{\Psi^2 + \abs{v_a}} \pb{\Psi^2 + \abs{v_b}}
\end{equation*}
\hp{2 \zeta}. Now \eqref{strong bound case 4}, and in particular \eqref{weak bound case 4}, follows easily (note that we did not assume \eqref{smallness on va and vb 2}).

Next, let $X_3$ be an expression in \eqref{category 2}. From Lemma \ref{lemma: Gvi Gvv}, \eqref{Rij estimate}, \eqref{apriori bound for Phi}, and \eqref{general bound on R double prime} we get
\begin{equation*}
\abs{X_3} \;\leq\; \varphi^{C_\zeta} \Psi \pb{\Psi + \abs{v_a}} \pb{\Psi + \abs{v_a} + \abs{v_b}} + \varphi^{C_\zeta} \pb{\Psi^2 + \abs{v_a}} \pb{\Psi + \abs{v_b}} + \varphi^{C_\zeta} \pb{\Psi + \abs{v_a}} \pb{\Psi^2 + \abs{v_b}}
\end{equation*}
\hp{2 \zeta}. Now \eqref{weak bound case 4} follows easily. Moreover, \eqref{strong bound case 4} under the assumption \eqref{smallness on va and vb 2} follows exactly like in paragraphs of \eqref{def R and cal R} and \eqref{telescopic for second category}, using the bound $\absb{\pb{\cal R'_{b \b v}}^{(a)} - \cal R'_{b \b v}} \leq \varphi^{C_\zeta} \Psi^3$ \hp{2 \zeta}, as follows from \eqref{estimate on R double prime} and \eqref{apriori bound for Phi}.

Finally, we consider the case \eqref{category 3}, i.e.\ $X_3 = \cal R_{\b v a} \cal R_{b \b v}$. Under the assumption \eqref{smallness on va and vb 2}, we find from \eqref{apriori bound for Phi}, \eqref{estimates on graded terms}, and \eqref{cal R minor estimate},
\begin{equation*}
\absb{\cal R_{b \b v}^{(a)} - \cal R_{b \b v}} \;\leq\; \varphi^{C_\zeta} \Psi^2 \,, \qquad
\absb{R_{\b v \b v}^{[a]}} + \absb{R_{\b v \b v}^{[b]}} \;\leq\; \varphi^{C_\zeta} \Psi^2\,, \qquad
\absb{R_{\b v \b v}^{[\emptyset]}} \;\leq\; \varphi^{C_\zeta} \Psi^3
\end{equation*}
\hp{2 \zeta}. Then the argument from the proof of Lemma \ref{lemma: final size for 2 moments} can be applied almost unchanged, and we get \eqref{strong bound case 4} assuming \eqref{smallness on va and vb 2}.
\end{proof}

\section{Proof of Theorems \ref{theorem: strong estimate} and \ref{theorem: delocalization}} \label{section: deloc and exterior}

By Lemma \ref{lemma: msc}, if $\eta \leq \kappa$ and $\abs{E} > 2$ then the control parameter on the right-hand side of \eqref{strong estimate} can also be expressed as
\begin{equation} \label{control parameter outside}
\sqrt{\frac{\im \msc(z)}{N \eta}} \;\asymp\; N^{-1/2} \kappa^{-1/4}\,,
\end{equation}
where $\kappa \equiv \kappa_E$ was defined in \eqref{def kappa}.

\begin{proof}[Proof of Theorem \ref{theorem: strong estimate}]
By polarization and linearity, it is enough to prove that
\begin{equation} \label{strong estimate 2}
\absb{G_{\b v \b v}(z) - \msc(z) } \;\leq\; \varphi^{C_\zeta} \sqrt{\frac{\im \msc(z)}{N \eta}}
\end{equation}
\hp{\zeta}, for all normalized $\b v$. Moreover, by symmetry it suffices to consider the case $2 + \varphi^{C_1} N^{-2/3} \leq E \leq \Sigma$. In particular, $\kappa \geq \varphi^{C_1} N^{-2/3}$. Using Lemma \ref{lemma: msc} we find that Theorem \ref{thm: ILSC} implies \eqref{strong estimate 2} if $\eta \geq \eta_0$, where we defined
\begin{equation*}
\eta_0 \;\deq\; N^{-1/2} \kappa^{1/4}\,.
\end{equation*}
Note that $\eta_0 \leq \kappa$.

It remains therefore to establish \eqref{strong estimate 2} when $0 \leq \eta \leq \eta_0$. Define
\begin{equation*}
z \;\deq\; E + \ii \eta\,, \qquad z_0 \;\deq\; E + \ii \eta_0\,.
\end{equation*}
By \eqref{control parameter outside} and \eqref{strong estimate 2} at $z_0$, it is enough to prove that
\begin{equation} \label{z z0 comparison 1}
\absb{\msc(z) - \msc(z_0)} \;\leq\; C N^{-1/2} \kappa^{-1/4}
\end{equation}
and
\begin{equation} \label{z z0 comparison 2}
\absb{G_{\b v \b v}(z) - G_{\b v \b v}(z_0)} \;\leq\; \varphi^{C_\zeta} N^{-1/2} \kappa^{-1/4}
\end{equation}
\hp{\zeta}.

Differentiating \eqref{identity for msc}, we find
\begin{equation} \label{diff msc}
\msc' \;=\; \frac{\msc^2}{1 - \msc^2}\,,
\end{equation}
which, by Lemma \ref{lemma: msc}, implies that $\msc' \asymp (\kappa + \eta)^{-1/2} = O(\kappa^{-1/2})$. Therefore we get
\begin{equation*}
\absb{\msc(z) - \msc(z_0)} \;\leq\; C \kappa^{-1/2} \eta_0 \;=\; C N^{-1/2} \kappa^{-1/4}\,,
\end{equation*}
which is \eqref{z z0 comparison 1}.

Next, by Theorem \ref{theorem: rigidity} we have $E \geq \lambda_N + \eta_0$ \hp{\zeta} provided $C_1$ is large enough. Therefore, since $\eta \leq \eta_0 \leq E - \lambda_N \leq E - \lambda_\alpha$ \hp{\zeta} for all $\alpha \leq N$, we get
\begin{equation} \label{strong 4}
\im G_{\b v \b v}(z) \;=\; \sum_\alpha \frac{\abs{\scalar{\b v}{\b u^{(\alpha)}}}^2 \eta}{(E - \lambda_\alpha)^2 + \eta^2} \;\leq\;
2 \sum_\alpha \frac{\abs{\scalar{\b v}{\b u^{(\alpha)}}}^2 \eta_0}{(E - \lambda_\alpha)^2 + \eta_0^2} \;=\; 2 \im G_{\b v \b v}(z_0) \;\leq\; \varphi^{C_\zeta} N^{-1/2} \kappa^{-1/4}
\end{equation}
\hp{\zeta}, by \eqref{strong estimate 2} at $z_0$ and the estimate $\im \msc(z_0) \leq C N^{-1/2} \kappa^{-1/4}$. Finally, we estimate the real part from
\begin{multline} \label{strong 5}
\absb{\re G_{\b v \b v}(z) - \re G_{\b v \b v}(z_0)} \;=\; \sum_\alpha \frac{(E - \lambda_\alpha)(\eta_0^2 - \eta^2) \abs{\scalar{\b u^{(\alpha)}}{\b v}}^2}{\pb{(E - \lambda_\alpha)^2 + \eta^2} \pb{(E - \lambda_\alpha)^2 + \eta^2_0}}
\\
\leq\; \frac{\eta_0}{E - \lambda_N} \sum_\alpha \frac{\eta_0 \abs{\scalar{\b u^{(\alpha)}}{\b v}}^2}{(E - \lambda_\alpha)^2 + \eta^2_0} \;\leq\; \im G_{\b v \b v}(z_0)
\end{multline}
\hp{\zeta}, where in the last step we used that $\eta_0 \leq E - \lambda_N$. Combining \eqref{strong 4} and \eqref{strong 5} completes the proof of \eqref{z z0 comparison 2}.
\end{proof}

\begin{proof}[Proof of Theorem \ref{theorem: delocalization}]
We begin with \eqref{three moment delocalization}, whose proof is immediate. Using Theorem \ref{thm: ILSC} with Condition {\bf A} and Remark \ref{remark: lattice}, we find
\begin{equation*}
C \;\geq\; \im G_{\b v \b v}(\lambda_\alpha + \ii \eta) \;=\; \sum_{\beta} \frac{\eta \abs{\scalar{\b u^{(\beta)}}{\b v}}^2}{(\lambda_\alpha - \lambda_\beta)^2 + \eta^2} \;\geq\; \eta^{-1} \abs{\scalar{\b u^{(\alpha)}}{\b v}}^2
\end{equation*}
\hp{\zeta}, where we used Theorem \ref{theorem: rigidity} to ensure that $\lambda_\alpha \in [-\Sigma, \Sigma]$ \hp{\zeta}. Choosing $\eta = \varphi^{\zeta} N^{-1}$ yields \eqref{three moment delocalization}.

In order to prove \eqref{two moment delocalization}, we set
\begin{equation*}
\eta \;\deq\; \gamma_b - \gamma_a \,, \qquad E \;\deq\; \gamma_a\,,
\end{equation*}
where $\gamma_\alpha$ is the classical location of the $\alpha$-th eigenvalue defined in \eqref{classical location}. Then we get
\begin{equation} \label{first step in delocalization}
\sum_{\alpha = a}^b \abs{\scalar{\b u^{(\alpha)}}{\b v}}^2 \;\leq\; \varphi^{C_\zeta} \sum_{\alpha = a}^b \frac{\eta^2 \abs{\scalar{\b u^{(\alpha)}}{\b v}}^2}{(\lambda_\alpha - E)^2 + \eta^2} \;\leq\; \varphi^{C_\zeta} \eta \im G_{\b v \b v}(E + \ii \eta)\,,
\end{equation}
where in the first step we used Theorem \ref{theorem: rigidity} to conclude that $(\lambda_\alpha - E)^2 \leq \varphi^{C_\zeta} \eta^2$ for $a \leq \alpha \leq b$. In order to invoke Theorem \ref{thm: ILSC} with Condition {\bf B}, we have to satisfy \eqref{main assumption on Psi}. Recalling Lemma \ref{lemma: msc}, we find that \eqref{main assumption on Psi} holds provided that
\begin{equation} \label{condition for delocalization}
\eta \;\geq\; \varphi^{C_0} N^{-5/6}\,, \qquad \kappa \;\leq\; \varphi^{- 2 C_0} \eta^2 N^{4/3}\,,
\end{equation}
where we abbreviated $\kappa \equiv \kappa_E$. From \eqref{classical location} we get
\begin{equation} \label{position of gamma alpha}
\gamma_\alpha + 2 \;\asymp\; \alpha^{2/3} N^{-2/3}
\end{equation}
for $\alpha \leq N/2$,
from which we deduce, recalling $E = \gamma_\alpha$,
\begin{equation*}
\kappa \;\asymp\; a^{2/3} N^{-2/3} \,, \qquad \eta \;\asymp\; (b^{2/3} - a^{2/3}) N^{-2/3}\,.
\end{equation*}
Hence \eqref{condition for delocalization} is satisfies provided that
\begin{equation*}
b^{2/3} - a^{2/3} \;\geq\; \varphi^{C_0} N^{-1/6} + \varphi^{C_0} a^{1/3} N^{-1/3}\,.
\end{equation*}
Since $b^{2/3} - a^{2/3} \geq b^{-1/3}(b - a) / 2$, we find that \eqref{condition for delocalization}, and hence \eqref{main assumption on Psi}, holds under the condition \eqref{main condition for delocalization}.

Therefore we may apply Theorem \ref{thm: ILSC} to the right-hand side of \eqref{first step in delocalization} to get
\begin{equation*}
\sum_{\alpha = a}^b \abs{\scalar{\b u^{(\alpha)}}{\b v}}^2  \;\leq\; \varphi^{C_\zeta} \eta \pbb{\frac{1}{N \eta} + \im \msc(E + \ii \eta)} \;\leq\; \varphi^{C_\zeta} N^{-1} \pB{(b^{2/3} - a^{2/3})^{3/2} + a^{1/3} (b^{2/3} - a^{2/3})}
\end{equation*}
\hp{\zeta},
where we used Lemma \ref{lemma: msc}. The claim now follows from the elementary inequalities
\begin{equation*}
b^{2/3} - a^{2/3} \;\leq\; (b - a)^{2/3} \,, \qquad b^{2/3} - a^{2/3} \;\leq\; a^{-1/3} (b - a)\,. \qedhere
\end{equation*}
\end{proof}

For future use, we record the following consequence of Theorem \ref{theorem: delocalization} which is useful in combination with dyadic decompositions. For any integer $K \leq N/4$ we have
\begin{equation} \label{dyadic delocalization}
\sum_{\alpha = K}^{2 K} \abs{\scalar{\b u^{(\alpha)}}{\b v}}^2 \;\leq\; \varphi^{C_\zeta} K N^{-1}
\end{equation}
\hp{\zeta}.

\section{Eigenvalue locations: proof of Theorem \ref{theorem: rank-k lde}} \label{section: eigenvalue locations}

\subsection{Basic facts from linear algebra} \label{sect: basic linalg}
We begin by collecting a few well-known tools from linear algebra, on which our analysis of the deformed spectrum relies.

We use the following representation of the eigenvalues of $\wt H$, which was already used in several papers on finite-rank deformations of random matrices \cite{SoshPert, BGGM1, BGGM2, BGN}.

\begin{lemma} \label{lemma: det identity}
If $\mu \in \R \setminus \sigma(H)$ and $\det(D) \neq 0$ then $\mu \in \sigma(\wt H)$ if and only if
\begin{equation*}
\det \pb{V^* G(\mu) V + D^{-1}}\;=\;0\,.
\end{equation*}
\end{lemma}
\begin{proof}
For the convenience of the reader, we give the simple proof. The claim follows from the computation
\begin{align*}
\det \pb{\wt H - \mu} &\;=\; \det(H - \mu) \det \pb{\umat+ (H - \mu)^{-1}V D V^*}
\\
&\;=\; \det(H - \mu)  \det \pb{\umat+ V^* (H - \mu)^{-1}V D}
\\
&\;=\; \det(H - \mu)  \det(D)  \det \pb{D^{-1}+ V^* (H - \mu)^{-1}V}\,,
\end{align*}
where in the second step we used the identity
$\det(\umat + AB) = \det(\umat + BA)$
which is valid for any $n \times m$ matrix $A$ and $m \times n$ matrix $B$.
\end{proof}

We shall also make use of the well-known Weyl's interlacing property, summarized in the following lemma.

\begin{lemma} \label{lemma: interlacing}
If $A$ is an $N \times N$ Hermitian matrix and $B = A + d \, \b v \b v^*$ with some $d > 0$ and $\b v \in \C^N$, then the eigenvalues of $A$ and $B$ are interlaced:
\begin{equation*}
\lambda_1(A) \;\leq\; \lambda_1(B) \;\leq\; \lambda_2(A) \;\leq\; \cdots \;\leq\; \lambda_{N - 1}(B) \;\leq\; \lambda_{N}(A) \;\leq\; \lambda_N(B)\,.
\end{equation*}
\end{lemma}

We shall occasionally need the eigenvalues of $H$ to be distinct. To that end, we assume without loss of generality that the law of $H$ is absolutely continuous; otherwise consider the matrix $H + \me^{-N} V$ where $V$ is a GOE/GUE matrix independent of $H$. It is immediate that this perturbation does not change any of $H$'s spectral statistics. Moreover, any Hermitian matrix with an absolutely continuous law has almost surely distinct eigenvalues.

\subsection{Warmup: the rank-one case}
In order to illustrate our method, we first present a much simplified proof which deals with the case $k = 1$.  Let $\b v \in \C^N$ be normalized and deterministic, and $d \in \R$ be deterministic (and possibly $N$-dependent). Define the deformed matrix
\begin{equation*}
\wt H \;\deq\; H + d \, \b v \b v^*\,.
\end{equation*}
For the following we note the elementary estimate
\begin{equation} \label{kappa sim d - 1}
\theta(d) - 2 \;\asymp\; (d - 1)^2\,,
\end{equation}
as follows from \eqref{def theta}.

\begin{theorem}\label{theorem: rank-one lde}
Fix $\zeta > 0$.
Then there is a constant $C_\zeta$ such that the following holds. For $0 \leq d \leq 1$ we have
\begin{equation*}
0 \;\leq\; \mu_N - \lambda_N \;\leq\; \varphi^{C_\zeta} \frac{d}{N (1 - d + N^{-1/3})}
\end{equation*}
\hp{\zeta}. For $1 \leq d \leq \Sigma - 1$ we have
\begin{equation*}
\abs{\mu_N - \theta(d)} \;\leq\; \varphi^{C_\zeta} \sqrt{\frac{d - 1 + N^{-1/3}}{N}}
\end{equation*}
\hp{\zeta}.

By symmetry, an analogous result holds for $d \leq 0$.
\end{theorem}

\begin{proof}
First we note that it is enough to consider $d \in \R_+ \setminus \qb{1 - \varphi^D N^{-1/3}, 1 + \varphi^D N^{-1/3}}$ for some arbitrary but fixed $D>0$. This follows from $\abs{\lambda_N - 2} \leq \varphi^{C_\zeta} N^{-2/3}$ \hp{\zeta} (see Theorem \ref{theorem: rigidity}), the monotonicity of the map $d \mapsto \lambda_N(H + d \, \b v \b v^*)$ (see Lemma \ref{lemma: interlacing}), and the observation that $\theta(1 + \epsilon) = 1 + \epsilon^2 + O(\epsilon^3)$ as $\epsilon \to 0$ (which implies that $\abs{\theta(d) - 2} \leq \varphi^{2D + 1} N^{-2/3}$ for $d \in \qb{1 - \varphi^D N^{-1/3}, 1 + \varphi^D N^{-1/3}}$).

The key identity\footnote{Here we ignore the possibility that $\mu_N \in \sigma(H)$. Since the law of $H$ is absolutely continuous, it is easy to check that the interlacing inequalities in Lemma \ref{lemma: interlacing} are strict with probability one; see e.g.\ the proof of Lemma \ref{lemma: limit rank one}.} for the proof is
\begin{equation*}
G_{\b v \b v}(\mu_N) \;=\; - \frac{1}{d}\,,
\end{equation*}
as follows from Lemma \ref{lemma: det identity}. Let us begin with the case $d \geq 1 + \varphi^D N^{-1/3}$. Since $\msc : \R \setminus (-2,2) \to [-1,1] \setminus \{0\}$ is bijective, we find from \eqref{identity for msc} that $\theta(d)$ is uniquely characterized by
\begin{equation} \label{identity for gamma}
\msc(\theta(d)) \;=\; - \frac{1}{d}\,.
\end{equation}
We therefore have to solve the equation
$\msc(\theta(d)) = G_{\b v \b v}(x)$
for $x \in [2 + \varphi^{C_1} N^{-2/3}, \infty)$, where $C_1$ the constant from Theorem \ref{theorem: strong estimate}. By Theorem \ref{theorem: strong estimate}, we have
\begin{equation} \label{Gvv msc for outliers}
G_{\b v \b v}(x) \;=\; \msc(x) + O\pb{\varphi^{C_\zeta} N^{-1/2} \kappa_x^{-1/4}}
\end{equation}
\hp{\zeta}.

Next, define the interval
\begin{equation*}
I_d \;\deq\; [x_-(d), x_+(d)] \,, \qquad x_\pm(d) \;\deq\; \theta(d) \pm \varphi^D N^{-1/2} (d - 1)^{1/2} \,.
\end{equation*}
We claim that
\begin{equation} \label{derivative of m estimated}
\kappa_x \;\asymp\; (d - 1)^2\,, \qquad m'(x) \;\asymp\; (d - 1)^{-1} \qquad (x \in I_d) \qquad
\end{equation}
The first relation of \eqref{derivative of m estimated} follows from
\begin{equation*}
\abs{x - \theta(d)} \;\leq\; \varphi^D N^{-1/2} (d - 1)^{1/2} \qquad \text{ and } \qquad  \theta(d) - 2 \;\geq\; c (d - 1)^2 \;\geq\; c \varphi^{3 D/2} N^{-1/2} (d - 1)^{1/2}\,,
\end{equation*}
where in the last step we used $d \geq 1 + \varphi^D N^{-1/3}$.
In order to prove the second relation of \eqref{derivative of m estimated}, we differentiate \eqref{diff msc} and use Lemma \ref{lemma: msc} to get
\begin{equation} \label{msc prime}
\msc'(x) \;\asymp\; \kappa_x^{-1/2}\,, \qquad  \msc''(x) \;\asymp\; \kappa_x^{-3/2}\,.
\end{equation}
Therefore we get from \eqref{msc prime} and the mean value theorem applied to $\msc'$ that
\begin{equation*}
\abs{\msc'(x) - \msc'(\theta(d))} \;\leq\; C \varphi^D N^{-1/2} (d - 1)^{1/2} (d - 1)^{-3} \;\leq\; C \varphi^{-D/2} (d - 1)^{-1} \,.
\end{equation*}
Therefore \eqref{derivative of m estimated} follows from $\msc'(\theta(d)) \asymp (d - 1)^{-1}$.

Now choose $D$ large enough that $x_-(d) \geq 2 + \varphi^{C_1} N^{-2/3}$ for $d \geq \varphi^D N^{-2/3}$. Thus \eqref{Gvv msc for outliers} and \eqref{derivative of m estimated} yield
\begin{equation} \label{exclusion for rank-one}
G_{\b v \b v}(x_-(d)) \;<\; \msc(\theta(d)) \;<\; G_{\b v \b v}(x_+(d))
\end{equation}
\hp{\zeta}, provided $D$ is chosen larger than the constant $C_\zeta$ in \eqref{Gvv msc for outliers}. Finally we observe that, by Theorem \ref{theorem: rigidity}, \hp{\zeta} the function $x \mapsto G_{\b v \b v}(x)$ is continuous and increasing on $[2 + \varphi^{C_1} N^{-2/3},\infty)$. It follows that \hp{\zeta} the equation $G_{\b v \b v}(x) = \msc(\theta(d))$ has precisely one solution, $x = \mu_N$, in $[2 + \varphi^{C_1} N^{-2/3},\infty)$. Moreover, this solution lies in $I_d$, which implies that it satisfies the claim of Theorem \ref{theorem: rank-one lde} for $d > 1$.

What remains is the case $d \leq 1 - \varphi^{D} N^{-1/3}$. Choose $x \deq 2 + \varphi^{C_1}N^{-2/3}$ where $C_1$ is a large constant to be chosen later. For large enough $C_1$ we find from Theorem \ref{theorem: strong estimate}
\begin{equation} \label{Gvv simple bulk}
G_{\b v \b v}(x) \;=\; \msc(x) + O\pb{N^{-1/3} \varphi^{-C_1/4}}
\end{equation}
\hp{\zeta}. From \eqref{bounds on msc} we find
\begin{equation} \label{Gvv simple bulk 2}
1 + \msc(x) \;\asymp\; N^{-1/3} \varphi^{C_1 / 2}\,,
\end{equation}
 which yields
\begin{equation*}
1 + G_{\b v \b v}(x) \;\geq\; 0 \;\geq\; 1 - \frac{1}{d}
\end{equation*}
\hp{\zeta}. Choosing $C_1$ large enough, we find as above that $y \mapsto G_{\b v \b v}(y)$ is \hp{\zeta} increasing and continuous for $y \geq x$, from which we deduce that
\begin{equation*}
\lambda_N \;\leq\; \mu_N \;\leq\; x
\end{equation*}
\hp{\zeta}. (The first inequality follows from Lemma \ref{lemma: interlacing}.)

Next, abbreviate $q \deq \varphi^{C_2}$ for some large constant $C_2$ to be chosen later. Using Theorem \ref{theorem: rigidity} we estimate, for $\lambda_N \leq \mu_N \leq x$ and large enough $C_2$,
\begin{align*}
\absBB{\sum_{\alpha \leq N - q} \frac{\abs{\scalar{\b u^{(\alpha)}}{\b v}}^2}{\lambda_\alpha - \mu_N} - \sum_{\alpha \leq N - q} \frac{\abs{\scalar{\b u^{(\alpha)}}{\b v}}^2}{\lambda_\alpha - x}}
&\;\leq\; \varphi^{C_\zeta} N^{-2/3} \sum_{\alpha \leq N - q} \frac{\abs{\scalar{\b u^{(\alpha)}}{\b v}}^2}{(\lambda_\alpha - \mu_N)^2}
\\
&\;\leq\; \varphi^{C_\zeta} N^{-2/3} \sum_{k \geq 1} \frac{2^k N^{-1}}{(2^{2k/3} N^{-2/3})^2} + \varphi^{C_\zeta} N^{-2/3}
\\
&\;\leq\; \varphi^{C_\zeta} N^{-1/3}
\end{align*}
\hp{\zeta}. In the second inequality we estimated the contribution of the eigenvalues $\alpha \geq N/2$ using the dyadic decomposition
\begin{equation*}
U_k \;\deq\; \hb{\alpha \in [N/2\,,\, N - q] \col N - 2^{k + 1} \leq \alpha \leq N - 2^k}
\end{equation*}
combined with Theorem \ref{theorem: rigidity}, the estimate
\begin{equation*}
2 - \gamma_\alpha \;\asymp\; (N - \alpha)^{2/3} N^{-2/3} \qquad (\alpha \geq N/2)\,,
\end{equation*}
and the delocalization estimate \eqref{dyadic delocalization}. A similar (in fact easier) dyadic decomposition works for the remaining eigenvalues $\alpha < N/2$ and yields the last term of the second line. Moreover, we have
\begin{equation*}
\sum_{\alpha > N - q} \frac{\abs{\scalar{\b u^{(\alpha)}}{\b v}}^2}{\abs{\lambda_\alpha - x}} \;\leq\; \varphi^{C_\zeta + C_2} N^{-1/3}
\end{equation*}
\hp{\zeta}, by Theorems \ref{theorem: rigidity} and \ref{theorem: delocalization}. Recalling \eqref{Gvv simple bulk} and \eqref{Gvv simple bulk 2}, we have therefore proved that
\begin{equation*}
-\frac{1}{d} \;=\; G_{\b v \b v}(\mu_N)\;=\; \sum_{\alpha  } \frac{\abs{\scalar{\b u^{(\alpha)}}{\b v}}^2}{\lambda_\alpha - \mu_N}
\;=\; -1 + O\pb{\varphi^{C_\zeta + C_1 + C_2} N^{-1/3}} + \sum_{\alpha > N - q} \frac{\abs{\scalar{\b u^{(\alpha)}}{\b v}}^2}{\lambda_\alpha - \mu_N}
\end{equation*}
\hp{\zeta}. Therefore
\begin{equation*}
\frac{1}{\mu_N - \lambda_N} \sum_{\alpha > N - q} \abs{\scalar{\b u^{(\alpha)}}{\b v}}^2 \;\geq\; \sum_{\alpha > N - q} \frac{\abs{\scalar{\b u^{(\alpha)}}{\b v}}^2}{\mu_N - \lambda_\alpha} \;=\; \frac{1}{d} - 1 + O\pb{\varphi^{C_\zeta + C_1 + C_2} N^{-1/3}}
\end{equation*}
\hp{\zeta}. Theorem \ref{theorem: delocalization} implies $\abs{\scalar{\b u^{(\alpha)}}{\b v}}^2 \leq \varphi^{C_\zeta} N^{-1}$, and the claim follows. This concludes the proof of Theorem \ref{theorem: rank-one lde}.
\end{proof}

\subsection{The permissible region}
The rest of this section is devoted to the proof of Theorem \ref{theorem: rank-k lde}.

\begin{definition}
We choose an event, denoted by $\Xi$, of $\zeta$-high probability on which the following statements hold.
\begin{enumerate}
\item
The eigenvalues of $H$ are distinct.
\item
For all $i = 1, \dots, k$ and $\alpha = 1, \dots, N$ we have $\scalar{\b v^{(i)}}{\b u^{(\alpha)}} \neq 0$\,.
\item
All statements of Theorems \ref{thm: ILSC}, \ref{theorem: strong estimate}, \ref{theorem: delocalization}, and \ref{theorem: rigidity} hold.
\end{enumerate}
\end{definition}

We note that such a $\Xi$ exists. As explained in Section \ref{sect: basic linalg}, we assume without loss of generality that the law of $H$ is absolutely continuous. Then conditions (i) and (ii) hold almost surely; we omit the standard proof. That condition (iii) holds with $\zeta$-high probability is a consequence of Theorems \ref{thm: ILSC}, \ref{theorem: strong estimate}, \ref{theorem: delocalization}, and \ref{theorem: rigidity} (see also Remark \ref{remark: lattice}).

For the whole remainder of the proof of Theorem \ref{theorem: rank-k lde}, we choose and fix an arbitrary realization $H \equiv H^\omega$ with $\omega \in \Xi$. Thus, the randomness of $H$ only comes into play in ensuring that $\Xi$ is of $\zeta$-high probability. The rest of the argument is entirely deterministic.

Fix $k^-, k^+ \in \N$ and define $k^0 \deq k - k^+ - k^- = \# \h{i \col \abs{d_i} \leq 1}$. Write
\begin{equation*}
\b d \;=\; (d_1, \dots, d_k) \;=\; (\b d^-, \b d^0, \b d^+) \, \qquad \b d^\sigma \;=\; (d^\sigma_1, \dots, d^\sigma_{k^\sigma}) \qquad (\sigma \;=\; -,0,+)\,.
\end{equation*}
We adopt the convention that
\begin{equation} \label{ordering of ds}
d_1^- \;\leq\; \cdots \;\leq\; d_{k^-}^- \;<\; -1 \;\leq\; d^0_1 \;\leq\; \cdots \;\leq\; d^0_{k^0} \;\leq\; 1 < d^+_1 \;\leq\; \cdots \;\leq\; d^+_{k^+}\,.
\end{equation}
Abbreviate
\begin{equation} \label{def of psi tilde}
\wt \psi_N \;\equiv\; \wt \psi \;\deq\; 2k \psi\,.
\end{equation}
For $\wt C_2 > 0$ define the sets
\begin{align*}
\cal D^-(\wt C_2) &\;\deq\; \hB{\b d^- \col - \Sigma + 1 \leq d^-_i \leq -1 - \varphi^{\wt C_2} \wt \psi N^{-1/3} \,,\, i =1, \dots, k^-}\,,
\\
\cal D^+(\wt C_2) &\;\deq\; \hB{\b d^+ \col 1 + \varphi^{\wt C_2} \wt \psi N^{-1/3} \leq d^+_i \leq \Sigma - 1 \,,\, i =1, \dots, k^+}\,,
\\
\cal D^0(\wt C_2) &\;\deq\; \hB{\b d^0 \col -1 + \varphi^{\wt C_2} \wt \psi N^{-1/3} \leq d^0_i \leq 1 - \varphi^{\wt C_2} \wt \psi N^{-1/3} \,,\, i =1, \dots, k^0}\,,
\end{align*}
the set of allowed $\b d$'s,
\begin{equation*}
\cal D(\wt C_2) \;\deq\; \hb{(\b d^-, \b d^0, \b d^+) \col \b d^\sigma \in \cal D^\sigma(\wt C_2)\,,\, \sigma = -,0,+}\,,
\end{equation*}
and the subset
\begin{equation*}
\cal D^*(\wt C_2) \;\deq\; \hb{\b d \in \cal D(\wt C_2) \col d_i \neq 0 \text{ for } i = 1, \dots, k}\,.
\end{equation*}

Let $\wt K > 0$ denote a constant to be chosen later, and define
\begin{equation*}
S(\wt K) \;\deq\; \pB{-\infty \,,\, -2 + \varphi^{\wt K} N^{-2/3}} \cup \pB{2 - \varphi^{\wt K} N^{-2/3} \,,\, \infty}\,.
\end{equation*}
We shall only consider eigenvalues of $\wt H$ in $S(\wt K)$ for some large but fixed $\wt K$.

Let $\wt C_3 > 0$ denote some large constant to be chosen later.  Define the intervals
\begin{align*}
I_i^-(\b d) &\;\deq\; \qbb{\theta(d_i^-) - \varphi^{\wt C_3} N^{-1/2} (-d_i^- - 1)^{1/2}\,,\, \theta(d_i^-) + \varphi^{\wt C_3} N^{-1/2} (-d_i^- - 1)^{1/2}} \qquad (i = 1, \dots, k^-)\,,
\\
I_i^+(\b d) &\;\deq\; \qbb{\theta(d_i^+) - \varphi^{\wt C_3} N^{-1/2} (d_i^+ - 1)^{1/2}\,,\, \theta(d_i^+) + \varphi^{\wt C_3} N^{-1/2} (d_i^+ - 1)^{1/2}} \qquad (i = 1, \dots, k^+)\,,
\\
I^0 &\;\deq\; \hB{x \in \R \col \dist(x, \sigma(H)) \leq N^{-2/3} \wt \psi^{-1}} \cap S(\wt K)\,.
\end{align*}
For $\b d \in \cal D(\wt C_2)$ define
\begin{equation*}
\Gamma(\b d) \;\deq\; I^0 \cup \pBB{\bigcup_{i = 1}^{k^-} I^-_i(\b d)} \cup \pBB{\bigcup_{i = 1}^{k^+} I^+_i(\b d)}\,.
\end{equation*}

The following proposition states that $\Gamma(\b d)$ is the ``permissible region'' for the eigenvalues of $\wt H$. Roughly, the allowed region consists of a small neighbourhood of each $\theta(d_i)$ for $i \in O$, as well as of small neighbourhoods of the eigenvalues of $H$. The latter regions house the sticking eigenvalues. Proposition \eqref{prop: permissible region} only establishes where the eigenvalues are allowed to lie; it gives no other information on their locations (such as the number of eigenvalues in each interval). Note that, by definition of $S(\wt K)$, the set $\Gamma(\b d)$ only keeps track of eigenvalues outside of the interval $\qb{-2 + \varphi^{\wt K}N^{-2/3},2 - \varphi^{\wt K}N^{-2/3}}$. This will eventually suffice for the statement \eqref{main result: bulk evs} thanks to the eigenvalue rigidity estimate for $H$, Theorem \ref{theorem: rigidity}, combined with eigenvalue interlacing; see \eqref{rank k interlacing} below.

\begin{proposition} \label{prop: permissible region}
For $\wt C_3$ and $\wt C_2(\wt C_3)$ large enough (depending on $\zeta$, $\wt K$, and the constant $C_1$ from Theorem \ref{theorem: strong estimate}) the following holds.
For any $\b d \in \cal D(\wt C_2)$ and $H \equiv H^\omega$ with $\omega \in \Xi$ we have
\begin{equation} \label{intervals are disjoint}
I^\pm_i(\b d) \cap I^0 \;=\; \emptyset \quad \text{for all} \quad i \;=\; 1, \dots, k^\pm
\end{equation}
as well as
\begin{equation} \label{permissible region statement}
\sigma(\wt H) \cap S(\wt K) \;\subset\; \Gamma(\b d)\,.
\end{equation}
\end{proposition}

\begin{proof}
Clearly, it is enough to prove the claim for $\b d \in \cal D^*(\wt C_2)$. We shall choose the constants $\wt C_3(\zeta,C_1)$ and $\wt C_2(\zeta,\wt K,C_1,\wt C_3)$ to be large enough during the proof. (Here $C_1$ is the constant from Theorem \ref{theorem: strong estimate}.)

First we prove \eqref{intervals are disjoint}. By definition of $\Xi$ (see Theorem \ref{theorem: rigidity}), we find that \eqref{intervals are disjoint} holds if
\begin{equation*}
2 + \varphi^{2 \wt C_2} N^{-2/3} - \varphi^{\wt C_3 + \wt C_2 / 2} N^{-2/3} \;>\; 2 + 2 \varphi^{C_\zeta} N^{-2/3} \;\geq\; \lambda_N + N^{-2/3} \wt \psi^{-1}\,,
\end{equation*}
which is satisfied provided that
\begin{equation} \label{C_2 geq C_3}
2 \wt C_2 \;\geq\; \wt C_3 + \wt C_2 / 2 + C_\zeta\,.
\end{equation}

In order to prove \eqref{permissible region statement}, we define, for each $z \in \C \setminus \sigma(H)$, the $k \times k$ matrix $M(z)$ through
\begin{equation} \label{definition of M}
M_{ij}(z) \;\deq\; G_{\b v^{(i)} \b v^{(j)}}(z) + \delta_{ij} d_i^{-1}\,.
\end{equation}
From Lemma \ref{lemma: det identity} we find that $x \in \sigma(\wt H) \setminus \sigma(H)$ if and only if $M(x)$ is singular. The proof therefore consists in locating $x \in \R \setminus \sigma(H)$ for which $M(x)$ is singular.

First we consider the case $x \geq 2 + \varphi^{\wt C_2} N^{-2/3}$.
On $\Xi$ we have
\begin{equation} \label{gap at edge}
\lambda_N \;\leq\; 2 + \varphi^{\wt C_2 - 1}  N^{-2/3} \qquad \text{and} \qquad \lambda_1 \;\geq\; -2 - \varphi^{\wt C_2 - 1} N^{-2/3}
\end{equation}
provided $\wt C_2$ is large enough (see Theorem \ref{theorem: rigidity}). In particular, by \eqref{gap at edge} and the definition of $\Xi$, we have $x \notin \sigma(H)$. By increasing $\wt C_2$ if necessary we may assume that $\wt C_2 \geq C_1$, where $C_1$ is the constant from Theorem \ref{theorem: strong estimate}. Therefore we get from Theorem \ref{theorem: strong estimate} and Lemma \ref{lemma: msc} that
\begin{equation} \label{M outside}
M(x + \ii y) \;=\; \msc(x + \ii y) + D^{-1} + O \pb{\varphi^{C_\zeta} N^{-1/2} \kappa_x^{-1/4}}
\end{equation}
for all $y \in [-\Sigma, \Sigma]$. (We include an imaginary part $y \neq 0$ for later applications of \eqref{M outside}; for the purposes of this proof we set $y = 0$.)

Let $i \in \{1, \dots, k^+\}$. Then we may repeat to the letter
the argument in the proof of Theorem \ref{theorem: rank-one lde} leading to \eqref{derivative of m estimated}.
Provided that $\wt C_3 \geq C_\zeta + 2$, where $C_\zeta$ is the constant in \eqref{M outside}, we therefore get that
\begin{equation*}
\absbb{\msc(x) + \frac{1}{d_i^+}} \;\geq\; \varphi^{C_\zeta + 1} N^{-1/2} \kappa_x^{-1/4} \qquad \text{if} \quad x \notin I_i^+(\b d)\,.
\end{equation*}
This takes care of the components $\b d^+$ in $D^{-1}$. In order to deal with the remaining components, $\b d^0$ and $\b d^-$, we observe that
\begin{equation*}
\msc(x) \;\in\; \qb{-1, -c}
\end{equation*}
for some $c > 0$ depending on $\Sigma$. It is now easy to put all the estimates associated with $i = 1, \dots, k$ together. Recalling \eqref{M outside} and choosing $\wt C_2$ large enough yields, for $C_\zeta$ denoting the constant from \eqref{M outside},
\begin{equation*}
\absbb{\msc(x) + \frac{1}{d_i}} \;\geq\; \varphi^{C_\zeta + 1} N^{-1/2} \kappa_x^{-1/4}
\end{equation*}
for all $i = 1, \dots, k$ provided that
\begin{equation} \label{cont condition +}
x \;\in\; \qb{2 + \varphi^{\wt C_2} N^{-2/3}, \Sigma} \setminus \bigcup_{i = 1}^{k^+} I_i^+(\b d)\,.
\end{equation}
We conclude\footnote{Here we use the well-known fact that if $\lambda \in \sigma(A + B)$ then $\dist(\lambda, \sigma(A)) \leq \norm{B}$.} from \eqref{M outside} that $M(x)$ is regular if \eqref{cont condition +} holds.

An almost identical argument applied to $\b d^-$ yields that $M(x)$ is regular if
\begin{equation} \label{regularity condition outside}
x \;\in\; \qb{-\Sigma, -2 - \varphi^{\wt C_2} N^{-2/3}} \cup \qb{2 + \varphi^{\wt C_2} N^{-2/3}, \Sigma} \setminus \pBB{\bigcup_{i = 1}^{k^-} I_i^-(\b d) \cup \bigcup_{i = 1}^{k^+} I_i^+(\b d)}\,.
\end{equation}

Next, we focus on the case
\begin{equation} \label{regularity condition inside}
x \;\in\; \qB{2 - \varphi^{\wt K} N^{-2/3}, 2 + \varphi^{\wt C_2} N^{-2/3}} \,, \qquad \dist(x, \sigma(H)) \;>\; N^{-2/3} \wt \psi^{-1}\,.
\end{equation}
Our aim is to prove that $M(x)$ is regular for any $x$ satisfying \eqref{regularity condition inside}. Once this is done, the regularity of $M(x)$ for $x$ satisfying \eqref{regularity condition outside} or \eqref{regularity condition inside} will imply \eqref{permissible region statement}. Choose $\eta \deq N^{-2/3} \wt \psi^{-1}$ and estimate
\begin{align*}
\abs{G_{\b v^{(i)} \b v^{(j)}}(x) - G_{\b v^{(i)} \b v^{(j)}}(x + \ii \eta)} &\;\leq\; \sum_\alpha \frac{\abs{\scalar{\b u^{(\alpha)}}{\b v^{(i)}}}^2 + \abs{\scalar{\b u^{(\alpha)}}{\b v^{(j)}}}^2}{2} \absbb{\frac{1}{\lambda_\alpha - x} - \frac{1}{\lambda_\alpha -x - \ii \eta}}
\\
&\;\leq\; \sum_\alpha \pB{\abs{\scalar{\b u^{(\alpha)}}{\b v^{(i)}}}^2 + \abs{\scalar{\b u^{(\alpha)}}{\b v^{(j)}}}^2} \frac{\eta}{(\lambda_\alpha - x)^2 + \eta^2}
\\
&\;=\; \im G_{\b v^{(i)} \b v^{(i)}}(x + \ii \eta) + \im G_{\b v^{(j)} \b v^{(j)}}(x + \ii \eta)\,,
\end{align*}
where in the second step we used \eqref{regularity condition inside}. Therefore, by definition of $\Xi$ (See also Theorem \ref{thm: ILSC}) and Lemma \ref{lemma: msc}, we get (recall that $\wt \psi \geq 1$)
\begin{equation*}
G_{\b v^{(i)} \b v^{(j)}}(x) \;=\; \delta_{ij} \msc(x + \ii \eta) + O \pbb{\varphi^{C_\zeta} \im \msc(x + \ii \eta) + \frac{\varphi^{C_\zeta}}{N \eta}} \;=\; -\delta_{ij} + O \pbb{\varphi^{C_\zeta} N^{-1/3} \pB{\wt \psi + \varphi^{\wt K/2} + \varphi^{\wt C_2/2}}}\,.
\end{equation*}
This implies, for any $x$ satisfying \eqref{regularity condition inside}, that
\begin{equation} \label{M in gap}
M(x) \;=\; -\umat + D^{-1} + O \pbb{\varphi^{C_\zeta} N^{-1/3} \pB{\wt \psi + \varphi^{\wt K/2} + \varphi^{\wt C_2/2}}}\,.
\end{equation}
Since
\begin{equation*}
\absbb{-1 + \frac{1}{d_i}} \;\geq\; \frac{1}{2} \, \varphi^{\wt C_2} \wt \psi N^{-1/3}
\end{equation*}
for all $i$, we find that $M(x)$ is regular provided $\wt C_2$ is chosen large enough that
\begin{equation*}
\wt C_2 -1 \;\geq\; C_\zeta + \wt K/2 + \wt C_2 /2\,.
\end{equation*}
This completes the analysis of the case \eqref{regularity condition inside}. The case
\begin{equation*}
x \;\in\; \qB{-2 - \varphi^{\wt C_2} N^{-2/3}, -2 + \varphi^{\wt K} N^{-2/3}} \,, \qquad \dist(x, \sigma(H)) \;>\; N^{-2/3} \wt \psi^{-1}
\end{equation*}
is handled similarly. This completes the proof.
\end{proof}

\subsection{The initial configuration}
In this section we fix a configuration $\b d(0) \equiv \b d$ that is \emph{independent of $N$}, and satisfies $k^0 = 0$ as well as
\begin{equation} \label{condition for initial data}
-\Sigma + 1 \leq d_1^- \;<\; \cdots \;<\; d_{k^-}^-\;<\; -1\,, \qquad
1 \;<\; d^+_1 \;<\; \cdots \;<\; d^+_{k^+} \;\leq\; \Sigma - 1\,.
\end{equation}
Note that $\b d \in \cal D^*(\wt C_2)$ for large enough $N$.

First we deal with the outliers.

\begin{proposition} \label{prop: initial outliers}
For $N$ large enough, each interval $I_i^-(\b d)$, $i = 1, \dots, k^-$, and $I_i^+(\b d)$, $i = 1, \dots, k^+$, contains precisely one eigenvalue of $\wt H$.
\end{proposition}
\begin{proof}
Let $i \in \{1, \dots,k^+\}$ and pick a small $N$-independent positively oriented closed contour $\cal C \subset \C \setminus [-2,2]$ that encloses $\theta(d_i^+)$ but no other point of the set $\bigcup_{\sigma = \pm} \bigcup_{i = 1}^{k^\sigma} \h{\theta(d_i^\sigma)}$. By Proposition \ref{prop: permissible region}, it suffices to show that the interior of $\cal C$ contains precisely one eigenvalue of $\wt H$. Define
\begin{equation*}
f_N(z) \;\deq\; \det \pb{M(z) + D^{-1}} \,, \qquad g(z) \;\deq\; \det \pb{\msc(z) + D^{-1}}\,.
\end{equation*}
The functions $g$ and $f_N$ are holomorphic on and inside $\cal C$ (for large enough $N$). Moreover, by construction of $\cal C$, the function $g$ has precisely one zero inside $\cal C$, namely at $z = \theta(d_i^+)$. Next, we have
\begin{equation*}
\min_{z \in \cal C} \abs{g(z)} \;\geq\; c \;>\; 0\,, \qquad \abs{g(z) - f_N(z)} \;\leq\; \varphi^{C_\zeta} N^{-1/2}\,,
\end{equation*}
where the second inequality follows from \eqref{M outside}. The claim now follows from Rouch\'e's theorem. The eigenvalues near $\theta(d_i^-)$, $i = 1, \dots, k^-$, are handled similarly.
\end{proof}

Before moving on, we record the following result on rank-one deformations.
\begin{lemma} \label{lemma: limit rank one}
Let $\b v \in \C^k$ be nonzero. Then for all $i = 1, \dots, k - 1$ and all
Hermitian $k \times k$ matrices $A$ we have
\begin{equation*}
\lim_{d \to \infty} \lambda_i(A + d \b v \b v^*) \;=\; \lim_{d \to -\infty} \lambda_{i + 1}(A + d \b v \b v^*)\,.
\end{equation*}
\end{lemma}
\begin{proof}
By Lemma \ref{lemma: det identity}, we find that $x \notin \sigma(A)$ is an eigenvalue of $A + d \b v \b v^*$ if and only if
\begin{equation*}
\scalar{\b v}{(A - x)^{-1} \b v} \;=\; -\frac{1}{d}\,.
\end{equation*}
Let
\begin{equation*}
E \;\deq\; \hB{A \col \text{the eigenvalues of $A$ are distinct} \,,\, \scalar{\b v}{\b u^{(i)}(A)} \neq 0 \text{ for all } i}\,,
\end{equation*}
where $\b u^{(i)}(A)$ denotes the eigenvector of $A$ associated with the eigenvalue $\lambda_i(A)$. (Note that $\b u^{(i)}(A)$ is well-defined in $E$, since the eigenvalues are distinct.) It is not hard to see that $E^c$ is dense in the space of Hermitian matrices.

We write the condition $\scalar{\b v}{(A - x)^{-1} \b v} = - d^{-1}$ as
\begin{equation*}
f(x) \;\deq\; \sum_{i} \frac{\abs{\scalar{\b v}{\b u^{(i)}(A)}}^2}{\lambda_i(A) - x} \;=\; -\frac{1}{d}\,,
\end{equation*}
Let $A \in E$. Then $f$ has $k$ singularities at the eigenvalues of $H$, away from which we have $f' > 0$ . Moreover, $f(x) \uparrow 0$ as $x \uparrow \infty$, and $f(x) \downarrow 0$ as $x \downarrow - \infty$. Thus, for any $d \in \R \setminus \{0\}$, the equation $f(x) = - d^{-1}$ has exactly $k$ solutions in $\R \setminus \sigma(A)$. Since $A + d \b v \b v^*$ has at most $k$ distinct eigenvalues, this proves that $\sigma(A + d \b v \b v^*) \cap \sigma(A) = \emptyset$ for all $d \in \R$. Moreover, the equation $f(x) = 0$ has exactly $k - 1$ solutions, $x_1, \dots, x_{k - 1}$. Since $f'(x_i) > 0$ for each $i = 1, \dots, k - 1$, it is easy to see that $x_i = \lim_{d \to \infty} \lambda_i(A + d \b v \b v^*) = \lim_{d \to -\infty} \lambda_{i + 1}(A + d \b v \b v^*)$.

Now the claim follows by approximating an arbitrary matrix $A$ by matrices in $E$, and by using the Lipschitz continuity of the map $A \mapsto \lambda_i(A)$.
\end{proof}

We now deal with the extremal bulk eigenvalues.
\begin{proposition} \label{prop: intial bulk}
Fix $0 < \delta < 1/3$ and $\wt K > 0$. Let $\b d$ be $N$-independent and satisfy \eqref{condition for initial data}. Then for large enough $N$ (depending on $\delta$ and $\wt K$) we have for all $\alpha$ satisfying $\lambda_\alpha \geq 2 - \varphi^{\wt K} N^{-2/3}$ that
\begin{equation*}
\abs{\lambda_\alpha - \mu_{\alpha - k^+}} \;\leq\; N^{-1 + \delta}\,.
\end{equation*}
Similarly, we have for all $\alpha$ satisfying $\lambda_\alpha \leq -2 + \varphi^{\wt K} N^{-2/3}$ that
\begin{equation*}
\abs{\lambda_\alpha - \mu_{\alpha + k^-}} \;\leq\; N^{-1 + \delta}\,.
\end{equation*}
\end{proposition}

\begin{proof}
We only prove the first statement; the proof of the second one is almost identical.
Abbreviate $\delta' \deq \delta / 2$.

Before embarking on the full proof, we first give a sketch of its main idea, under some simplifying assumptions. Let $A \in \N$ be some fixed constant, and assume that, for each $\alpha \geq N - A$,  the neighbours of $\lambda_\alpha$ are further than $N^{-1 + \delta'}$ away from $\lambda_\alpha$. (This assumption in fact holds with probability $1 - o(1)$, a fact we shall neither use nor prove.) We claim that there is \emph{at least} one eigenvalue of $\wt H$ in the interval $[x_-^\alpha, x_+^\alpha]$ surrounding $\lambda_\alpha$, where
\begin{equation*}
x_\pm^\alpha \;\deq\; \lambda_\alpha \pm N^{-1 + \delta'} / 3\,.
\end{equation*}
Before sketching the proof of the above claim, we show how to use it to conclude the argument. By Proposition \ref{prop: initial outliers}, there are at least $k^+$ eigenvalues in $(x_+^N, \infty)$. Recall that by assumption $k^0 = 0$, i.e.\ $\abs{d_i} > 1$ for all $i$. Therefore using interlacing, i.e.\ a repeated application of Lemma \ref{lemma: interlacing}, we conclude that there are exactly $k^+$ eigenvalues in $(x_+^N, \infty)$. From the above claim we find that there is at least one eigenvalue in $[x_-^N, x_+^N]$. Using interlacing we find that there are at most $k^+ + 1$ eigenvalues in $[x_-^N, \infty)$. We conclude that there is exactly one eigenvalue in $[x_-^N,x_+^N]$. We may move on to the $(N - 1)$-th eigenvalue: we have proved that there are (i) at least $k^+ + 1$ eigenvalues in $[x_-^N,\infty)$ (from the previous step), (ii) at least one eigenvalue in $[x_-^{N - 1}, x_+^{N - 1}]$ (from the claim), and (iii) at most $k^+ + 2$ eigenvalues in $[x_-^{N - 1}, \infty)$ (from interlacing); we conclude that there is exactly one eigenvalue in $[x_-^{N - 1}, x_+^{N - 1}]$. Continuing in this fashion concludes the proof.

Let us now complete the sketch of the proof of the above claim. Assume for simplicity that $H$ and $\wt H$ have no common eigenvalues. From Lemma \ref{lemma: det identity} we find that $x$ is an eigenvalue of $\wt H$ if and only if the matrix $M(x)$, defined in \eqref{definition of M}, is singular. Thus, we have to prove that there is an $x \in [x_-^\alpha, x_+^\alpha]$ such that $M(x)$ is singular. The idea of the argument is to do a spectral decomposition of $G$, and resum all terms not associated with $\lambda_\alpha$ to get something close to $\re m(x) \approx -1$. More precisely, we write
\begin{align*}
M_{ij}(x) &\;=\;
\frac{\scalar{\b v^{(i)}}{\b u^{(\alpha)}} \scalar{\b u^{(\alpha)}}{\b v^{(j)}}}{\lambda_\alpha - x}
+ \sum_{\beta \neq \alpha} \frac{\scalar{\b v^{(i)}}{\b u^{(\beta)}} \scalar{\b u^{(\beta)}}{\b v^{(j)}}}{\lambda_\beta - x} + \delta_{ij} d_i^{-1}
\\
&\;\approx\; \frac{\scalar{\b v^{(i)}}{\b u^{(\alpha)}} \scalar{\b u^{(\alpha)}}{\b v^{(j)}}}{\lambda_\alpha - x} + \re m(x) \delta_{ij} + \delta_{ij} d_i^{-1}\,,
\end{align*}
where the sum over $\beta$ was replaced with $\re m(x) \delta_{ij}$ (up to negligible error terms). This approximation will be justified using Theorems \ref{thm: ILSC} and \ref{theorem: delocalization}; it uses that $x \in [x_-^\alpha, x_+^\alpha]$ and consequently all eigenvalues $\lambda_\beta$, $\beta \neq \alpha$, are separated from $x$ by at least $N^{-1 + \delta} / 3$. Introducing the vector $\b y  = (y_i) \in \C^k$, defined by $y_i \deq \scalar{\b v^{(i)}}{\b u^{(\alpha)}}$, we therefore get
\begin{equation} \label{approx M}
M(x) \;\approx\; \frac{\b y \b y^*}{\lambda_\alpha - x} - \umat + D^{-1}\,,
\end{equation}
where we used that $\re m(x) \approx -1$. By assumption, $\abs{d_i} > 1$ for all $i$; therefore the matrix $- \umat + D^{-1}$ is strictly negative. Also, Theorem \ref{theorem: delocalization} implies that $\abs{y_i} \leq \varphi^{C_\zeta} N^{-1/2}$. Thus it is easy to conclude that all eigenvalues of $M(x_-^\alpha)$ are negative. The first term on the right-hand side of \eqref{approx M} is a rank-one matrix. As $x$ approaches $\lambda_\alpha$ from the left, its nonzero eigenvalue tends to $+\infty$. By continuity, there must therefore exist an $x \in [x_-^\alpha, \lambda_\alpha)$ such that $M(x)$ is singular. This concludes the sketch of the proof of the claim.

Now we turn towards the detailed proof in the general case. Since eigenvalues of $H$ may be separated by less than $N^{-1 + \delta'}$, we begin by clumping together eigenvalues of $H$ which are separated by less than $N^{-1 + \delta'}$. More precisely, we construct a partition $\cal A = (A_q)_q$ of $\{1, \dots, N\}$, defined as the finest partition in which $\alpha$ and $\beta$ belong to the same block if $\abs{\lambda_\alpha - \lambda_\beta} \leq N^{-1 + \delta'}$. Thus, each block consists of a sequence of consecutive integers. We order the blocks of $\cal A$ in a ``decreasing'' fashion, in such a way that if $q < r$ then $\lambda_\alpha > \lambda_\beta$ for all $\alpha \in A_q$ and $\beta \in A_r$.

We now derive a bound on the size of the blocks near the edge. Roughly, we shall show that if $\lambda \in A_q$ and $\lambda \geq 2 - \varphi^{C} N^{-2/3}$ then $\abs{A_q} \leq \varphi^{C'}$. Let $C_4$ be a large constant to be chosen later. Now choose $\alpha$ and $\beta$ satisfying $0 \leq \alpha \leq \beta \leq \varphi^{C_4}$ such that $N - \alpha$ and $N - \beta$ belong to the same block. Then by definition of $\Xi$ and $\cal A$ we have
\begin{equation*}
c \qB{(\beta / N)^{2/3} - (\alpha / N)^{2/3}} - \varphi^{C_\zeta} N^{-2/3} \;\leq\; \lambda_{N - \alpha} - \lambda_{N - \beta} \;\leq\; (\beta - \alpha) N^{-1 + \delta'}\,,
\end{equation*}
where we used the statement of Theorem \ref{theorem: rigidity} and the definition \eqref{classical location}. Thus we get the condition
\begin{equation*}
N^{-2/3} \qB{c \beta^{-1/3} (\beta - \alpha) - \varphi^{C_\zeta}} \;\leq\; N^{-1+ \delta'}(\beta - \alpha)\,.
\end{equation*}
We conclude that if $\alpha$ and $\beta$ satisfy $0 \leq \alpha \leq \beta \leq \varphi^{C_4}$ and $N - \alpha$ and $N - \beta$ belong to the same block, then
\begin{equation} \label{size of block}
\beta - \alpha \;\leq\; \varphi^{C_\zeta + C_4/3 + 1}\,.
\end{equation}

Let $\alpha_*$ denote the largest integer such that $\lambda_{N - \alpha_*} \geq 2 - \varphi^{\wt K}N^{-2/3}$. In particular, by definition of $\Xi$ (see Theorem \ref{theorem: rigidity}) we have
\begin{equation} \label{bound on alpha star}
\alpha_* \;\leq\; \varphi^{3\wt K/2 + C_\zeta}\,.
\end{equation}
Now we choose $C_4 \equiv C_4(\zeta,\wt K)$ large enough that
\begin{equation*}
C_4 \;\geq\; \max \pB{3 \wt K/2 + C_\zeta\,,\, C_\zeta + C_4/3 + 1} + 2\,.
\end{equation*}
Next, define $Q$ through $N - \alpha_* \in A_Q$. Therefore we get from \eqref{size of block} and \eqref{bound on alpha star} that any $\alpha \leq \varphi^{C_4}$ such that $N - \alpha \in A_Q$ satisfies
\begin{equation*}
\alpha \;\leq\; \alpha^* + \varphi^{C_\zeta + C_4 / 3 + 1} \;\leq\; \varphi^{C_4 - 1}\,.
\end{equation*}
Since blocks are contiguous, we conclude that
\begin{equation} \label{estimate on size of Aq}
\abs{A_q} \;\leq\; \varphi^{C_4 - 1}\,.
\end{equation}
for each $q = 1, \dots, Q$. Moreover, by definition of $\Xi$ (see Theorem \ref{theorem: rigidity}), we find
\begin{equation*}
\abs{\lambda_{N - \alpha} - 2} \;\leq\; \varphi^{2 C_4 / 3 + C_\zeta} N^{-2/3}\,.
\end{equation*}
for all $q = 1, \dots, Q$ and all $\alpha$ such that $N - \alpha \in A_q$.

Now we are ready for the main argument. Pick $q \in \{1, \dots, Q\}$ and abbreviate
\begin{equation*}
a^q \deq \min_{\alpha \in A_q} \lambda_\alpha \,, \qquad
b^q \deq \max_{\alpha \in A_q} \lambda_\alpha \,.
\end{equation*}
We introduce the path
\begin{equation*}
x_t^q \;\deq\; a^q - N^{-1 + \delta'}/3 + \pb{b^q - a^q + 2 N^{-1 + \delta'}/3} \, t\,, \qquad (t \in [0,1])\,,
\end{equation*}
which will serve to count eigenvalues. (Note that $x_0^q = a^q - N^{-1 + \delta'}/3$ and $x_1^q = b^q + N^{-1 + \delta'}/3$.) The interval $[x_0^q, x_1^t]$ contains precisely those eigenvalues of $H$ that are in $A_q$, and its endpoints $x_0^q$ and $x_1^q$ are at a distance greater than $N^{-1 + \delta'}/3$ from any eigenvalue of $H$. Thus, $[x_0^q, x_1^t]$ is the correct generalization of the interval $[x_-^\alpha, x_+^\alpha]$ from the sketch given at the beginning of this proof.

In order to avoid problems with exceptional events, we add some randomness to $D$. Recall that $D$ satisfies \eqref{condition for initial data}. Let $\Delta$ be a $k \times k$ Hermitian random matrix whose upper triangular entries are independent and have an absolutely continuous law supported in the unit disk. For $\epsilon > 0$ define
\begin{equation*}
\wt H^\epsilon  \;\deq\; H + V (D^{-1} + \epsilon \Delta)^{-1} V^*\,.
\end{equation*}
From now on we use ``almost surely'' to mean almost surely with respect to the randomness of $\Delta$.
Our main goal is to prove that for each $\epsilon > 0$, almost surely, there are at least $\abs{A_q}$ eigenvalues of $\wt H^\epsilon$ in $\q{x_0^q, x_1^q} \setminus \sigma(H)$. (Having done this, we shall deduce, by taking $\epsilon \to 0$, that $\wt H$ has at least $\abs{A_q}$ eigenvalues in $[x_0^q, x_1^q]$.)

For $x \notin \sigma(H)$ define
\begin{equation*}
M_{ij}^\epsilon(x) \;\deq\; G_{\b v^{(i)} \b v^{(j)}}(x) + \delta_{ij} d_i^{-1} + \epsilon \Delta_{ij} \qquad (i,j = 1, \dots, k)\,.
\end{equation*}
Then (assuming $x \notin \sigma(H)$) we know that $x \in \sigma(\wt H^\epsilon)$ if and only if $M^\epsilon(x)$ is singular. Split
\begin{equation*}
G_{\b v^{(i)} \b v^{(j)}}(x) \;=\; \sum_{\alpha \in A_q} \frac{\scalar{\b v^{(i)}}{\b u^{(\alpha)}} \scalar{\b u^{(\alpha)}}{\b v^{(j)}}}{\lambda_\alpha - x} + \sum_{\alpha \notin A_q} \frac{\scalar{\b v^{(i)}}{\b u^{(\alpha)}} \scalar{\b u^{(\alpha)}}{\b v^{(j)}}}{\lambda_\alpha - x}\,.
\end{equation*}
Let $x \in [x_0^q, x_1^q]$.
Similarly to the proof of \eqref{M in gap}, we choose $\eta \deq N^{-1 + \delta'}$ and estimate
\begin{multline*}
\absBB{\sum_{\alpha \notin A_q} \frac{\scalar{\b v^{(i)}}{\b u^{(\alpha)}} \scalar{\b u^{(\alpha)}}{\b v^{(j)}}}{\lambda_\alpha - x}
-
\sum_{\alpha \notin A_q} \frac{\scalar{\b v^{(i)}}{\b u^{(\alpha)}} \scalar{\b u^{(\alpha)}}{\b v^{(j)}}}{\lambda_\alpha - x - \ii \eta}
}
\\
\leq\; 2\pB{\im G_{\b v^{(i)} \b v^{(i)}}(x + \ii \eta) + \im G_{\b v^{(j)} \b v^{(j)}}(x + \ii \eta)}\,,
\end{multline*}
where we used that $\abs{x - \lambda_\alpha} \geq 2 N^{-1 + \delta'}/3$ for $\alpha \notin A_q$. Moreover,
\begin{equation*}
\absBB{\sum_{\alpha \in A_q} \frac{\scalar{\b v^{(i)}}{\b u^{(\alpha)}} \scalar{\b u^{(\alpha)}}{\b v^{(j)}}}{\lambda_\alpha - x - \ii \eta}} \;\leq\; \varphi^{C_\zeta + C_4} N^{- \delta'}\,,
\end{equation*}
where we used \eqref{size of block} and the definition of $\Xi$ (see Theorem \ref{theorem: delocalization}). Estimating $G_{\b v^{(i)} \b v^{(j)}}(x + \ii \eta) - \msc(x + \ii \eta)$ therefore yields, similarly to \eqref{M in gap},
\begin{equation*}
M_{ij}^\epsilon(x) \;=\; \sum_{\alpha \in A_q} \frac{\scalar{\b v^{(i)}}{\b u^{(\alpha)}} \scalar{\b u^{(\alpha)}}{\b v^{(j)}}}{\lambda_\alpha - x} - \delta_{ij} + \delta_{ij} d_i^{-1} + \epsilon \Delta_{ij} + O \pb{\varphi^{C_\zeta + C_4} N^{- \delta'/2}}\,.
\end{equation*}
Introducing the vector
\begin{equation*}
\b y^{(\alpha)} \;=\; (y^{(\alpha)}_i)_{i = 1}^k\,, \qquad y^{(\alpha)}_i \;\deq\; \scalar{\b v^{(i)}}{\b u^{(\alpha)}}\,,
\end{equation*}
we get
\begin{equation}  \label{main estimate for M}
M^\epsilon(x) \;=\; \sum_{\alpha \in A_q} \frac{\b y^{(\alpha)} (\b y^{(\alpha)})^*}{\lambda_\alpha - x} - \umat + D^{-1} + \epsilon \Delta + R(x)\,, \qquad 
R(x) \;=\; O \pb{\varphi^{C_\zeta + C_4} N^{- \delta'/2}}\,,
\end{equation}
where $R(x)$ is continuous in $x$ and independent of $\Delta$. Compare this to \eqref{approx M} in the sketch given at the beginning of the proof. By Theorem \ref{theorem: delocalization}, for $\alpha \in A_q$ we have
\begin{equation} \label{associated estiamte for y}
\abs{y^{(\alpha)}_i} \;=\; O \pb{\varphi^{C_\zeta} N^{-1/2}}\,.
\end{equation}

We may now start the counting of the eigenvalues of $\wt H$ in $[x_0^q, x_1^q]$. We have to prove that there are at least $L \deq \abs{A_q}$ distinct points $x$ in $[x_0^q, x_1^q]$ at which $M^\epsilon(x)$ has a zero eigenvalue. As in the simple continuity argument given in the sketch at the beginning of this proof, we shall make use of continuity. However, having to find $L$ such values $x$ instead of just one is a significant complication\footnote{This complication is also visible in the joint arrangement of the eigenvalues of $H$ and $\wt H$. If all eigenvalues of $H$ are well-separated (by at least $N^{-1 + \delta'}$) then, as outlined in the sketch at the beginning of the proof, each eigenvalue $\lambda_\alpha$ of $H$ has an associated eigenvalue of $\wt H$, which lies in the interval $[\lambda_\alpha - N^{-1 + \delta'}/3, \lambda_\alpha)$. In fact, this eigenvalue typically lies at a distance $N^{-1}$ to the left of $\lambda_\alpha$, as follows from \eqref{approx M} and the fact that the typical size of $\b y$ is $N^{-1/2}$. However, if two eigenvalues of $H$ are closer than $N^{-1}$, this simple ordering breaks down. In general, therefore, all we can say about the eigenvalues of $\wt H$ associated with the eigenvalues of $H$ in $A_q$ is that they are close to the \emph{group} $\{\lambda_\alpha\}_{\alpha \in A_q}$. Since the diameter of this group is small (see \eqref{diameter of the group} below), this will be enough.}. Before coming to the full counting argument, we give a sketch of its main idea. See Figure \ref{fig: beads} for a graphical depiction of this sketch. We extend the real line $\R$, on which the eigenvalues of $M^\epsilon(x)$ reside, to the real projective line $\ol \R = \R \cup \{\infty\} \cong S^1$. One can think of $\ol \R$ as a ring with two distinguished points, $0$ at the bottom and $\infty$ at the top. Thanks to Lemma \ref{lemma: limit rank one}, it is possible to label the $k$ eigenvalues of $M^\epsilon(x_t^q)$ so that they are continuous $\ol \R$-valued functions (denoted by $\tilde e_1^\epsilon(t), \dots, \tilde e_k^\epsilon(t)$ below) on $[0,1]$. Thus, we get a family of $k$ beads moving continuously counterclockwise on a ring. At $t = 0$, the eigenvalues are all strictly negative (and finite), i.e.\ all beads lie in the left half of the ring. As $t$ is continuously increased from $0$ to $1$, the beads move counterclockwise around the ring. Our goal is to count the number of times $0$ is hit by a bead. Thanks to the explicit form of the first term on the right-hand side of \eqref{main estimate for M}, we know that the point $\infty$ is hit exactly $L$ times as $t$ ranges from $0$ to $1$. Since at time $t = 0$ all beads were in the left half of the ring, and since the beads move continuously counterclockwise, we conclude by continuity that $0$ is hit at least $L$ times as $t$ ranges from $0$ to $1$. Below, we denote the times at which $\infty$ is hit by $s_1, \dots, s_L$, and the times at which $0$ is hit by $t_1, \dots, t_L$. One nuisance we have to deal with in the proof is the possibility of several beads crossing one of the two points $0$ or $\infty$ simultaneously. Such events are not admissible for our counting. For instance, if at time $t$ a bead is at $0$ while another is at $\infty$, we cannot conclude that $x_t^q$ is an eigenvalue of $\wt H$; indeed, because there is a bead at $\infty$, we know that $x_t^q$ is an eigenvalue of $H$, and hence Lemma \ref{lemma: det identity} is not applicable. However, such pathological events almost surely do not occur. Avoiding them was the reason for introducing $\Delta$. Note that the final result of the counting argument -- the number of eigenvalues of $\wt H^\epsilon$ in $[x_0^q, x_1^q]$ -- is stable under the limit $\epsilon \to 0$. This will allow us to conclude the proof.

\begin{figure}[ht!]
\begin{center}
\includegraphics{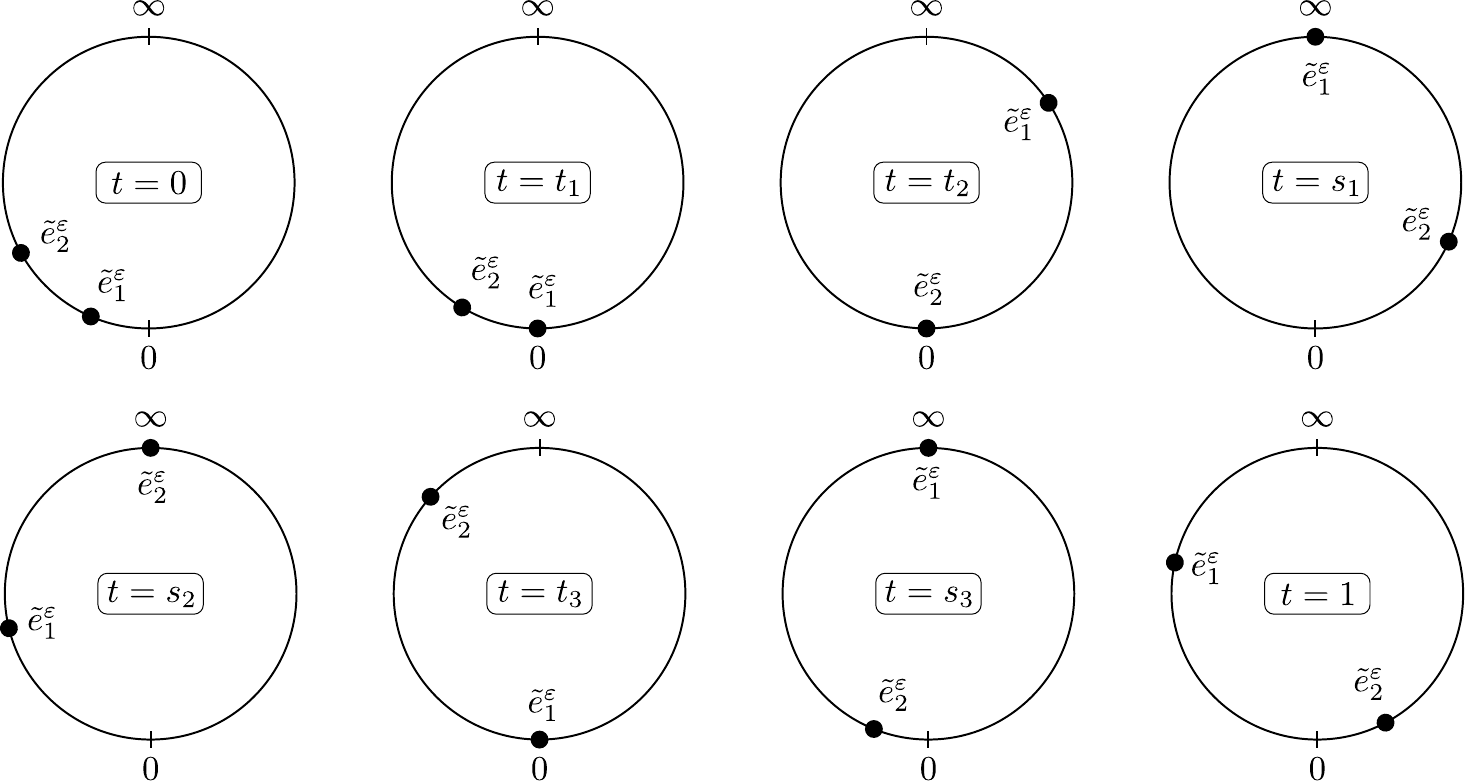}
\end{center}
\caption{A graphical representation of the movement of the eigenvalues (or ``beads'') $\tilde e_1^\epsilon(t), \tilde e_2^\epsilon(t)$ of $M^\epsilon(x^q_t)$ as $t$ ranges from $0$ to $1$. In this example we have $L = 3$, $k = 2$, and $0 < t_1 < t_2 < s_1 < s_2 < t_3 < s_3 < 1$. \label{fig: beads}}
\end{figure}

Now we give the full proof.
Recall that $\abs{d_i} > 1$ is independent of $N$ for all $i$. Thus we get from \eqref{main estimate for M} and \eqref{associated estiamte for y} that, for large enough $N$ and small enough $\epsilon$, all eigenvalues of $M^\epsilon(x_0^q)$ are negative.  (Here we used that $\abs{\lambda_\alpha - x_0^q} \geq N^{-1 + \delta'} / 3$ for $\alpha \in A_q$.) We shall vary $t$ continuously from $0$ to $1$ and count the number of eigenvalues crossing the origin. Let $L \deq \abs{A_q}$ and denote by
\begin{equation*}
0 \;<\; s_1 \;<\; s_2 \;<\; \cdots \;<\; s_L \;<\; 1
\end{equation*}
the values of $t$ at which $x_t^q \in \sigma(H)$. (Recall that the eigenvalues of $H$ are distinct.) It is also convenient to write $s_0 = 0$ and $s_{L + 1} = 1$. For $t \in [0,1] \setminus \{s_1, \dots s_L\}$, let
\begin{equation*}
e_1^\epsilon(t) \;\leq\; e_2^\epsilon(t) \;\leq\; \cdots \;\leq\; e_k^\epsilon(t)
\end{equation*}
denote the ordered eigenvalues of $M^\epsilon(x_t^q)$. We record the following fundamental properties of $e_1^\epsilon(t), \dots, e_k^\epsilon(t)$.
\begin{enumerate}
\item
For all $i = 1, \dots, k$, we have $e_i^\epsilon(0) < 0$ for $N$ large enough and $\epsilon$ small enough (depending on $N$).
\item
For every $\ell = 0, \dots, L$ and $i = 1, \dots, k$, the function $e_i^\epsilon$ is continuous on $(s_\ell, s_{\ell + 1})$.
\item
At each singular point $s_\ell$, $\ell = 1, \dots, L$, we have
\begin{equation*}
e_i^\epsilon (s_\ell^-) \;=\; e_{i+1}^\epsilon (s_\ell^+) \qquad (i = 1, \dots k - 1)\,.
\end{equation*}
(In particular, both one-sided limits exist.)
\end{enumerate}
Property (i) was proved after \eqref{associated estiamte for y}. Property (ii) follows from \eqref{main estimate for M}. Property (iii) follows from Lemma \ref{lemma: limit rank one}, using \eqref{main estimate for M} and the fact that $R(x)$ is continuous.

Moreover, the two following claims are true almost surely.
\begin{itemize}
\item[(a)]
For each $\ell = 1, \dots, L$ and $i = 1, \dots, k - 1$ we have $e^\epsilon_i(s_\ell^-) \neq 0$. (The remaining index $k$ satisfies $e_k^\epsilon(s_\ell^-) = + \infty$.)
\item[(b)]
If $e_i^\epsilon(t) = 0$ for some $t \in [0,1] \setminus \{s_1, \dots, s_L\}$ then $e_j^\epsilon(t) \neq 0$ for all $j \neq i$.
\end{itemize}
In terms of beads $\tilde e_1^\epsilon(t), \dots, \tilde e_k^\epsilon(t) \in \ol R$ (see below), the properties (a) and (b) can be informally summarized as: (a) if a bead is at $\infty$ then there is no bead at $0$, (b) at most one bead is at $0$. We omit the standard\footnote{The ``standard'' arguments rely on the fact that the set of singular Hermitian matrices is an algebraic variety of codimension one. In addition, the proof of (a) requires the following fact. Let $P$ be a rank-one orthogonal projector on $\C^k$ and $A$ a Hermitian $k \times k$ matrix; then, as $x \to \pm \infty$, exactly $k - 1$ eigenvalues of the matrix $A + x P$ converge, and their limits coincide with the eigenvalues of $A$ restricted to a map from $\ker P$ to $\ker P$. The proof of (b) uses that the set of Hermitian matrices with multiple eigenvalues at zero is an algebraic variety of codimension two.} proofs of (a) and (b), which rely on the fact that the law of $\Delta$ is absolutely continuous.

In order to conclude our main argument, it is convenient to regard the eigenvalues $e^\epsilon_1(t), \dots, e^\epsilon_k(t)$ as elements of $\ol \R = \R \cup \{\infty\} \cong S^1$, the real projective line. From properties (ii) - (iii), it is apparent that we may rearrange the eigenvalues of $M^\epsilon(x^q_t)$ as $\tilde e_1^\epsilon(t), \dots, \tilde e_k^\epsilon(t) \in \ol R$ and extend them to functions (``beads'') on whole interval $[0,1]$ in such a way that, almost surely, each $\tilde e_i^\epsilon$ is a continuous $\ol \R$-valued function  on $[0,1]$. 

We now claim the following.
\begin{itemize}
\item[$(*)$]
Almost surely, there are $L$ distinct times $t_1 < t_2 < \cdots < t_L \in [0,1] \setminus \{s_1, \dots, s_L\}$ such that for each $\ell = 1, \dots, L$ there is an $i = 1, \dots, k$ with $\tilde e_i^\epsilon(t_\ell) = 0$.
\end{itemize}
Let us prove $(*)$. Let $n_i \in \N$ denote the number of times that $\tilde e_i^\epsilon$ hits $\infty$ as $t$ ranges from $0$ to $1$. From \eqref{main estimate for M} we find that $\sum_{i = 1}^k n_i = L$ (recall that the eigenvalues of $H$ are distinct). Moreover, again from \eqref{main estimate for M}, we find that each such passage of $\infty$ by $\tilde e_i^\epsilon$ always takes place in the same direction, namely from the positive reals to the negative reals with $t$ increasing. More precisely, if $\tilde e_i^\epsilon(t_*) = \infty$ then there is a neighbourhood $I \ni t_*$ such that for all $t \in I$ we have
\begin{equation*}
\tilde e_i^\epsilon(t) \;\in\; \R_+ \quad \text{for} \quad t < t_* \qquad \text{and} \qquad \tilde e_i^\epsilon(t) \;\in\; \R_- \quad \text{for} \quad t > t_*.
\end{equation*}
Since at time zero we have $\tilde e_i^\epsilon(0) \in \R_-$ (see Property (i) above) we conclude that $\tilde e_i^\epsilon$ has at least $n_i$ distinct zeros. (Recall that $n_i$ was defined as the number of times $\wt e_i^\epsilon$ hits $\infty$.) Moreover, by Property (a), the zeros $\tilde e_i^\epsilon$ are almost surely in $[0,1] \setminus \{s_1, \dots, s_L\}$. By Property (b), the zeros of $e_1^\epsilon, \dots, e_k^\epsilon$ are almost surely disjoint. Since $\sum_{i = 1}^k n_i = L$, the claim $(*)$ follows.

From $(*)$ we conclude that, almost surely, $M^\epsilon(x)$ is singular in at least $L$ points in the set $[x_0^q, x_1^q] \setminus \sigma(H)$. Therefore $\wt H^\epsilon$ has almost surely at least $L$ eigenvalues in $[x_0^q, x_1^q]$. Taking $\epsilon \to 0$, we find that $\wt H$ has at least $L = \abs{A_q}$ eigenvalues in $[x_0^q, x_1^q]$.

What remains is to prove that $\wt H$ has at most $\abs{A_q}$ eigenvalues in $[x_0^q, x_1^q]$. We prove this using interlacing, similarly to the corresponding argument given in the sketch at the beginning of the proof. Together with Proposition \ref{prop: initial outliers}, we have proved that there are at least $\abs{A_1} + k^+$ eigenvalues of $\wt H$ in $[x_0^1, \infty)$. By interlacing (i.e.\ a repeated application of Lemma \ref{lemma: interlacing}), we find that there are at most $\abs{A_1} + k^+$ eigenvalues of $\wt H$ in $[x_0^1, \infty)$. We deduce, again using Proposition \ref{prop: initial outliers}, that there are exactly $\abs{A_1}$ eigenvalues of $\wt H$ in $[x_0^1, x_1^1]$.

We have proved that there are at least $\abs{A_1} + \abs{A_2} + k^+$ eigenvalues of $\wt H$ in $[x_0^2, \infty)$. Using eigenvalue interlacing, we find that there are at most $\abs{A_1} + \abs{A_2} + k^+$ eigenvalues of $\wt H$ in $[x_0^2, \infty)$.  We conclude that there are exactly $\abs{A_2}$ eigenvalues of $\wt H$ in $[x_0^2, x_1^2]$.

We may now repeat this argument for $q = 3, 4, \dots, Q$, to get that $\wt H$ has exactly $\abs{A_q}$ eigenvalues in $[x_0^q, x_1^q]$, for $q = 1,2, \dots, Q$. Moreover, by \eqref{estimate on size of Aq}, we find for any $\alpha \in A_q$ that
\begin{equation} \label{diameter of the group}
\sup \, \hB{\abs{x - \lambda_\alpha} \col \alpha \in A_q \,,\, x \in [x_0^q,x_1^q]} \;\leq\; \varphi^{C_4} N^{-1 + \delta'} \;\leq\; N^{-1 + \delta}\,.
\end{equation}
Therefore the proof is complete.
\end{proof}

\subsection{Bootstrapping and conclusion of the proof of Theorem \ref{theorem: rank-k lde}}
We may now complete the proof of Theorem \ref{theorem: rank-k lde}. In order to extend the statements of Propositions \ref{prop: initial outliers} and \ref{prop: intial bulk} to arbitrary $N$-dependent configurations $\b d \in \cal D(\wt C_2)$, we continuously deform an $N$-independent $\b d$, for which Propositions \ref{prop: initial outliers} and \ref{prop: intial bulk} hold, to the desired $N$-dependent $\b d$. The statements of Propositions \ref{prop: initial outliers} and \ref{prop: intial bulk} remain valid for all intermediate $\b d$'s; this will follow from the continuity of the eigenvalues of $\wt H$ as a function of $\b d$ and from Proposition \ref{prop: permissible region}. Roughly, Proposition \ref{prop: permissible region} establishes a forbidden region, for arbitrary $\b d$, which the eigenvalues of $\wt H$ cannot cross since they are deformed continuously.

Let $\b d(1) \equiv \b d_N(1) \in \cal D^*(\wt C_2)$ be given (and possibly $N$-dependent), with associated $N$-independent indices $k^-, k^0, k^+$. Choose an $N$-independent $\b d(0) \in \cal D(\wt C_2)$ with the same indices $k^-, k^0, k^+$, such that $\b d^0(0) = 0$ and $(\b d^-(0), \b d^+(0))$ satisfies \eqref{condition for initial data}. We shall use a bootstrap argument by choosing a continuous (possibly $N$-dependent) path $(\b d(t) \col 0 \leq t \leq 1)$ that connects $\b d(0)$ and $\b d(1)$. We require the path $\b d(t)$ to have the following properties.
\begin{enumerate}
\item
For all $t \in [0,1]$ the point $\b d(t)$ satisfies \eqref{ordering of ds} and  $\b d(t) \in \cal D(\wt C_2)$.
\item
If $I^+_i(\b d(1)) \cap I^+_j(\b d(1)) = \emptyset$ for a pair $1 \leq i < j \leq k^+$ then $I^+_i(\b d(t)) \cap I^+_j(\b d(t)) = \emptyset$ for all $t \in [0,1]$. The same restriction is imposed for $+$ replaced with $-$.
\end{enumerate}
It is easy to see that such a path exists. Informally, condition (ii) states that if the allowed regions for the outliers $i$ and $j$ do not over lap at time $t = 1$ (i.e.\ the outliers can be distinguished), then they may not overlap at any earlier time.

We continue to work at fixed $N$ and with a fixed realization $H \equiv H^\omega$ with $\omega \in \Xi$. Let $\wt C_2$ and $\wt C_3$ be the constants from Proposition \ref{prop: permissible region}, and choose $\delta > 0$ such that $\wt \psi \leq N^{1/3 - \delta}$. Define
\begin{equation*}
\wt H(t) \;\deq\; H + V \diag(d_1(t), \dots, d_k(t)) V^*
\end{equation*}
and abbreviate $\mu_\alpha(t) = \lambda_\alpha(\wt H(t))$.
By Propositions \ref{prop: initial outliers} and \ref{prop: intial bulk}, we have that
\begin{subequations} \label{init outliers}
\begin{align}
\mu_{N - k^+ + i}(0) &\;\in\; I^+_i(\b d(0)) \qquad (i = 1, \dots, k^+)\,,
\\
\mu_i(0) &\;\in\; I^-_i(\b d(0)) \qquad (i = 1, \dots, k^-)\,,
\end{align}
\end{subequations}
as well as
\begin{subequations} \label{init bulk}
\begin{align}
\lambda_\alpha \;\geq\; 2 - \varphi^{\wt K} N^{-2/3} \qquad &\Longrightarrow \qquad \abs{\lambda_\alpha - \mu_{\alpha - k^+}(0)} \;\leq\; N^{-2/3} \wt \psi^{-1}\,,
\\
\lambda_\alpha \;\leq\; -2 + \varphi^{\wt K} N^{-2/3} \qquad &\Longrightarrow \qquad \abs{\lambda_\alpha - \mu_{\alpha + k^-}(0)} \;\leq\; N^{-2/3} \wt \psi^{-1}\,.
\end{align}
\end{subequations}

In order to invoke a continuity argument, we note that Proposition \ref{prop: permissible region} yields
\begin{equation} \label{condition for continuity argument}
\sigma(\wt H(t)) \cap S(\wt K) \;\subset\; \Gamma(\b d(t))
\end{equation}
for all $t \in [0,1]$. Moreover, since $t \mapsto \wt H(t)$ is continuous, we find that $\mu_\alpha(t)$ is continuous in $t \in [0,1]$ for all $\alpha$.

Let us first analyse the outliers. We focus on the positive outliers associated with $\b d^+$; the negative ones are dealt with in the same way. Assume first that the $k^+$ intervals $I^+_1(\b d(t)), \dots, I^+_{k^+}(\b d(t))$ are disjoint for $t = 1$. Then, from Property (ii) above, we know that they are disjoint for all $t \in [0,1]$. Thus we find, from \eqref{init outliers}, \eqref{condition for continuity argument}, and the continuity of $t \mapsto \mu_\alpha(t)$ that
\begin{equation} \label{non-deg outliers}
\mu_{N - k^+ + i}(t) \;\in\; I^+_i(\b d(t)) \qquad (i = 1, \dots, k^+)
\end{equation}
for all $t \in [0,1]$, and in particular for $t = 1$.

If $I^+_1(\b d(1)), \dots, I^+_{k^+}(\b d(1))$ are not disjoint, the situation is only slightly more complicated. Let $\cal B$ denote the finest partition of $\{1, \dots, k^+\}$ such that $i$ and $j$ belong to the same block of $\cal B$ if $I_i^+(\b d(1)) \cap I_j^+(\b d(1)) \neq \emptyset$. Note that the blocks of $\cal B$ are sequences of consecutive integers. Denote by $B_i$ the block of $\cal B$ that contains $i$. Then \eqref{init outliers} and \eqref{condition for continuity argument} yield, instead of \eqref{non-deg outliers}, that
\begin{equation} \label{deg outliers}
\mu_{N - k^+ + i}(t) \;\in\; \bigcup_{j \in B_i} I_j^+(\b d(t)) \qquad (i = 1, \dots, k^+)
\end{equation}
for all $t \in [0,1]$. At $t = 1$, the right-hand side of \eqref{deg outliers} is an interval that contains $\theta(d_j)$ for all $j \in B_i$. In order to estimate its size, we pick a $j \in B_i$ that is not the largest element of $B_i$. To streamline notation, abbreviate $d \deq d^+_j(1)$ and $d' \deq d^+_{j + 1}(1)$. Our first task is to estimate $d' - d$. Since $I^+_j(\b d(1)) \cap I^+_{j + 1}(\b d(1)) \neq \emptyset$, we have
\begin{equation*}
\pbb{1 - \frac{1}{(d')^2}} (d' - d) \;\leq\;  \theta(d') - \theta(d) \;\leq\; 2 \varphi^{\wt C_3} N^{-1/2} (d' - 1)^{1/2}\,.
\end{equation*}
where the second inequality follows from the definition of $I_i^+(\cdot)$. This yields
\begin{equation*}
d' - d \;\leq\; C \varphi^{\wt C_3} N^{-1/2} (d' - 1)^{-1/2} \;\leq\; C \varphi^{\wt C_3} N^{-1/2} (d - 1)^{-1/2}\,,
\end{equation*}
where the constant $C$ depends only on $\Sigma$.
Thus we get
\begin{equation*}
(d' - 1)^{1/2} \;\leq\; (d - 1)^{1/2} \pbb{1 + \frac{d' - d}{d - 1}} \;\leq\; (d - 1)^{1/2} \pB{1 + C \varphi^{\wt C_3} N^{-1/2} (d - 1)^{-3/2}} \;\leq\; (d - 1)^{1/2} (1 + o(1))\,,
\end{equation*}
where the last inequality follows from \eqref{C_2 geq C_3}. Repeating this estimate of $\theta(d_{j + 1}^+(1)) - \theta(d_j^+(1))$ for the remaining $j \in B_i$, we find
\begin{equation*}
\diam \pBB{\bigcup_{j \in B_i} I_j^+(\b d(1))} \;\leq\; (2 \abs{B_i} + 2) \, \varphi^{\wt C_3} N^{-1/2} \min_{j \in B_i} (d_j^+(1) - 1)^{1/2} (1 + o(1))\,.
\end{equation*}
This immediately yields
\begin{equation*}
\abs{\mu_{N - k^+ + i}(1) - \theta(d^+_i)} \;\leq\; \varphi^{\wt C_3 + 1} N^{-1/2} (d_i^+(1) - 1)^{1/2} \qquad (i = 1, \dots, k^+)\,,
\end{equation*}
and the claim follows.

What remains is the analysis of the extremal bulk eigenvalues. Once again, we make use of a continuity argument. As before, we only consider positive eigenvalues, $\lambda_\alpha \geq 2 - \varphi^{\wt K} N^{-2/3}$ for some $\wt K$ to be chosen below. Note that by interlacing, Lemma \ref{lemma: interlacing}, we have
\begin{equation} \label{rank k interlacing}
\lambda_{\alpha - k} \;\leq\; \mu_\alpha \;\leq\; \lambda_{\alpha + k}
\end{equation}
(using the convention that $\lambda_{\alpha} = +\infty$ for $\alpha > N$). Recall the role of $K$ from the assumptions of Theorem \ref{theorem: rank-k lde}. Therefore using the definition of $\Xi$ (see Theorem \ref{theorem: rigidity}), we find that there is a $\wt K = \wt K(K)$ such that if $\alpha \geq N - \varphi^{K}$ then
\begin{equation*}
\lambda_{\alpha - k} \;\geq\; 2 - \varphi^{\wt K} N^{-2/3} \qquad \text{and} \qquad \mu_\alpha \;\geq\; 2 - \varphi^{\wt K} N^{-2/3}\,.
\end{equation*}
Let now $\alpha$ satisfy $N - \varphi^K \leq \alpha \leq N - k^+$. Using \eqref{init bulk}, \eqref{condition for continuity argument}, and Proposition \ref{prop: permissible region}, we find
\begin{equation} \label{bulk estimates for continuity}
\abs{\lambda_{\alpha + k^+} - \mu_{\alpha}(0)} \;\leq\; N^{-2/3} \wt \psi^{-1} \qquad \text{and} \qquad \dist(\mu_\alpha(t), \sigma(H)) \;\leq\; N^{-2/3} \wt \psi^{-1}
\end{equation}
for all $t \in [0,1]$. In addition, we know the two following facts about $\mu_\alpha(t)$, for all $t \in [0,1]$.
\begin{enumerate}
\item
$\mu_\alpha(t)$ is in the same connected component of $I^0 \subset \R$ as $\mu_\alpha(0)$ (by continuity of $\mu_\alpha(t)$ and Proposition \ref{prop: permissible region}).
\item
$\mu_\alpha(t)$ satisfies the interlacing bound \eqref{rank k interlacing} for all $t \in [0,1]$.
\end{enumerate}
Let $B_\alpha$ be the set of $\beta = 1, \dots, N$ such that $\lambda_{\beta}$ and $\lambda_\alpha$ are in the same connected component of $I^0$. Thus we conclude from (i) and (ii) that
\begin{equation*}
\mu_\alpha(t) \;\in\; \bigcup_{\substack{\beta \in B_{\alpha + k^+} \col \\ \abs{\alpha + k^+ - \beta} \leq k}} \qb{\lambda_\beta - N^{-2/3} \wt \psi^{-1} \,,\, \lambda_\beta + N^{-2/3} \wt \psi^{-1}}\,.
\end{equation*}
Thus we get
\begin{equation} \label{final sticking bound}
\abs{\lambda_{\alpha + k^+} - \mu_{\alpha}(t)} \;\leq\; 2k N^{-2/3} \wt \psi^{-1}
\end{equation}
for all $t \in [0,1]$.
Choosing
\begin{equation*}
C_2 \;\deq\; \wt C_2 + 1 \,, \qquad C_3 \;\deq\; \wt C_3 + 1
\end{equation*}
completes the proof of Theorem \ref{theorem: rank-k lde} (recall the definition \eqref{def of psi tilde}).

\section{Distribution of the outliers: proof of Theorem \ref{theorem: outlier distributions}} \label{section: ditribution of outl}

\subsection{Reduction to the law of $G_{\b v^{(i)} \b v^{(i)}}(\theta(d_i))$}
The following proposition reduces the problem to analysing a single explicit random variable.

\begin{proposition} \label{proposition: rank k distribution}
There is a constant $C_2$, depending on $\zeta$, such that the following holds.
Suppose that
\begin{equation*}
\abs{d_i} \;\leq\; \Sigma - 1 \,, \qquad \absb{\abs{d_i} - 1} \;\geq\; \varphi^{C_2} N^{-1/3}
\end{equation*}
for all $i = 1, \dots, k$. Suppose moreover that for all $i \in O$ \eqref{non-degeneracy condition} holds. Recall the definitions \eqref{set of outliers} and \eqref{indices of outliers}. Then we have for all $i \in O$
\begin{equation*}
N^{1/2} (\abs{d_i} - 1)^{-1/2} \pb{\mu_{\alpha(i)} - \theta(d_i)} \;=\; - (1 + O(\varphi^{-1})) (\abs{d_i} + 1) N^{1/2} (\abs{d_i} - 1)^{1/2} \pbb{G_{\b v^{(i)} \b v^{(i)}}(\theta(d_i)) + \frac{1}{d_i}} + O (\varphi^{-1})
\end{equation*}
\hp{\zeta}.
\end{proposition}

Before proving Proposition \ref{proposition: rank k distribution}, we record the following auxiliary result.

\begin{lemma} \label{lemma: derivative of Gvv}
Let $C_1$ denote the constant from Theorem \ref{theorem: strong estimate}. For any
\begin{equation*}
x \;\in\; \qb{-\Sigma, -2 - \varphi^{C_1} N^{-2/3}} \cup \qb{2 + \varphi^{C_1} N^{-2/3}, \Sigma}
\end{equation*}
and any normalized $\b v \in \C^N$ we have
\begin{equation} \label{derivative Gvv 1}
\absb{\partial_x G_{\b v \b v}(x) - \partial_x \msc(x)} \;\leq\; \varphi^{C_\zeta} N^{-1/3} \kappa_x^{-1}
\end{equation}
\hp{\zeta}. More generally, we have, for any normalized $\b v , \b w \in \C^N$,
\begin{equation} \label{derivative Gvv 2}
\absb{\partial_x G_{\b v \b w}(x) - \partial_x \msc(x) \scalar{\b v}{\b w}} \;\leq\; \varphi^{C_\zeta} N^{-1/3} \kappa_x^{-1}
\end{equation}
\hp{\zeta}.
\end{lemma}

\begin{proof}
By symmetry, we may assume that $x \geq 0$. Moreover, \eqref{derivative Gvv 2} follows from \eqref{derivative Gvv 1} and polarization.

We therefore prove \eqref{derivative Gvv 1} for $x \geq 0$.
We have
\begin{equation*}
\partial_x G_{\b v \b v}(x) \;=\; \sum_\alpha \frac{\abs{\scalar{\b u^{(\alpha)}}{\b v}}^2}{(\lambda_\alpha - x)^2}\,.
\end{equation*}
Choose $x \geq 2 + N^{-2/3} \varphi^{C_1}$ and abbreviate $\kappa \equiv \kappa_x$. Thus we get, for $\eta \geq \varphi^\zeta N^{-1}$,
\begin{align*}
\absbb{\partial_x G_{\b v \b v}(x) - \frac{1}{\eta} \im G_{\b v \b v}(x + \ii \eta)} &\;=\;\absBB{\sum_\alpha \frac{\abs{\scalar{\b u^{(\alpha)}}{\b v}}^2}{(\lambda_\alpha - x)^2} - \sum_\alpha \frac{\abs{\scalar{\b u^{(\alpha)}}{\b v}}^2}{(\lambda_\alpha - x)^2 + \eta^2}}
\\
&\;\leq\; \frac{\eta^2}{(x - \lambda_N)^2} \frac{1}{\eta} \im G_{\b v \b v}(x + \ii \eta)
\\
&\;\leq\; 2 \frac{\eta^2}{\kappa^2} \frac{1}{\eta} \im G_{\b v \b v}(x + \ii \eta)
\end{align*}
\hp{\zeta}, where in the last step we used Theorem \ref{theorem: rigidity}. (In the proof of Theorem \ref{theorem: strong estimate}, the constant $C_1$ was chosen large enough for this application of Theorem \ref{theorem: rigidity}; see \eqref{strong 4}.) A similar calculation using the definition \eqref{definition of msc} yields
\begin{equation*}
\absbb{\partial_x \msc(x) - \frac{1}{\eta} \im \msc(x + \ii \eta)} \;\leq\; \frac{\eta^2}{\kappa^2} \frac{1}{\eta} \im \msc(x + \ii \eta)\,.
\end{equation*}
Therefore we get, using Theorem \ref{theorem: strong estimate} and Lemma \ref{lemma: msc},
\begin{align*}
\absb{\partial_x G_{\b v \b v}(x) - \partial_x \msc(x)} &\;\leq\; \frac{2 \eta}{\kappa^2} \pB{\im G_{\b v \b v}(x + \ii \eta) + \im \msc(x + \ii \eta)} + \frac{1}{\eta} \varphi^{C_\zeta} \sqrt{\frac{\im \msc(x + \ii \eta)}{N \eta}}
\\
&\;\leq\; C \eta^2 \kappa^{-5/2} + \varphi^{C_\zeta} \pb{\eta \kappa^{-2} + \eta^{-1}} N^{-1/2} \kappa^{-1/4}
\end{align*}
\hp{\zeta}. Choosing $\eta \deq N^{-1/6} \kappa^{3/4}$ yields the claim.
\end{proof}

\begin{proof}[Proof of Proposition \ref{proposition: rank k distribution}]
We only prove the claim for the case $d_i > 1$; the case $d_i < -1$ is handled similarly.

For $2 + \varphi^{C_1} N^{-2/3} \leq x \leq \Sigma$, where $C_1$ is the constant from Theorem \ref{theorem: strong estimate}, we define the $k \times k$ Hermitian matrices $A(x)$ and $\wt A(x)$ through
\begin{equation*}
A_{ij}(x) \;\deq\; G_{\b v^{(i)} \b v^{(j)}}(x) - \msc(x) \delta_{ij} + d_i^{-1} \delta_{ij}\,,
\qquad \wt A_{ij}(x) \;\deq\; \delta_{ij} \pB{G_{\b v^{(i)} \b v^{(i)}}(x) - \msc(x) + d_i^{-1}}\,.
\end{equation*}
(Here we subtract $m(x) \umat$ so as to ensure that $\partial_x A(x)$ is well-behaved; see below.)
We denote the ordered eigenvalues of $A(x)$ and $\wt A(x)$ by $a_1(x) \leq \cdots \leq a_k(x)$ and $\wt a_1(x) \leq \cdots \leq \wt a_k(x)$ respectively.

For the rest of the proof we fix $i \in O$ satisfying $d_i > 1$. We abbreviate $\theta_i \deq \theta(d_i)$. We begin by comparing the eigenvalues of $\wt A(\theta_i)$ and $D^{-1}$. Define the eigenvalue index $r \equiv r(i) = 1, \dots, k$ through
\begin{equation} \label{def of wt a r}
\wt a_{r}(x) \;=\; \frac{1}{d_i} + G_{\b v^{(i)} \b v^{(i)}}(x) - \msc(x)\,.
\end{equation}
In particular,
\begin{equation*}
\wt a_{r}(\theta_i) \;=\; G_{\b v^{(i)} \b v^{(i)}}(\theta_i) + \frac{2}{d_i}\,.
\end{equation*}
Theorem \ref{theorem: strong estimate} implies that
\begin{equation} \label{error wt a d}
\absB{G_{\b v^{(j)} \b v^{(j)}}(\theta_i) - \msc(\theta_i)} \;\leq\; \varphi^{C_\zeta} N^{-1/2} (d_i - 1)^{-1/2}\,.
\end{equation}
\hp{\zeta} for $j = 1, \dots, k$. In particular,
\begin{equation*}
\absbb{\wt a_r(\theta_i) - \frac{1}{d_i}} \;\leq\; \varphi^{C_\zeta} N^{-1/2} (d_i - 1)^{-1/2}
\end{equation*}
\hp{\zeta}. Moreover, \eqref{error wt a d} and the condition \eqref{non-degeneracy condition} yield, for $j \neq i$,
\begin{equation} \label{perturbed non-degeneracy}
\absB{G_{\b v^{(j)} \b v^{(j)}}(\theta_i) - \msc(\theta_i)} \;\ll\; \abs{d_i - d_j}
\end{equation}
\hp{\zeta}, provided $C_2$ is chosen large enough. We therefore conclude that
\begin{equation} \label{non-deg for wt A}
\min_{j \neq r} \absb{\wt a_j(\theta_i) - \wt a_{r}(\theta_i)} \;\geq\; \varphi^{C_2 - 1} N^{-1/2} (d_i - 1)^{-1/2}
\end{equation}
\hp{\zeta}, provided $C_2$ is large enough.

Next, we compare the eigenvalues of $A(\theta_i)$ and $\wt A(\theta_i)$ using second-order perturbation theory (the first-order correction vanishes by definition of $\wt A$ and $A$). Theorem \ref{theorem: strong estimate} yields
\begin{equation*}
\norm{A(\theta_i) - \wt A(\theta_i)} \;\leq\; \varphi^{C_\zeta} N^{-1/2} (d_i - 1)^{-1/2}
\end{equation*}
\hp{\zeta}.
Therefore \eqref{non-deg for wt A} and nondegenerate second-order perturbation theory yield, for large enough $C_2$,
\begin{equation} \label{a a wt}
a_r(\theta_i) \;=\; \wt a_r(\theta_i) + O \pbb{\frac{\varphi^{C_\zeta} N^{-1} (d_i - 1)^{-1}}{\min_{j \neq r} \abs{\wt a_j(\theta_i) - \wt a_r(\theta_i)}}}
\;=\; \wt a_r(\theta_i) + O\pB{\varphi^{C_\zeta - C_2} N^{-1/2} (d_i - 1)^{-1/2}}
\end{equation}
\hp{\zeta}.

Next, we analyse $A(x)$ and make the link to $\mu_{\alpha(i)}$. From Lemma \ref{lemma: derivative of Gvv} we find
\begin{equation*}
\normb{\partial_x A(x)} \;\leq\; \varphi^{C_\zeta} N^{-1/3} \kappa_x^{-1}
\end{equation*}
\hp{\zeta}. In particular, we have for all $j = 1, \dots, k$ that
\begin{equation} \label{a's move slowly}
\absb{a_j(x) - a_j(y)} \;\leq\; \varphi^{C_\zeta} N^{-1/3} (\kappa_x^{-1} + \kappa_y^{-1}) \abs{x - y}
\end{equation}
\hp{\zeta}, provided that $2 + \varphi^{C_1} N^{-2/3} \leq x,y \leq \Sigma$.

Recall the definition \eqref{indices of outliers} of $\alpha(i)$.
From Lemma \ref{lemma: det identity} and Theorem \ref{theorem: rigidity}, we know that $\mu_{\alpha(i)}$ is  characterized by the property that there is a $q \equiv q(i) \in \{1, \dots, k\}$ such that
\begin{equation*}
a_{q}(\mu_{\alpha(i)}) \;=\; - \msc(\mu_{\alpha(i)})\,.
\end{equation*}
By Theorem \ref{theorem: rank-k lde} we have
\begin{equation} \label{apriori bound on mu alpha}
\abs{\mu_{\alpha(i)} - \theta_i} \;\leq\; \varphi^{C_3} N^{-1/2} (d_i - 1)^{1/2}
\end{equation}
\hp{\zeta}. Provided $C_2$ is large enough (depending on $C_3$), it is easy to see from \eqref{apriori bound on mu alpha} that
\begin{equation} \label{kappa sim}
\mu_{\alpha(i)} - 2 \;\asymp\; \theta_i - 2 \;\asymp\; (d_i - 1)^{2}
\end{equation}
\hp{\zeta}.
Thus we find, using \eqref{a's move slowly}, \eqref{apriori bound on mu alpha}, and \eqref{kappa sim}, that for large enough $C_2$ we have
\begin{equation} \label{m sc mu alpha i}
\msc(\mu_{\alpha(i)})  \;=\; - a_{q}(\theta_i) + O\pB{\varphi^{C_\zeta} N^{-5/6} (d_i - 1)^{-3/2}}
\end{equation}
\hp{\zeta}. (Here we absorbed the constant $C_3$ into $C_\zeta$.)

We now prove that $q = r$ \hp{\zeta} provided $C_2$ is large enough.  Assume by contradiction that $q \neq r$. Then we get, using Theorem \ref{theorem: strong estimate} and the condition \eqref{non-degeneracy condition}, that
\begin{equation} \label{separation of a's}
\absbb{a_q(\theta_i) - \frac{1}{d_i}} \;\geq\; \varphi^{C_2 - 1} N^{-1/2} (d_i - 1)^{-1/2}
\end{equation}
\hp{\zeta}. Moreover, \eqref{a's move slowly}, \eqref{apriori bound on mu alpha}, and \eqref{kappa sim} yield
\begin{align*}
a_q(\theta_i) &\;=\; a_q(\mu_{\alpha(i)}) + O \pB{\varphi^{C_\zeta} N^{-5/6} (d_i - 1)^{-3/2}}
\\
&\;=\; -\msc(\mu_{\alpha(i)}) + O \pB{\varphi^{C_\zeta} N^{-5/6} (d_i - 1)^{-3/2}}
\\
&\;=\; \frac{1}{d_i} + O \pB{\varphi^{C_\zeta} N^{-1/2} (d_i - 1)^{-1/2} + \varphi^{C_\zeta} N^{-5/6} (d_i - 1)^{-3/2}}
\end{align*}
\hp{\zeta}, where in the last step we used \eqref{msc prime}. Together with \eqref{separation of a's}, this yields the desired contradiction provided $C_2$ is large enough. Hence $q = r$.

Putting \eqref{def of wt a r}, \eqref{m sc mu alpha i}, and \eqref{a a wt} together, we get
\begin{equation*}
\msc(\mu_{\alpha(i)}) \;=\; - G_{\b v^{(i)} \b v^{(i)}}(\theta_i) - \frac{2}{d_i} + O \pB{\varphi^{C_\zeta} N^{-5/6} (d_i - 1)^{-3/2} + \varphi^{C_\zeta - C_2} N^{-1/2} (d_i - 1)^{-1/2}}
\end{equation*}
\hp{\zeta}. Thus we find that, for all $x$ between $\theta_i$ and $\mu_{\alpha(i)}$, we have
\begin{equation*}
\msc'(x) \;=\; \msc'(\theta_i) + O(\varphi^{C_3} N^{-1/2} (d_i - 1)^{-5/2}) \;=\; \msc'(\theta_i) (1 + O(\varphi^{-1}))
\end{equation*}
\hp{\zeta}, where we used \eqref{msc prime} and \eqref{apriori bound on mu alpha}. Using \eqref{identity for gamma}, \eqref{kappa sim}, and \eqref{msc prime}, we conclude that
\begin{equation*}
\mu_{\alpha(i)} - \theta_i \;=\; - (1 + O(\varphi^{-1})) \frac{G_{\b v^{(i)} \b v^{(i)}}(\theta_i) + d_i^{-1}}{\msc'(\theta_i)} + O \pB{\varphi^{C_\zeta} N^{-5/6} (d_i - 1)^{-1/2} + \varphi^{C_\zeta - C_2} N^{-1/2} (d_i - 1)^{1/2}}
\end{equation*}
\hp{\zeta}. The claim now follows for large enough $C_2$, using the identity \eqref{identity for gamma}.
\end{proof}

\subsection{The GOE/GUE case} \label{section: GOE/GUE}
By Proposition \ref{proposition: rank k distribution}, it is enough to analyse the random variable
\begin{equation}
X \;\deq\; N^{1/2} (\abs{d} + 1) (\abs{d} - 1)^{1/2} \pbb{G_{\b v \b v}(\theta) + \frac{1}{d}}\,,
\end{equation}
where $\b v \in \C^N$ is normalized, $d$ satisfies
\begin{equation} \label{conditions for d GUE}
1 + \varphi^{C_2} N^{-1/3} \;\leq\; \abs{d} \;\leq\; \Sigma - 1\,,
\end{equation}
and we abbreviated $\theta \;\equiv\; \theta(d)$. For definiteness, we choose $d > 1$ in the following.

The following notion of convergence of random variables is convenient for our needs.
\begin{definition} \label{definition: asymptotic distribution}
Two sequences of random variables, $\{A_N\}$ and $\{B_N\}$, are \emph{asymptotically equal in distribution}, denoted $A_N \simd B_N$, if they are tight and satisfy
\begin{equation} \label{conv in asympt dist}
\lim_{N \to \infty} \pb{\E f(A_N) - \E f(B_N)} \;=\; 0
\end{equation}
for all bounded and continuous $f$.
\end{definition}

\begin{remark}
Definition \ref{definition: asymptotic distribution} extends the notion of convergence in distribution, in the sense that $\E f(A_N)$ need not have a limit as $N \to \infty$.
\end{remark}

\begin{remark} \label{remark: cont to smooth}
In order to show that $A_N \simd B_N$, it suffices to establish the tightness of either $\{A_N\}$ or $\{B_N\}$ and to verify \eqref{conv in asympt dist} for all $f \in C_c^\infty(\R)$. Indeed, if $\{A_N\}$ is tight then so is $\{B_N\}$, by \eqref{conv in asympt dist}. By tightness of $A_N$ and $B_N$, we may replace in \eqref{conv in asympt dist} the bounded and continuous $f$ with a compactly supported continuous function $g$. Next, we can approximate $g$ uniformly with $C_c^\infty$-functions.
\end{remark}

\begin{remark}
Clearly, $A_N \simd B_N$ if $A_N \eqdist B_N$ for all $N$.
\end{remark}

\begin{lemma} \label{lemma: error in weak conv}
Let $A_N \simd B_N$ and $R_N$ satisfy $\lim_N \P(\abs{R_N} \leq \epsilon_N) = 1$, where $\{\epsilon_N\}$ is a positive null sequence. Then $A_N \simd B_N + R_N$.
\end{lemma}
\begin{proof}
By Remark \ref{remark: cont to smooth}, it suffices to prove \eqref{conv in asympt dist} for $f \in C^1(\R)$ such that $f$ and $f'$ are bounded. Then
\begin{align*}
\E f(A_N) - \E f(B_N + R_N) &\;=\; \pb{\E f(A_N) - \E f(B_N)} + \pb{\E f(B_N) - \E f(B_N + R_N)}
\\
&\;=\; o(1) + \E \qB{\ind{\abs{R_N} \leq \epsilon_N} \pb{f(B_N) - f(B_N + R_N)}}
\\
&\;=\; o(1)\,
\end{align*}
where in the last step we used the boundedness of $f'$.
\end{proof}

\begin{lemma} \label{lemma: sum for weak conv}
Let $\{A_N\}$, $\{A_N'\}$, $\{B_N\}$, and $\{B_N'\}$ be sequences of random variables. Suppose that $A_N \simd A_N'$, $B_N \simd B_N'$, $A_N$ and $B_N$ are independent, and $A_N'$ and $B_N'$ are independent. Then
\begin{equation*}
A_N + B_N \;\simd\; A_N' + B_N'\,.
\end{equation*}
\end{lemma}
\begin{proof}
Without loss of generality, we may assume that $A_N, B_N, A_N', B_N'$ are independent (after replacing $A_N'$ and $B_N'$ with new random variables without changing their laws.) Then for any $\lambda \in \R$ we have
\begin{align*}
\E\, \me^{\ii \lambda(A_N + B_N)} - \E\, \me^{\ii \lambda (A_N' + B_N')} &\;=\; \E \qB{\me^{\ii \lambda A_N} (\me^{\ii \lambda B_N} - \me^{\ii \lambda  B_N'}) + (\me^{\ii \lambda A_N} - \me^{\ii \lambda  A_N'}) \me^{\ii \lambda B_N'}}
\\
&\;=\;
\E\, \me^{\ii \lambda A_N} \, \E (\me^{\ii \lambda B_N} - \me^{\ii \lambda B_N'}) + \E (\me^{\ii \lambda A_N} - \me^{\ii \lambda A_N'}) \, \E\, \me^{\ii \lambda B_N'}
\\
&\;\longrightarrow\; 0
\end{align*}
as $N \to \infty$.

Next, we observe that $A_N + B_N$ and $A_N' + B_N'$ are tight. Therefore, recalling Remark \ref{remark: cont to smooth}, we find that it suffices to prove
\begin{equation*}
\E\, f(A_N + B_N) - \E\, f(A_N' + B_N') \;\longrightarrow\; 0
\end{equation*}
$f \in C_c^\infty$. Denoting by $\hat f$ the Fourier transform of $f$, we find
\begin{equation*}
\E\, f(A_N + B_N) - \E\, f(A_N' + B_N') \;=\; \int \dd \lambda \, \hat f(\lambda) \, \qB{\E\, \me^{\ii \lambda(A_N + B_N)} - \E\, \me^{\ii \lambda (A_N' + B_N')}} \;\longrightarrow\; 0
\end{equation*}
by dominated convergence.
\end{proof}

\begin{proposition} \label{GUE distribution}
Let $H$ be a GOE/GUE matrix. Assume that $d$ satisfies \eqref{conditions for d GUE}. Then for large enough $C_2$ we have
\begin{equation*}
X \;\simd\; \cal N \pbb{0, \frac{2 (d+1)}{\beta d^2}}\,.
\end{equation*}
\end{proposition}
\begin{proof}
By unitary invariance, we have $G_{\b v \b v}  \eqdist G_{11}$, where $\eqdist$ denotes equality in distribution. In order to handle the exceptional low-probability events, we add a small imaginary part to the spectral parameter $z \deq \theta + \ii N^{-4}$. Throughout the following we abbreviate $G \equiv G(z)$ and $\msc \equiv \msc(z)$. Writing $\b a^* \deq (h_{12}, h_{13}, \dots, h_{1N})$, we get from Schur's formula and \eqref{identity for msc} that
\begin{multline} \label{G11 expanded}
G_{11} \;=\; \frac{1}{h_{11} - z - \b a^*  G^{(1)} \b a} \;=\; \frac{1}{- \msc - z + h_{11} - \pb{\b a^* G^{(1)} \b a - \msc}}
\\
=\; \msc - \msc^2 h_{11} + \msc^2 \pb{\b a^* G^{(1)} \b a - \msc} + O(\abs{h_{11}}^2) + O \pB{\absb{\b a^* G^{(1)} \b a - \msc}^2}
\end{multline}
\hp{\zeta}. Again by unitary invariance, we have $\b a^* G^{(1)} \b a \eqdist \norm{\b a}^2 G^{(1)}_{22}$. Moreover, both sides are independent of $h_{11}$, so that
\begin{equation} \label{sum of indeps}
- \msc^2 h_{11} + \msc^2 \pb{\b a^* G^{(1)} \b a - \msc} \;\eqdist\; - \msc^2 h_{11} + \msc^2 \pb{\norm{\b a}^2 G^{(1)}_{22} - \msc}\,.
\end{equation}
In order to estimate the error term in \eqref{G11 expanded}, we write
\begin{equation} \label{decoupl of G and a}
\norm{\b a}^2 G_{22}^{(1)} - \msc \;=\; \pb{\norm{\b a}^2 -1 }G_{22}^{(1)} + (G_{22}^{(1)} - \msc)\,.
\end{equation}
Using \eqref{Gij Gijk} to estimate $G_{22}^{(1)} - G_{22}$, as well as Theorem \ref{theorem: strong estimate}, Lemma \ref{lemma: LDE}, and Lemma \ref{lemma: msc}, we therefore find that
\begin{equation} \label{G11 error GUE}
\absb{\norm{\b a}^2 G_{22}^{(1)} - \msc} \;\leq\; \varphi^{C_\zeta} N^{-1/2} (d - 1)^{-1/2}
\end{equation}
\hp{\zeta}. Moreover, we have the trivial bound $\E \absb{\norm{\b a}^2 G_{22}^{(1)} - \msc}^k \leq (kN)^{C k}$ for $k \in \N$.

From \eqref{G11 expanded}, \eqref{sum of indeps}, \eqref{decoupl of G and a}, and \eqref{G11 error GUE}, we conclude that there exist random variables $\wt R_1$ and $\wt R_2$ satisfying
\begin{equation} 
\abs{\wt R_1} + \abs{\wt R_2} \;\leq\; \varphi^{C_\zeta} N^{-1} (d - 1)^{-1}
\end{equation}
\hp{\zeta}, the rough bound
\begin{equation} 
\E (\abs{\wt R_1} + \abs{\wt R_2})^k \;\leq\; (kN)^{C k}\,,
\end{equation}
and
\begin{align*}
\pb{G_{11}^{(2)} - \msc} + \wt R_1 &\;=\;
- \msc^2 h_{11} + \msc^2 \pb{\b a^* G^{(1)} \b a - \msc}
\\
&\;\eqdist\; - \msc^2 h_{11} + \msc^2 \pb{\norm{\b a}^2 G^{(1)}_{22} - \msc}
\\
&\;=\; - \msc^2 h_{11} + \msc^3 \pb{\norm{\b a}^2 - 1} + \msc^2 (G_{22}^{(1)} - \msc) + \wt R_2\,.
\end{align*}

Defining
\begin{align*}
Y_1 &\;\deq\; N^{1/2} (d+1) (d - 1)^{1/2} \re \pb{G_{11}^{(2)} - \msc}\,, & Y_2 &\;\deq\; N^{1/2} (d+1) (d - 1)^{1/2} \re \pb{G_{22}^{(1)} - \msc}\,,
\\
W &\;\deq\; N^{1/2} \re \pB{- \msc^2 h_{11} + \msc^3 \pb{\norm{\b a}^2 - 1}} \,, & R_i &\;\deq\; N^{1/2} (d+1) (d - 1)^{1/2} \re \wt R_i \quad (i = 1,2)\,,
\end{align*}
we therefore get
\begin{equation} \label{main equation for Y}
Y_1 + R_1 \;\eqdist\; (d+1) (d - 1)^{1/2} W + \msc^2 Y_2 + R_2\,.
\end{equation}

In order to infer the distribution of $Y_1$ from \eqref{main equation for Y}, we observe that the random variables $Y_2$ and $W$ are independent. Also, $Y_1 \eqdist Y_2$. Recalling Theorem \ref{theorem: strong estimate} and \eqref{Gij Gijk}, we find the bounds
\begin{equation} \label{estimate of R 1}
\abs{Y_i} \;\leq\; \varphi^{C_\zeta}\,, \qquad \abs{R_i} \;\leq\; \varphi^{C_\zeta} N^{-1/2} (d - 1)^{-1/2} \qquad (i = 1,2)
\end{equation}
\hp{\zeta}, and the rough bounds
\begin{equation} \label{estimate of R 2}
\abs{Y_i} \;\leq\; N^2\,, \qquad \E \abs{R_i}^k \;\leq\; (kN)^{C k} \qquad (i = 1,2)\,.
\end{equation}
Moreover, by the Central Limit Theorem 
\begin{equation} \label{CLT for GUE 0}
\pbb{\frac{2 (d^2 + 1)}{\beta d^6}}^{-1} W \;\simd\; \cal N(0,1)\,,
\end{equation}
where we used \eqref{identity for gamma}.

Next, let $B$ and $Z_2$ be independent random variables whose laws are given by
\begin{equation*}
B \;\eqdist\; \cal N\pbb{0, \frac{2 (d^2 + 1)}{\beta d^6}}\,, \qquad
Z_2 \;\eqdist\; \cal N (0,\xi^2)\,,
\end{equation*}
where we introduced
\begin{equation*}
\xi^2 \;\equiv\; \xi_N^2 \;\deq\; d^4 \frac{(d^2 - 1)(d+1)}{d^4 - 1} \, \frac{2 (d^2 + 1)}{\beta d^6} \;=\; \frac{2 (d+1)}{\beta d^2}\,.
\end{equation*}
Defining
\begin{equation} \label{recursion for Z}
Z_1 \;\deq\; (d+1)(d - 1)^{1/2} B + d^{-2} Z_2\,,
\end{equation}
we find that $Z_1 \eqdist Z_2$. Moreover, a standard moment calculation and the definition of $W$ yield
\begin{equation} \label{CLT for GUE}
\lim_{N \to \infty} (\E W^k - \E B^k) \;=\; 0\,;
\end{equation}
as usual, only the pairings in the moment expansion of $\E W^k$ survive the limit $N \to \infty$.
(See also \eqref{CLT for GUE 0}, which however cannot be used to deduce \eqref{CLT for GUE} directly.)

We now compare the distributions of $Y_1$ and $Z_1$ by computing moments.
Note that the family $\{\E Z_1^k\}_{N \in \N}$ is bounded for each $k \in \N$.
We claim that
\begin{equation} \label{convergence of moments}
\lim_{N \to \infty} \pb{\E Y_1^k - \E Z_1^k} \;=\; 0
\end{equation}
for all $k \in \N$. (This will imply that $Y_1 \simd Z_1$.) We shall prove \eqref{convergence of moments} by induction on $k$. Taking the expectation of \eqref{main equation for Y} yields
\begin{equation*}
\E Y_1 \;=\; \msc^2 \E Y_1 + O\pB{\varphi^{C_\zeta} N^{-1/2} (d - 1)^{-1/2}}
\end{equation*}
where we used \eqref{estimate of R 1}, \eqref{estimate of R 2}, and $\E W = O(N^{-1/2})$.
Therefore
\begin{equation*}
\E Y_1 \;\leq\; C \varphi^{C_\zeta} N^{-1/2} (d - 1)^{-3/2} \;=\; o(1)
\end{equation*}
provided $C_2$ in \eqref{conditions for d GUE} is large enough. Here we used that
\begin{equation} \label{comp m_sc and d}
\msc(z) \;=\; d^{-1} + O(N^{-3})\,,
\end{equation}
as follows from the definition of $z = \theta + \ii N^{-4}$, \eqref{diff msc}, Lemma \ref{lemma: msc}, and \eqref{identity for gamma}. Therefore \eqref{convergence of moments} for $k = 1$ follows using $\E Z_1 = 0$.

For the induction step, we assume that \eqref{convergence of moments} holds for all $k' \leq k - 1$. From \eqref{main equation for Y} we find
\begin{multline} \label{expansion of moment equation}
\E Y_1^k + \sum_{l = 1}^k \binom{k}{l} \E \pb{R_1^l Y_1^{k - l}}
\\
=\; \E \pb{(d+1)(d - 1)^{1/2} W + \msc^2 Y_2}^k 
+ \sum_{l = 1}^k \binom{k}{l} \E \pB{R_1^l \pb{(d+1)(d - 1)^{1/2} W + \msc^2 Y_2}^{k - l}}\,.
\end{multline}
We estimate the summands on the left-hand side by
\begin{align*}
\absb{\E \pb{R_1^l Y_1^{k - l}}} &\;\leq\; N^C \exp(- \varphi^C) + \pB{\varphi^{C_\zeta} N^{-1/2} (d - 1)^{-1/2}}^l \E \abs{Y_1}^{k - l}
\\
&\;\leq\; C \pB{\varphi^{C_\zeta} N^{-1/2} (d - 1)^{-1/2}}^l \varphi^{C_\zeta}
\\
&\;\leq\; \varphi^{C_\zeta} N^{-1/2} (d - 1)^{-1/2}\,,
\end{align*}
where in the first step we used \eqref{estimate of R 1} and \eqref{estimate of R 2}, in the second step the estimate $\E \abs{Y_1}^{k - l} \leq \varphi^{C_\zeta}$ as follows from the induction assumption \eqref{convergence of moments} applied to even moments (recall that $Y_1$ is real) as well as \eqref{estimate of R 1} and \eqref{estimate of R 2}, and in the third step the fact that $l \geq 1$. Note that the constant $C_\zeta$ is independent of $k$. A similar estimate applies to the summands on the right-hand side of \eqref{expansion of moment equation}. Thus \eqref{expansion of moment equation} yields
\begin{align*}
\E Y_1^k &\;=\; \E \pb{(d+1)(d - 1)^{1/2} W + \msc^2 Y_2}^k + O \pb{\varphi^{C_\zeta} N^{-1/2} (d - 1)^{-1/2}}
\\
&\;=\; \msc^{2k} \, \E Y_1^k + \sum_{l = 2}^k \binom{k}{l} \E \pb{(d+1)(d - 1)^{1/2} W}^l \, \E \pb{\msc^2 Y_2}^{k - l} + O \pb{\varphi^{C_\zeta} N^{-1/2} (d - 1)^{-1/2}}\,,
\end{align*}
where in the second step we used the induction assumption and the estimate $\E W = O(N^{-1/2})$.
Therefore we get
\begin{equation} \label{main moment equation for Y}
\E Y_1^k \;=\; \frac{1}{1 - \msc^{2k}} \sum_{l = 2}^k \binom{k}{l} \E \pb{(d+1)(d - 1)^{1/2} W}^l \, \E \pb{\msc^2 Y_2}^{k - l} + O \pb{\varphi^{C_\zeta} N^{-1/2} (d - 1)^{-3/2}}\,,
\end{equation}
where we used \eqref{comp m_sc and d}.

In order to conclude the proof of \eqref{convergence of moments}, we deduce from \eqref{recursion for Z} that
\begin{equation} \label{main moment equation for Z}
\E Z_1^k \;=\; \frac{1}{1 - d^{-2k}} \sum_{l = 2}^k \binom{k}{l} \E \pb{(d+1)(d - 1)^{1/2} B}^l \, \E \pb{d^{-2} Z_2}^{k - l}\,.
\end{equation}
Using the induction assumption \eqref{convergence of moments} for $k' = k - l$, \eqref{comp m_sc and d}, and the condition $l \geq 2$, we get from \eqref{main moment equation for Y}, \eqref{main moment equation for Z}, and \eqref{CLT for GUE} that
\begin{equation*}
\lim_{N \to \infty} \pb{\E Y_1^k - \E Z_1^k} \;=\; 0
\end{equation*}
for large enough $C_2$. This concludes the proof of \eqref{convergence of moments}.

Next, by definition we have $\xi^{-1} Z_1 \eqdist \cal N(0,1)$. Moreover, we have that $\xi \in [c,C]$ for some  positive constants $c$ and $C$ depending only on $\Sigma$. Together with \eqref{convergence of moments} for $k = 2$, we infer that the families $\{\xi^{-1} Y_1\}_{N \in \N}$ and $\{\xi^{-1} Z_1\}_{N \in \N}$ are tight. Therefore we get from \eqref{convergence of moments} that
\begin{equation} \label{comparison Y Z}
\lim_{N \to \infty} \pb{\E f(\xi^{-1}Y_1) - \E f(\xi^{-1}Z_1)} \;=\; 0
\end{equation}
for any continuous bounded function $f$.
Next, we estimate
\begin{multline*}
\absb{G_{11}(\theta) - G_{11}^{(2)}(z)} \;\leq\; \absb{G_{11}(\theta) - G_{11}(z)} + \absb{G_{11}(z) - G_{11}^{(2)}(z)}
\\
\leq\; N^{-4} N^2 + \varphi^{C_\zeta} N^{-1} (d - 1)^{-1} \;\leq\; \varphi^{C_\zeta} N^{-1} (d - 1)^{-1}
\end{multline*}
\hp{\zeta},
where in the second step we used Lemma \ref{lemma: derivative of Gvv}, \eqref{diff msc}, and Lemma \ref{lemma: msc} to estimate the first term, and Theorem \ref{theorem: strong estimate} and \eqref{kappa sim d - 1} to estimate the second term. Therefore
\begin{equation*}
X \;\eqdist\; N^{1/2} (d+1)(d - 1)^{1/2} \pb{G_{11}(\theta) + d^{-1}} \;=\; Y_1 + O \pb{\varphi^{C_\zeta} N^{-1/2} (d - 1)^{-1/2}} \;=\; Y_1 + o(1)
\end{equation*}
\hp{\zeta}, where in the second step we used \eqref{comp m_sc and d}. Therefore \eqref{comparison Y Z}, the fact that $Z \eqdist Z_1$, and dominated convergence yield
\begin{equation} \label{weak conv for X}
\lim_{N \to \infty} \pb{\E f(\xi^{-1}X) - \E f(\xi^{-1}Z)} \;=\; 0\,.
\end{equation}
The claim now follows from Lemma \ref{lemma: weak conv trick} below.
\end{proof}

\begin{lemma} \label{lemma: weak conv trick}
Let $\{\xi_N\}$ be a bounded deterministic sequence. Let $A_\infty, A_1, A_2, \dots$ be random variables such that $A_N$ converges weakly to $A_\infty$. Then we have for any bounded continuous function $f$
\begin{equation*}
\E f(\xi_N A_N) - \E f(\xi_N A_\infty) \;\longrightarrow\; 0
\end{equation*}
as $N \to \infty$.
\end{lemma}
\begin{proof}
By Skorokhod's representation theorem, there exist new random variables $\wt A_\infty, \wt A_1, \wt A_2, \dots$ such that $A_\infty \eqdist \wt A_\infty$, $A_N \eqdist \wt A_N$ for all $N \in \N$, and $\wt A_N \to \wt A_\infty$ almost surely. Let $\omega$ be such that $\wt A_N(\omega) \to \wt A_\infty(\omega)$. By assumption on $\xi_N$, we find that there exists a $C \equiv C(\omega)$ such that $\xi_N \wt A_N(\omega) \in [-C, C]$ and $\xi_N \wt A_\infty(\omega) \in [-C, C]$ for all $N \in \N$. Since $f$ is uniformly continuous on $[-C,C]$, we find that
\begin{equation*}
\lim_{N \to \infty} \pB{f(\xi_N \wt A_N(\omega)) - f(\xi_N \wt A_\infty(\omega))} \;=\; 0\,.
\end{equation*}
The claim now follows by dominated convergence.
\end{proof}

\subsection{The almost-GOE/GUE case} \label{section: almost Gaussian}
As it turns out, replacing the matrix element $h_{ij}$ with a Gaussian in the Green function comparison step below (Section \ref{sec: Green function comparison}) is only possible if $\abs{v_i} \leq \varphi^{-D}$ and $\abs{v_i} \leq \varphi^{-D}$, for some large enough constant $D > 0$. If this assumption is not satisfied, we first have to replace $h_{ij}$ with a Gaussian using a different method, which effectively keeps track of the fluctuations of $G_{\b v \b v}$ resulting from large components of $\b v$.
Thus we shall proceed in two steps:
\begin{enumerate}
\item
We compare the original Wigner matrix $H$ with $\wh H$, a Wigner matrix obtained from $H$ by replacing the $(i,j)$-th entry of $H$ with a Gaussian whenever $\abs{v_i} \leq \varphi^{-D}$ and $\abs{v_j} \leq \varphi^{-D}$. 
\item
We compare the matrix $\wh H$ to a Gaussian matrix.
\end{enumerate}
The step (ii) is performed in this section. To simplify notation, we write $H$ instead of $\wh H$ throughout this section. The step (i) is performed using Green function comparison in Section \ref{sec: Green function comparison} below.

The following shorthand will prove useful.
\begin{definition} \label{definition: sum of indep G}
Let $\{\sigma_N\}$ be a bounded positive sequence. If $A_N$ and $B_N$ are independent random variables with $B_N \simd \cal N(0, \sigma_N^2)$, and if $S_N \simd A_N + B_N$, then we write
\begin{equation*}
S_N \;\simd\; A_N + \cal N(0, \sigma_N^2)\,.
\end{equation*}
\end{definition}

For the following we write
\begin{equation*}
X \;=\; \nu N^{1/2} \pbb{G_{\b v \b v}(\theta) + \frac{1}{d}}\,, \qquad \nu \equiv \nu_N \;\deq\; (d+1) (d - 1)^{1/2}\,.
\end{equation*}

\begin{proposition} \label{proposition: almost-GUE}
Fix $D > 0$. Let $\b v \in \C^N$ be normalized and $H$ be a Wigner matrix such that if $\abs{v_i} \leq \varphi^{-D}$ and $\abs{v_j} \leq \varphi^{-D}$ then $h_{ij}$ is Gaussian. Then we have
\begin{equation*}
X \;\simd\; -\nu N^{1/2} d^{-2} \scalar{\b v}{H \b v}
+ \cal N \pBB{0 \,,\, \frac{2 (d+1)}{\beta d^4} + \frac{4 \nu^2 Q(\b v)}{d^5}
+ \frac{\nu^2 R(\b v)}{d^6}}\,,
\end{equation*}
where $Q(\b v)$ and $R(\b v)$ were defined in \eqref{definition of QRS}.
\end{proposition}

\begin{proof}
As before, we consistently drop the spectral parameter $z = \theta$ from our notation.

Let $M \in \N$ denote the number of entries of $\b v$ satisfying $\abs{v_i} > \varphi^{-D}$. Since $\b v$ is normalized, we have $M \leq \varphi^{2D}$.
To simplify notation, we assume (after a suitable permutation of the rows and columns of $H$) that the entries of $\b v$ satisfy $\abs{v_i} > \varphi^{-D}$ for $i \leq M$ and $\abs{v_i} \leq \varphi^{-D}$ for $i > M$. Split $\b v = \binom{\b u}{\b w}$, where $\b u \in \C^M$ and $\b w \in \C^{N - M}$. (Throughout the following we assume that $\b w \neq 0$; the case $\b w = 0$ may be easily handled by approximation with nonzero $\b w$.) We also split
\begin{equation*}
H \;=\;
\begin{pmatrix}
A & B^*
\\
B & H_0
\end{pmatrix}\,,
\end{equation*}
where $A$ is an $M \times M$ matrix and $H_0$ an $(N - M) \times (N - M)$ matrix with Gaussian entries. Choose a deterministic orthogonal/unitary $(N - M) \times (N - M)$ matrix $S$ such that $S \b w = (\norm{\b w}, 0, \dots, 0)^*$. Thus we get
\begin{align*}
G_{\b v \b v} &\;=\; \b v^*
\begin{pmatrix}
\umat & 0
\\
0 & S^*
\end{pmatrix}
\begin{pmatrix}
\umat & 0
\\
0 & S
\end{pmatrix}
\begin{pmatrix}
A - z & B^*
\\
B & H_0 - z
\end{pmatrix}^{-1}
\begin{pmatrix}
\umat & 0
\\
0 & S^*
\end{pmatrix}
\begin{pmatrix}
\umat & 0
\\
0 & S
\end{pmatrix}
\b v
\\
&\;\eqdist\;
\begin{pmatrix}
\b u \\ \norm{\b w} \\ 0
\end{pmatrix}^{\!\! *}
\begin{pmatrix}
A - z & B^* S^*
\\
SB & H_0 - z
\end{pmatrix}^{-1}
\begin{pmatrix}
\b u \\ \norm{\b w} \\ 0
\end{pmatrix}\,,
\end{align*}
where we used that $S H_0 S^* \eqdist H_0$ and the fact that $A$, $B$, and $H_0$ are independent.

Next, we split
\begin{equation*}
S \;=\;
\begin{pmatrix}
\b w^* / \norm{\b w}        
\\
\wt S
\end{pmatrix}\,, \qquad
H_0 \;=\;
\begin{pmatrix}
g & \b a^*
\\
\b a & H_1
\end{pmatrix}\,,
\end{equation*}
where $\b a \in \C^{N - M -1}$ is a vector of i.i.d.\ Gaussians. Note that $\wt S^*$ is an isometry, i.e.\ $\wt S \wt S^* = \umat$.
Thus we may write
\begin{align}
G_{\b v \b v} &\;\eqdist\;
\begin{pmatrix}
\b u \\ \norm{\b w} \\ 0
\end{pmatrix}^{\!\! *}
\begin{pmatrix}
A - z & B^* \b w / \norm{\b w} & B^* \wt S^*
\\
\b w^* B / \norm{\b w} & g - z & \b a^*
\\
\wt S B & \b a & H_1 - z
\end{pmatrix}^{-1}
\begin{pmatrix}
\b u \\ \norm{\b w} \\ 0
\end{pmatrix}
\notag \\ 
&\;\eqd\;
\begin{pmatrix}
\b x \\ 0
\end{pmatrix}^{\!\! *}
\begin{pmatrix}
E - z & F^*
\\
F & H_1 - z
\end{pmatrix}^{-1}
\begin{pmatrix}
\b x \\ 0
\end{pmatrix}
\notag \\ \label{Gvv in terms of x}
&\;\eqd\; \Gamma\,,
\end{align}
where the second equality defines the right-hand side using self-explanatory notation. Note that, by definition, $\norm{\b x} = \norm{\b v} = 1$.

Next, we claim that
\begin{equation} \label{estimate on F star F}
(F^* F)_{ij} \;=\; \delta_{ij} + O \pb{\varphi^{C_\zeta} N^{-1/2}}
\end{equation}
\hp{\zeta}. In order to prove \eqref{estimate on F star F}, write
\begin{equation*}
F^* F \;=\;
\begin{pmatrix}
B^* \wt S^* \wt S B & B^* \wt S^* \b a
\\
\b a^* \wt S B & \b a^* \b a
\end{pmatrix}\,.
\end{equation*}
We consider four cases. First, if $1 \leq i \neq j \leq M$ we find using \eqref{two-set LDE} that
\begin{equation*}
\absb{(F^* F)_{ij}} \;=\; \absBB{ \sum_{k,l} B^*_{ik} (\wt S^* \wt S)_{kl} \wt B_{lj} } \;\leq\; \frac{\varphi^{C_\zeta}}{N} \pBB{\sum_{k,l} \absb{(\wt S^* \wt S)_{kl}}^2}^{1/2}
\;=\; \frac{\varphi^{C_\zeta}}{N} \pB{\tr (\wt S^* \wt S)^2}^{1/2} \;\leq\; \varphi^{C_\zeta} N^{-1/2}
\end{equation*}
\hp{\zeta}.
Second, if $1 \leq i \leq M$ we find using \eqref{diag LDE} and \eqref{offdiag LDE} that
\begin{equation*}
\absb{(F^* F)_{ii} - 1} \;=\; \absBB{\sum_{k,l} B^*_{ik} (\wt S^* \wt S)_{kl} B_{li} - 1} \;\leq\; \absbb{\frac{1}{N} \sum_i (\wt S^* \wt S)_{ii} - 1} + \frac{\varphi^{C_\zeta}}{N} \pBB{\sum_{k,l} \absb{(\wt S^* \wt S)_{kl}}^2}^{1/2} \;\leq\; \varphi^{C_\zeta} N^{-1/2}
\end{equation*}
\hp{\zeta}. Third, for $i = M + 1$ we have by \eqref{diag LDE}
\begin{equation*}
\absb{(F^* F)_{ii} - 1} \;=\; \abs{\b a^* \b a - 1} \;\leq\; \varphi^{C_\zeta} N^{-1/2}
\end{equation*}
\hp{\zeta}.
Finally, for $1 \leq i < j = M+1$ we have by \eqref{two-set LDE}
\begin{equation*}
\absb{(F^* F)_{ij}} \;=\; \absBB{\sum_{k,l} B^*_{ik} \wt S^*_{kl} a_l} \;\leq\;  \frac{\varphi^{C_\zeta}}{N} \pBB{\sum_{k,l} \abs{\wt S^*_{kl}}^2}^{1/2} \;=\; \frac{\varphi^{C_\zeta}}{N} \pB{\tr \wt S^* \wt S}^{1/2} \;\leq\; \varphi^{C_\zeta} N^{-1/2}
\end{equation*}
\hp{\zeta}. This completes the proof of \eqref{estimate on F star F}.

Next, abbreviate $G_1(z) \deq (H_1 - z)^{-1}$. Since $N^{1/2} (N - M - 1)^{-1/2} H_1$ is an $(N - M - 1) \times (N - M - 1)$ GOE/GUE matrix, we find from \eqref{estimate on F star F}, Theorem \ref{theorem: strong estimate}, and Lemma \ref{lemma: msc} that
\begin{equation} \label{error for Schur}
\absB{\pb{F^* G_1 F}_{ij} - \delta_{ij} \msc} \;\leq\; \varphi^{C_\zeta} N^{-1/2} (d - 1)^{-1/2}
\end{equation}
\hp{\zeta}. Therefore Schur's formula yields
\begin{align}
\Gamma &\;=\; \b x^* \pB{- z - \msc - \pb{-E + F^* G_1 F - \msc}}^{-1} \b x
\notag \\
&\;=\; \msc \norm{\b x}^2 - \msc^2 \scalar{\b x}{E \b x} + \msc^2 \pB{\scalar{F \b x}{G_1 F \b x} - \msc \norm{\b x}^2}
+ O \pB{\varphi^{C_\zeta} N^{-1} (d - 1)^{-1}}\,.
\end{align}
\hp{\zeta}, where in the second step we expanded using \eqref{identity for msc}, and estimated the error term using \eqref{error for Schur} as well as the bounds $M \leq \varphi^{C_\zeta}$ and $\abs{E_{ij}} \leq \varphi^{C_\zeta} N^{-1/2}$\,. Recalling that $\norm{\b x} = 1$, we find
\begin{multline} \label{x F x computation 2}
\Gamma - \msc
\;=\; - \msc^2
\begin{pmatrix}
\b u \\ \norm{\b w}
\end{pmatrix}^{\!\! *}
\begin{pmatrix}
A & B^* \b w / \norm{\b w}
\\
\b w^* B / \norm{\b w} & g
\end{pmatrix}
\begin{pmatrix}
\b u \\ \norm{\b w}
\end{pmatrix}
+ \msc^2 \norm{F \b x}^2 \pbb{\frac{1}{\norm{F \b x}^2} \scalar{F \b x}{G_1 F \b x} - \msc}
\\
+ \msc^3 \pb{\norm{F \b x}^2 - 1} + O \pB{\varphi^{C_\zeta} N^{-1} (d - 1)^{-1}}
\end{multline}
\hp{\zeta}.

Next, from $F \b x = \wt S B \b u + \norm{\b w} \b a$ we find
\begin{align*}
\norm{F \b x}^2 &\;=\; \scalarb{B \b u}{\wt S^* \wt S  B \b u} + 2 \norm{\b w} \re \scalarb{B \b u}{\wt S^* \b a} + \norm{\b w}^2 \norm{\b a}^2
\\
&\;=\; \scalarb{B \b u}{B \b u} - \absb{\scalar{\b w}{B \b u}}^2 + 2 \norm{\b w} \re \scalarb{B \b u}{\wt S^* \b a} + \norm{\b w}^2 \norm{\b a}^2\,.
\end{align*}
Applying \eqref{LDE} to $\scalar{\b w}{B \b u} = \sum_{i,j} \ol w_i u_j B_{ij}$ (with $N$ in \eqref{LDE} replaced by $M (N - M)$), we find
\begin{equation*}
\absb{\scalar{\b w}{B \b u}}^2 \;\leq\; \varphi^{C_\zeta} N^{-1}\,.
\end{equation*}
Similarly, using \eqref{diag LDE} and \eqref{offdiag LDE} we find that
\begin{equation*}
\norm{B \b u}^2 \;=\; \norm{\b u}^2 + O \pb{\varphi^{C_\zeta} N^{-1/2}}\,, \qquad \norm{\wt S B \b u}^2 \;=\; \norm{\b u}^2 + O \pb{\varphi^{C_\zeta} N^{-1/2}}
\end{equation*}
\hp{\zeta},
using \eqref{two-set LDE} that
\begin{equation*}
\absb{\scalarb{B \b u}{\wt S^* \b a}} \;\leq\; \varphi^{C_\zeta} N^{-1/2}
\end{equation*}
\hp{\zeta},
and using \eqref{diag LDE} that
\begin{equation*}
\norm{\b a}^2 \;=\; 1 + O\pb{\varphi^{C_\zeta} N^{-1/2}}
\end{equation*}
\hp{\zeta}.
Using $\norm{\b u} \geq \varphi^{-D}$ (by definition of $\b u$), we therefore conclude that
\begin{equation} \label{F x computed}
\norm{F \b x}^2 \;=\; \norm{B \b u}^2 + 2 \re \frac{\norm{\b u} \norm{\b w}}{\norm{\wt S B \b u} \norm{\b a}} \scalarb{\wt S B \b u}{\b a} + \norm{\b w}^2 \norm{\b a}^2 + O \pb{\varphi^{C_\zeta} N^{-1}}
\;=\; 1 + O \pb{\varphi^{C_\zeta} N^{-1/2}}
\end{equation}
\hp{\zeta}.
Using Theorem \ref{theorem: strong estimate} applied to $G_1$ (recall that $F$ and $H_1$ are independent), we therefore get from \eqref{x F x computation 2} that
\begin{multline}
\Gamma - \msc \;=\; - \msc^2 \pB{\scalar{\b u}{A \b u} + \norm{\b w}^2 g + 2 \re \scalar{\b w}{B \b u}} + \msc^2 \pBB{\frac{1}{\norm{F \b x}^2} \scalar{F \b x}{G_1 F \b x} - \msc}
\\
+ \msc^3 \pBB{ \norm{B \b u}^2 - \norm{\b u}^2 + 2 \re \frac{\norm{\b u} \norm{\b w}}{\norm{\wt S B \b u} \norm{\b a}} \scalarb{\wt S B \b u}{\b a} + \norm{\b w}^2 \pb{\norm{\b a}^2 - 1}} + O \pB{\varphi^{C_\zeta} N^{-1} (d - 1)^{-1}}
\end{multline}
\hp{\zeta}. We write this as
\begin{equation} \label{Gamma decomposed}
\Gamma - \msc \;=\; \Gamma_1 + \cdots + \Gamma_6 + O\pB{\varphi^{C_\zeta} N^{-1} (d - 1)^{-1}}
\end{equation}
\hp{\zeta}, where
\begin{gather*}
\Gamma_1 \;\deq\; - \msc^2 \scalar{\b u}{A \b u}\,, \qquad
\Gamma_2 \;\deq\; - \msc^2 \norm{\b w}^2 g\,, \qquad
\Gamma_3 \;\deq\; \msc^2 \pBB{\frac{1}{\norm{F \b x}^2} \scalar{F \b x}{G_1 F \b x} - \msc}\,,
\\
\Gamma_4 \;\deq\; - 2 \msc^2 \re \scalar{\b w}{B \b u} + \msc^3 \pb{\norm{B \b u}^2 - \norm{\b u}^2} \,, \qquad
\Gamma_5 \;\deq\; 2 \msc^3 \re \frac{\norm{\b u} \norm{\b w}}{\norm{\wt S B \b u} \norm{\b a}} \scalarb{\wt S B \b u}{\b a}\,,
\\
\Gamma_6 \;\deq\; \msc^3 \norm{\b w}^2 \pb{\norm{\b a}^2 - 1}\,.
\end{gather*}

We now claim that $\Gamma_1, \dots, \Gamma_6$ are independent. In order to prove this, let $f_1, \dots, f_6$ be indicator functions of Borel sets in $\R$. Write $\b a = a \b \omega$ in polar coordinates, where $a > 0$ and $\b \omega \in S^{N - M - 2}$. Since $\b a$ is Gaussian, $a$ and $\b \omega$ are independent. Denote by $\rho_1, \dots, \rho_6$ the laws of $A, B, g, a, \b \omega, H_1$ respectively. Then we get
\begin{align*}
\E \prod_{i = 1}^6 f_i(\Gamma_i) &\;=\; \int \dd \rho_1(A) \, \dd \rho_2(B) \, \dd \rho_3(d) \, \dd \rho_4(a) \, \dd \rho_5(\b \omega) \, \dd \rho_6(H_1) \, \prod_{i = 1}^6 f_i(\Gamma_i)
\\
&\;=\; \pb{\E f_1(\Gamma_1)} \pb{\E f_2(\Gamma_2)} \pb{\E f_6(\Gamma_6)} \int \dd \rho_2(B) \, \dd \rho_5(\b \omega) \, \dd \rho_6(H_1) \, f_3(\Gamma_3) f_4(\Gamma_4) f_5(\Gamma_5)
\\
&\;=\; \pb{\E f_1(\Gamma_1)} \pb{\E f_2(\Gamma_2)} \pb{\E f_6(\Gamma_6)} \pb{\E f_3(\Gamma_3)} \pb{\E f_5(\Gamma_5)} \int \dd \rho_2(B)  f_4(\Gamma_4)
\\
&\;=\; \prod_{i = 1}^6 \E f_i(\Gamma_i)\,,
\end{align*}
where the second equality follows by definition of the $\Gamma$'s, and the third from the invariance of the law of $\b \omega$ under rotations (applied to $\Gamma_5$) and from the invariance of the law of $H_1$ under orthogonal/unitary conjugations (applied to $\Gamma_3$). This proves the independence of $\Gamma_1, \dots, \Gamma_6$.

Next, we identify the asymptotic laws of $\Gamma_1, \dots, \Gamma_6$. There is nothing to be done with $\Gamma_1$. By definition,
\begin{equation} \label{dist Gamma 2}
\nu N^{1/2} \Gamma_2 \;\eqdist\; \cal N \pb{0, 2 \nu^2 \beta^{-1} \msc^4 \norm{\b w}^4 }\,.
\end{equation}
Since $F \b x$ is independent of $H_1$ and $M \leq \varphi^{2D}$, we get from Proposition \ref{GUE distribution} that
\begin{equation} \label{dist Gamma 3}
\nu N^{1/2} \Gamma_3 \;\simd\; \cal N\pbb{0, \msc^4 \frac{2 (d+1)}{\beta d^2}}\,.
\end{equation}
In order to analyse $\Gamma_4$, we define $b_i \deq (B \b u)_i$ for $i = 1, \dots, N - M$. Then $\{b_i\}_i$ are independent and satisfy
\begin{equation*}
\E b_i \;=\; 0\,, \qquad \E \abs{b_i}^2 = \frac{1}{N} \norm{\b u}^2\,, \qquad \E \abs{b_i}^4 \;=\; \frac{4 - \beta}{N^2} \norm{\b u}^4 + \frac{1}{N^2} \sum_j \pb{M^{(4)}_{ij} - 4 + \beta} \abs{u_j}^4\,.
\end{equation*}
Thus we find
\begin{equation*}
\Gamma_4 \;=\; \sum_i \pB{- 2 \msc^2 \re \ol w_i b_i + \msc^3 \pb{\abs{b_i}^2 - \E \abs{b_i}^2}} + O (M/N)\,.
\end{equation*}
The variance of the term in parentheses is
\begin{align*}
&\mspace{-10mu} \E \pB{- 2 \msc^2 \re \ol w_i b_i + \msc^3 \pb{\abs{b_i}^2 - \E \abs{b_i}^2}}^2
\\
&\;=\; 4 \msc^4 \E (\re \ol w_i b_i)^2 - 4 \msc^5 \E \re \pb{(\ol w_i b_i) \abs{b_i}^2} + \msc^6 \E \pb{\abs{b_i}^2 - \E \abs{b_i}^2}^2
\\
&\;=\; 4 \msc^4 \beta^{-1} N^{-1} \norm{\b u}^2 \abs{w_i}^2 - 4 \msc^5 N^{-3/2} \re \pbb{\ol w_i \sum_j M^{(3)}_{ij} u_j \abs{u_j}^2}
\\
&\mspace{40mu}
+ \msc^6 N^{-2} \pbb{(3 - \beta) \norm{\b u}^4 + \sum_j \pb{M^{(4)}_{ij} - 4 + \beta} \abs{u_j}^4}\,.
\end{align*}
Since $\abs{w_i} \leq \varphi^{-D}$, we get from the Central Limit Theorem and Lemma \ref{lemma: weak conv trick} that
\begin{equation} \label{dist Gamma 4}
\nu N^{1/2} \Gamma_4
\;\simd\; \cal N \pbb{0, \nu^2 \frac{4 \msc^4}{\beta} \norm{\b u}^2 \norm{\b w}^2 - 4 \nu^2 \msc^5 Q(\b w, \b u) + \nu^2 \msc^6 \pB{2 \beta^{-1} \norm{\b u}^4 +R(\b u)}}\,,
\end{equation}
where we abbreviated
\begin{equation*}
Q(\b w, \b u) \;\deq\; N^{-1/2} \re \sum_{i,j} \ol w_i M^{(3)}_{ij} u_j \abs{u_j}^2\,, \qquad R(\b u) \;\deq\; \frac{1}{N} \sum_{i,j} \pb{M^{(4)}_{ij} - 4 + \beta} \abs{u_j}^4\,,
\end{equation*}
and used that $3 - \beta = 2 \beta^{-1}$ for $\beta = 1,2$. Since $\norm{\b u} \leq 1$ and $\norm{\b w} \leq 1$, we find that $Q(\b w, \b u) \leq C$ and $R(\b u) \leq C$ for some positive constant $C$. Next, using $\Gamma_5 = 2 \msc^3 \norm{\b w} \re \scalar{\wt S B \b u}{\b a} + O(\varphi^{C_\zeta} N^{-1})$ \hp{\zeta} and
\begin{equation*}
\E \pb{2 \re \scalar{\wt S B \b u}{\b a}}^2 \;=\; \frac{4}{\beta N^{2}} (N - M - 1) \norm{\b u}^2\,,
\end{equation*}
we find from the Central Limit Theorem and Lemma \ref{lemma: weak conv trick} that
\begin{equation} \label{dist Gamma 5}
\nu N^{1/2} \Gamma_5 \;\simd\; \cal N \pb{0, 4 \nu^2 \beta^{-1} \msc^6 \norm{\b u}^2 \norm{\b w}^2}\,.
\end{equation}
Finally, we have $\norm{\b a}^2 - 1 = \norm{\b a}^2 - \E \norm{\b a}^2 + O(M/N)$ and
\begin{equation*}
\E \pb{\abs{a_i}^2 - \E \abs{a_i}^2}^2 \;=\; 2 \beta^{-1} N^{-2}\,.
\end{equation*}
Thus we conclude from the Central Limit Theorem and Lemma \ref{lemma: weak conv trick} that
\begin{equation} \label{dist Gamma 6}
\nu N^{1/2} \Gamma_6 \;\simd\; \cal N \pb{0, 2 \nu^2 \beta^{-1} \msc^6 \norm{\b w}^4}\,.
\end{equation}

Next, \eqref{dist Gamma 2} -- \eqref{dist Gamma 6} imply that $\nu N^{1/2} \Gamma_2, \dots, \nu N^{1/2} \Gamma_6$ are tight (as $N$-dependent random variables). Moreover, an easy variance calculation shows that $\nu N^{1/2} \Gamma_1$ is also tight. Therefore we get from \eqref{Gvv in terms of x}, \eqref{Gamma decomposed}, \eqref{dist Gamma 2} -- \eqref{dist Gamma 6}, Lemma \ref{lemma: error in weak conv}, and Lemma \ref{lemma: sum for weak conv} that (recall the notation from Definition \ref{definition: sum of indep G})
\begin{equation*}
X \;\simd\; - \nu N^{1/2} \msc^2 \scalar{\b u}{A \b u} + \cal N(0, V_1)\,,
\end{equation*}
where
\begin{multline*}
V_1 \;\deq\; \frac{2 (d+1)}{\beta d^6} + \frac{2 \nu^2}{\beta d^4} \pb{\norm{\b w}^4 + 2 \norm{\b u}^2 \norm{\b w}^2} + \frac{4 \nu^2}{d^5} Q(\b w, \b u)
\\
+ \frac{\nu^2}{d^6} R(\b u)
+ \frac{2 \nu^2}{\beta d^6} \pB{\norm{\b u}^4 + 2 \norm{\b u}^2 \norm{\b w}^2 + \norm{\b w}^4}\,.
\end{multline*}
Here we used \eqref{identity for gamma}.

Next, from
\begin{equation*}
\scalar{\b v}{H \b v} \;=\; \scalar{\b u}{A \b u} + 2 \re \scalar{\b w}{B \b u} + \scalar{\b w}{H_0 \b w}\,,
\end{equation*}
the Central Limit Theorem, Lemma \ref{lemma: weak conv trick}, and Lemma \ref{lemma: sum for weak conv} we find
\begin{equation} \label{replace A with H}
\nu N^{1/2} \scalar{\b v}{H \b v} \;\simd\; \nu N^{1/2} \scalar{\b u}{A \b u} + \cal N \pbb{0, \frac{2 \nu^2}{\beta} (1 - \norm{\b u}^4)}\,.
\end{equation}
Moreover, using that the dimension $M$ of $\b u$ satisfies $M \leq \varphi^{2D}$ and the fact that $\max_i \abs{w_i} \leq \varphi^{-D}$, we find
\begin{equation*}
Q (\b w, \b u)  \;=\; Q(\b v) + O(\varphi^{-D}) \,, \qquad R(\b u) \;=\; R(\b v) + O(\varphi^{-2D})\,.
\end{equation*}
Therefore we get, using Lemma \ref{lemma: sum for weak conv} and recalling that $1 = \norm{\b v}^2 = \norm{\b u}^2 + \norm{\b w}^2$,
\begin{equation*}
X \;\simd\; -\nu N^{1/2} d^{-2} \scalar{\b v}{H \b v} + \cal N\pbb{0\,,\, \frac{2 (d+1)}{\beta d^6} + \frac{4 \nu^2}{d^5} Q(\b v)
+ \frac{\nu^2}{d^6} R(\b v)
+ \frac{2 \nu^2}{\beta d^6}}\,.
\end{equation*}
This concludes the proof.
\end{proof}

\subsection{Conclusion of the proof of Theorem \ref{theorem: outlier distributions}} \label{sec: Green function comparison}
In this section we compute the distribution of $G_{\b v \b v}(\theta) - \msc(\theta)$ for a general Wigner matrix $H$, and hence complete the proof of Theorem \ref{theorem: outlier distributions}. We use the Green function comparison method from the proof of Lemma \ref{lemma: three moment bound}.

Let $H = (h_{ij}) = (N^{-1/2} W_{ij})$ be an arbitrary real symmetric / Hermitian Wigner matrix, $V = (N^{-1/2} V_{ij})$ a GOE/GUE matrix independent of $H$, and $\b v \in \C^N$ be normalized. For $D > 0$ define the subset
\begin{equation*}
I_D \;\deq\; \hb{i = 1, \dots, N \col \abs{v_i} \leq \varphi^{-D}}\,.
\end{equation*}
Define a new Wigner matrix $\wh H = (\wh h_{ij}) = (N^{-1/2} \wh W_{ij})$ through
\begin{equation*}
\wh W_{ij} \;\deq\;
\begin{cases}
V_{ij} & \text{if } i \in I_D \text{ and } j \in I_D
\\
W_{ij} & \text{otherwise}\,.
\end{cases}
\end{equation*}
Thus, $\wh H$ satisfies the assumptions of Proposition \ref{proposition: almost-GUE}. Let
\begin{equation*}
J_D \;\deq\; \hb{1 \leq i \leq j \leq N \col i \in I_D \text{ and } j \in I_D}
\end{equation*}
be the set of matrix indices to be replaced. Similarly to \eqref{ordering map}, we choose a bijective map $\phi \col J_D \to \{1, \dots, \gamma_{\rm max}(D)\}$ and denote by $H_\gamma = (h_{ij}^\gamma)$ the matrix defined by

\begin{equation*}
h_{ij}^\gamma \;\deq\;
\begin{cases}
N^{-1/2} W_{ij} & \text{if } \phi(i,j)\le \gamma
\\
N^{-1/2} \wh W_{ij} & \text{otherwise}\,.
\end{cases}
\end{equation*}
In particular, $H_0 = \wh H$ and $H_{\gamma_{\rm max}(D)} = H$. Let now $(a,b) \in J_D$ satisfy $\phi(a,b) = \gamma$. Similarly to \eqref{defHg1}, we write
\begin{equation*}
H_{\gamma-1} \;=\; Q + N^{-1/2} V \where V \;\deq\; V_{ab} E^{(ab)} + \ind{a \neq b} V_{ba} E^{(ba)}\,,
\end{equation*}
and
\begin{equation*}
H_\gamma \;=\; Q + N^{-1/2} W \where W \;\deq\; W_{ab} E^{(ab)}+ \ind{a \neq b} W_{ba} E^{(ba)}\,.
\end{equation*}

In order to avoid singular behaviour on exceptional low-probability events, we add a small imaginary part to the spectral parameter $\theta$, and set $z \deq \theta + \ii N^{-4}$.
Abbreviate
\begin{equation} \label{definition of x}
x \;\deq\; \nu N^{1/2} \re (G_{\b v \b v}(z) - \msc(z))\,.
\end{equation}
Thus we have the rough bound $\abs{x} \leq N^4$ which we shall tacitly use in the following. We use the notation \eqref{defG}, which gives rise to the quantities $x_R, x_S, x_T$ defined through \eqref{definition of x} with $G$ replaced by $R,S,T$ respectively. We may now state the main comparison estimate.

\begin{lemma} \label{lemma: Green function comp with shift}
Provided $D$ is a large enough constant, the following holds.
Let $f \in C^3(\R)$ be bounded with bounded derivatives and $q \equiv q_N$ be an arbitrary deterministic real sequence. Then
\begin{align} \label{comparison for T}
\E f(x_T + q) &\;=\; \E f(x_R + q) + Y_{ab} \E f'(x_R + q) + A_{ab} + O \pb{\varphi^{-1} \wh {\cal E}_{ab}}\,,
\\ \label{comparison for S}
\E f(x_S + q) &\;=\; \E f(x_R + q) + A_{ab} + O \pb{\varphi^{-1} \wh {\cal E}_{ab}}\,,
\end{align}
where $A_{ab}$ satisfies $\abs{A_{ab}} \leq \varphi^{-1}$,
\begin{equation*}
Y_{ab} \;\deq\; - \nu N^{-1} \re \pB{\msc^4 M^{(3)}_{ab} \ol v_a v_b + \msc^4 M^{(3)}_{ba} \ol v_b v_a}\,,
\end{equation*}
and
\begin{equation*}
\wh {\cal E}_{ab} \;\deq\; \sum_{\sigma, \tau = 0}^2 N^{-2 + \sigma/2 + \tau/2} \abs{v_a}^\sigma \abs{v_b}^\tau + \delta_{ab} \sum_{\sigma = 0}^2 N^{-1 + \sigma/2} \abs{v_a}^\sigma\,.
\end{equation*}
\end{lemma}

Before proving Lemma \ref{lemma: Green function comp with shift}, we show how it implies Theorem \ref{theorem: outlier distributions}.

\begin{proof}[Proof of Theorem \ref{theorem: outlier distributions}]
Fix $D > 0$ large enough that the conclusion of Lemma \ref{lemma: Green function comp with shift} holds. By Remark \ref{remark: cont to smooth}, we may assume that $f \in C_c^\infty(\R)$. Let $\gamma = \phi(a,b)$. Since $\abs{v_a} \leq \varphi^{-D}$ and $\abs{v_b} \leq \varphi^{-D}$, we find
\begin{equation} \label{Y squared}
Y_{ab}^2 \;\leq\; \varphi^{-1} \wh {\cal E}_{ab}\,.
\end{equation}
Applying \eqref{comparison for T} and \eqref{comparison for S} with $f$ replaced by $f'$ yields
\begin{equation*}
Y_{ab} \E f'(x_T + q) \;=\; Y_{ab} \E f'(x_R + q) + Y_{ab} A_{ab} + O \pb{\varphi^{-1} \wh {\cal E}_{ab}}\,.
\end{equation*}
Subtracting this from \eqref{comparison for T} and using $\abs{A_{ab}} \leq \varphi^{-1}$ yields
\begin{equation*}
\E f(x_T + q) \;=\; \E f(x_R + q) + Y_{ab} \E f'(x_T + q) + A_{ab} + O \pb{\varphi^{-1} \wh {\cal E}_{ab} + \varphi^{-1} \abs{Y_{ab}}}\,.
\end{equation*}
Subtracting \eqref{comparison for S} yields
\begin{equation*}
\E f(x_\gamma + q) \;=\; \E f(x_{\gamma - 1} + q) + Y_{ab} \E f'(x_\gamma + q) + O \pb{\varphi^{-1} \wh {\cal E}_{ab} + \varphi^{-1} \abs{Y_{ab}}}\,,
\end{equation*}
where we introduced the notation $x_\gamma \deq \nu N^{1/2} \re \pb{ (H_\gamma - z)^{-1}_{\b v \b v} - \msc(z)}$. Using \eqref{Y squared} we therefore get
\begin{equation} \label{iterable identity}
\E f(x_\gamma + q - Y_{ab}) \;=\; \E f(x_{\gamma - 1} + q) + O \pb{\varphi^{-1} \wh {\cal E}_{ab} + \varphi^{-1} \abs{Y_{ab}}}\,.
\end{equation}
We now iterate \eqref{iterable identity}, starting at $\gamma = 1$ and $q = 0$. Using that $\sum_{a,b} \wh {\cal E}_{ab} \leq C$ and $\sum_{a,b} \abs{Y_{ab}} \leq C$, we find after $\gamma_{\rm max}$ iterations of \eqref{iterable identity}
\begin{equation*}
\E f \pBB{x_{\gamma_{\rm max}(D)} - \sum_{\gamma = 1}^{\gamma_{\rm max}(D)} Y_{\varphi^{-1}(\gamma)}} \;=\; \E f(x_0) + O(\varphi^{-1})\,.
\end{equation*}
Moreover, using $\abs{v_a} \leq \varphi^{-D}$ and $\abs{v_b} \leq \varphi^{-D}$, we find that
\begin{align*}
\sum_{\gamma = 1}^{\gamma_{\rm max}(D)} Y_{\varphi^{-1}(\gamma)} &\;=\; - \nu N^{-1} \re \sum_{a , b \in I_D} \ind{a \leq b} \msc(z)^4 \pB{M^{(3)}_{ab} \ol v_a v_b + M^{(3)}_{ba} \ol v_b v_a}
\\
&\;=\; - \nu N^{-1} \re \sum_{a , b  = 1}^N \msc(z)^4 M^{(3)}_{ab} \ol v_a v_b + O(\varphi^{-2D})\,.
\end{align*}
Using Lemma \ref{lemma: sum for weak conv} we find
\begin{equation*}
\nu N^{1/2} \pB{(H - z)^{-1}_{\b v \b v} - \msc(z)} \;\simd\; \nu N^{1/2} \pB{(\wh H - z)^{-1}_{\b v \b v} - \msc(z)} - \nu N^{-1} \re \sum_{a , b  = 1}^N \msc(z)^4 M^{(3)}_{ab} \ol v_a v_b\,.
\end{equation*}
Using Lemma \ref{lemma: derivative of Gvv}, it is now easy to remove the imaginary part $N^{-4}$ of $z$ to get
\begin{equation*}
\nu N^{1/2} \pB{(H - \theta)^{-1}_{\b v \b v} + d^{-1}} \;\simd\; \nu N^{1/2} \pB{(\wh H - \theta)^{-1}_{\b v \b v} + d^{-1}} - \frac{\nu S(\b v)} {d^4}\,.
\end{equation*}
Since $\wh H$ satisfies the assumptions of Proposition \ref{proposition: almost-GUE}, we find
\begin{equation*}
\nu N^{1/2} \pB{(H - \theta)^{-1}_{\b v \b v} + d^{-1}} \;\simd\; - \frac{\nu N^{1/2} \scalar{\b v}{H \b v}} {d^2} 
- \frac{\nu S(\b v)} {d^4}
+ \cal N \pBB{0 \,,\, \frac{2 (d+1)}{\beta d^4} + \frac{4 \nu^2 Q(\b v)}{d^5}
+ \frac{\nu^2 R(\b v)}{d^6}}\,,
\end{equation*}
using the notation of Definition \ref{definition: sum of indep G}. Now Theorem \ref{theorem: outlier distributions} follows from Proposition \ref{proposition: rank k distribution} and Lemma \ref{lemma: error in weak conv}.
\end{proof}

\begin{proof}[Proof of Lemma \ref{lemma: Green function comp with shift}]
As before, we consistently drop the spectral parameter $z = \theta + \ii N^{-4}$ from $G$ and $\msc$. We focus on \eqref{comparison for T}. From Theorem \ref{theorem: strong estimate}, \eqref{Rva estimate}, and \eqref{Sva Rva} (with $S$ replaced by $T$), we find
\begin{equation} \label{Tva and Rva estimate}
\abs{T_{\b v a}} \;\leq\; \varphi^{C_\zeta} N^{-1/2} (d - 1)^{-1/2} + C \abs{v_a} \,, \qquad \abs{R_{\b v a}} \;\leq\; \varphi^{C_\zeta} N^{-1/2} (d - 1)^{-1/2} + C \abs{v_a} + \varphi^{C_\zeta} N^{-1/2} \abs{v_b}
\end{equation}
\hp{\zeta},
and similar results hold for $T_{a \b v}$, $T_{\b v b}$, $T_{b \b v}$, $R_{a \b v}$, $R_{\b v b}$, and $R_{b \b v}$. Similarly, from the first inequality of \eqref{Svv Rvv} (with $S$ replaced by $T$), we get
\begin{equation*}
\absb{T_{\b v \b v} - R_{\b v \b v}} \;\leq\; \varphi^{C_\zeta} N^{-1/2} \pB{N^{-1} (d - 1)^{-1} + N^{-1/2} (d - 1)^{-1/2}(\abs{v_a} + \abs{v_b}) + \abs{v_a} \abs{v_b} + N^{-1/2} (\abs{v_a}^2 + \abs{v_b}^2)}
\end{equation*}
\hp{\zeta}.
This yields
\begin{equation} \label{xT - xR}
\abs{x_T - x_R} \;\leq\; \varphi^{\wt C_\zeta} \qB{N^{-1} (d - 1)^{-1/2} + N^{-1/2} (\abs{v_a} + \abs{v_b}) + (d - 1)^{1/2} \abs{v_a} \abs{v_b} + N^{-1/2} (d - 1)^{1/2} (\abs{v_a}^2 + \abs{v_b}^2)}
\end{equation}
\hp{\zeta} for some constant $\wt C_\zeta$. Now choose $D \geq \wt C_\zeta + 1$. By definition of $J_D$, we have that $\abs{v_a} \leq \varphi^{-D}$ and $\abs{v_b} \leq \varphi^{-D}$. Therefore
\begin{equation*}
\abs{x_T - x_R}^3 \;\leq\; \varphi^{-1} \wh {\cal E}_{ab}
\end{equation*}
\hp{\zeta}.
This yields
\begin{equation} \label{main Taylor expansion}
\E f(x_T + q) \;=\; \E f(x_R + q) + \E \pb{f'(x_R + q) (x_T - x_R)} + \frac{1}{2} \E \pb{f''(x_R + q)(x_T - x_R)^2} + O \pb{\varphi^{-1} \wh {\cal E}_{ab}}\,.
\end{equation}

In order to analyse $x_T - x_R = \nu N^{1/2} \re (T_{\b v \b v} - R_{\b v \b v})$, we write
\begin{equation*}
x_T - x_R \;=\; y_1 + y_2 + y_3 + y_4\,,
\end{equation*}
where
\begin{equation*}
\qquad y_k \;\deq\;
\begin{cases}
\nu N^{1/2 - k / 2} \re \pb{(-R W)^k R}_{\b v \b v} & \text{if } k = 1,2,3
\\
\nu N^{-3/2} \re \pb{(-R W)^4 T}_{\b v \b v} & \text{if } k = 4\,.
\end{cases}
\end{equation*}
Using \eqref{Tva and Rva estimate}, it is easy to check that $y_1$ is bounded by the right-hand side of \eqref{xT - xR}, and that
\begin{equation} \label{yk bound}
\abs{y_k} \;\leq\; \varphi^{C_\zeta} N^{1/2 - k/2} \pB{N^{-1} (d - 1)^{-1/2} + (d - 1)^{1/2} \pb{\abs{v_a}^2 + \abs{v_b}^2}} \qquad (k = 2, 3, 4)
\end{equation}
\hp{\zeta}. In particular,
\begin{equation*}
x_T - x_R \;=\; y_1 + y_2 + y_3 + O\pb{\varphi^{-1} \wh {\cal E}_{ab}}
\end{equation*}
\hp{\zeta}. Moreover, using, $\abs{v_a} \leq \varphi^{-D}$, $\abs{v_b} \leq \varphi^{-D}$, \eqref{yk bound} for $k = 2$, and the fact that $y_1$ is bounded by the right-hand side of \eqref{xT - xR}, we find that
\begin{equation*}
\abs{y_1} \abs{y_2} \;\leq\; \varphi^{-1} \wh {\cal E}_{ab}
\end{equation*}
\hp{\zeta},
provided $D$ is chosen large enough. Similarly, using \eqref{yk bound} we find that $\abs{y_k} \abs{y_{k'}} \leq \varphi^{-1} \wh {\cal E}_{ab}$ for $k,k' \geq 2$ for large enough $D$. Thus we conclude from \eqref{main Taylor expansion} that
\begin{equation*}
\E f(x_T + q) \;=\; \E f(x_R + q) + \E \pb{f'(x_R + q) y_3} + A_{ab} + O\pb{\varphi^{-1} \wh {\cal E}_{ab}}\,,
\end{equation*}
where
\begin{equation*}
A_{ab} \;\deq\; \E \pb{ f'(x_R + q) (y_1 + y_2)} + \frac{1}{2} \E \pb{f''(x_R + q) y_1^2}
\end{equation*}
depends on the randomness only through $R$ and the first two moments of $W_{ab}$. Moreover, from \eqref{yk bound} and the fact that $y_1$ is bounded by the right-hand side of \eqref{xT - xR}, we conclude that $\abs{A_{ab}} \leq \varphi^{-1}$.

What remains is the analysis of the term $\E \pb{f'(x_R + q) y_3}$. We shall prove that
\begin{equation} \label{shift computation}
\absB{\E \pb{f'(x_R + q) y_3} - Y_{ab} \, \E f'(x_R + q)} \;\leq\; C \varphi^{-1} \wh {\cal E}_{ab}\,.
\end{equation}
If $a = b$, it is easy to see from \eqref{yk bound} and the definition of $Y_{ab}$ that
\begin{equation*}
\abs{y_3} + \abs{Y_{ab}} \;\leq\; \varphi^{-1} \wh{\cal E}_{ab}\,,
\end{equation*}
from which \eqref{shift computation} follows.

Let us therefore assume that $a \neq b$. We multiply out the matrix product in $\pb{(-RW)^3 R}_{\b v \b v}$ and regroup the resulting eight terms according to the number, $r$, of off-diagonal matrix elements ($R_{ab}$ or $R_{ba}$) of $R$. (By convention, the endpoint matrix elements $R_{\b v \cdot}$ and $R_{\cdot \b v}$ are not counted as off-diagonal.) This gives, in self-explanatory notation, $y_3 = \sum_{r = 0}^2 y_{3,r}$. Using Theorem \ref{theorem: strong estimate} and \eqref{Tva and Rva estimate}, we find
\begin{equation*}
\abs{y_{3,1}} + \abs{y_{3,2}} \;\leq\; \varphi^{C_\zeta} N^{-3/2} \pB{N^{-1/2} (d - 1)^{-1/2} + \abs{v_a} + \abs{v_b}}^2 \;\leq\; \varphi^{-1} \wh {\cal E}_{ab}
\end{equation*}
\hp{\zeta}\,. Therefore it suffices to prove that
\begin{equation} \label{shift computation 2}
\absB{\E \pb{f'(x_R + q) y_{3,0}} - Y_{ab} \, \E f'(x_R + q)} \;\leq\; C \varphi^{-1} \wh {\cal E}_{ab}
\end{equation}
for $a \neq b$. By definition,
\begin{align*}
y_{3,0} &\;=\; - \nu N^{-1} \re \pB{R_{\b v a} W_{ab} R_{bb} W_{ba} R_{aa} W_{ab} R_{b \b v} + R_{\b v b} W_{ba} R_{aa} W_{ab} R_{bb} W_{ba} R_{a \b v}}\,.
\end{align*}
Using \eqref{Tva and Rva estimate} and Theorem \ref{theorem: strong estimate} we find
\begin{equation*}
\absB{y_{3,0} + \nu N^{-1} \re \pB{\msc^2 \abs{W_{ab}}^2 \pb{R_{\b v a} W_{ab} R_{b \b v} + R_{\b v b} W_{ba} R_{a \b v}}}} \;\leq\; \varphi^{-1} \wh {\cal E}_{ab}
\end{equation*}
\hp{\zeta}. We only deal with the first term of $y_{3,0}$; the second one is dealt with analogously. Recalling the definition of $Y_{ab}$, we conclude that, in order to establish \eqref{shift computation 2}, it suffices to prove
\begin{equation} \label{res estimate step 6}
\absbb{\E \pbb{f'(x_R + q) \nu N^{-1} \re \pB{\msc^2 M^{(3)}_{ab} R_{\b v a} R_{b \b v} - \msc^4 M^{(3)}_{ab} \ol v_a v_b}}} \;\leq\; C \varphi^{-1} \wh {\cal E}_{ab}
\end{equation}
\hp{\zeta}; here we used that $R$ is independent of $W_{ab}$.

Setting $\b u = (u_i)$ with $u_i \deq \ind{i \notin \{a,b\}} v_i$ and recalling \eqref{def G prime} and \eqref{G G prime}, we get
\begin{align} \label{exp ab of R}
R_{\b v a} \;=\; \ol v_a R_{aa} + \ol v_b R_{ba} + R_{\b u a}
\;=\; \ol v_a \msc + \ol v_a (R_{aa} - \msc) + \ol v_b R_{ba} + \msc \cal R_{\b u a} + (R_{aa} - \msc) \cal R_{\b u a}\,,
\end{align}
where we defined
\begin{equation*}
\cal R_{\b u a} \;\deq\; - \sum_{i}^{(a)} R^{(a)}_{\b u i} h_{ia}\,, \qquad \cal R_{b \b u} \;\deq\; - \sum_i^{(b)} h_{b i} R^{(b)}_{i \b u}\,;
\end{equation*}
see \eqref{def R prime R double prime}.
The second and third terms are estimated using \eqref{Tva and Rva estimate} and Theorem \ref{theorem: strong estimate}:
\begin{equation} \label{res exp estimate 1}
\abs{R_{aa} - \msc} + \abs{R_{ba}} \;\leq\; \varphi^{C_\zeta} N^{-1/2} (d - 1)^{-1/2}
\end{equation}
\hp{\zeta}.
Moreover, since $R^{(a)} = T^{(a)}$, we find from Lemma \eqref{LDE}, Theorem \ref{theorem: strong estimate}, and \eqref{Gvw minor} that
\begin{multline} \label{res exp estimate 2}
\absb{\cal R_{\b u a}} \;\leq\; \varphi^{C_\zeta} \pBB{\frac{1}{N} \sum_{i}^{(a)} \absb{T_{\b u i}^{(a)}}^2}^{1/2}
\\
\leq\;
\varphi^{C_\zeta} \pBB{\frac{1}{N} \sum_i \pB{\abs{u_i}^2 + N^{-1} (d - 1)^{-1}}}^{1/2} \;\leq\; \varphi^{C_\zeta} N^{-1/2} (d - 1)^{-1/2}
\end{multline}
\hp{\zeta}. A similar estimate holds for $\cal R_{b \b u}$. Using \eqref{exp ab of R}, \eqref{res exp estimate 1}, \eqref{res exp estimate 2}, and \eqref{Tva and Rva estimate} we get
\begin{multline} \label{estimate of remaining m3 terms}
\nu N^{-1} \absB{\E \pB{f'(x_R + q) \pb{R_{\b v a} R_{b \b v} - \msc^2  \ol v_a v_b}}}
\\
\leq\; \nu N^{-1} \absB{\E \qB{f'(x_R + q) \pB{\msc^2 v_b \cal R_{\b u a} + \msc^2 \ol v_a \cal R_{b \b u} + \msc^2 \cal R_{\b u a} \cal R_{b \b u} } }} + C \varphi^{-1} \wh {\cal E}_{ab}
\end{multline}
\hp{\zeta}.

What remains is to estimate the right-hand side of \eqref{estimate of remaining m3 terms}. Defining
\begin{equation*}
x_R^{(a)} \;\deq\; \nu N^{1/2} \re (R^{(a)}_{\b v \b v} - \msc)\,,
\end{equation*}
we find from \eqref{Gvw minor} and \eqref{Tva and Rva estimate} that
\begin{equation*}
\absb{x_R - x_R^{(a)}} \;\leq\; \varphi^{C_\zeta} \pB{N^{-1/2} (d - 1)^{-1/2} + N^{1/2} (d - 1)^{1/2} \abs{v_a}^2 + N^{-1/2} (d - 1)^{1/2} \abs{v_b}^2}
\end{equation*}
\hp{\zeta}. Using \eqref{res exp estimate 2} and using that the derivative of $f$ is bounded, we may estimate the first term of \eqref{estimate of remaining m3 terms} as
\begin{equation*}
\nu N^{-1} \absB{\E \qB{f'(x_R + q) v_b \cal R_{\b u a}}} \;\leq\; \nu N^{-1} \absB{\E \qB{f'(x_R^{(a)} + q) v_b \cal R_{\b u a}}} + C \varphi^{-1} \wh {\cal E}_{ab} \;=\; C \varphi^{-1} \wh {\cal E}_{ab}
\end{equation*}
\hp{\zeta}. In the second step we used that $x_R^{(a)}$ is independent of the the $a$-th column of $Q$ and that $\E_a \cal R_{\b u a} = 0$. The second term of \eqref{estimate of remaining m3 terms} is similar. In order to estimate the third, we have to make $\cal R_{b \b u}$ independent of the $a$-th column of $Q$.
(See the definition \eqref{defHg1}.)
We estimate, using \eqref{Gvw minor}, $R^{(b)} = T^{(b)}$, \eqref{LDE}, and \eqref{Tva and Rva estimate}
\begin{align*}
\absBB{\cal R_{b \b u} + \sum_i^{(ab)} h_{b i} R^{(ab)}_{i \b u}} &\;\leq\; \absb{h_{ba} T_{a \b u}^{(b)}} + \absBB{\sum_i^{(ab)} h_{bi} \frac{T_{ia}^{(b)} T_{a \b u}^{(b)}}{T_{aa}^{(b)}}}
\\
&\;\leq\; \varphi^{C_\zeta} \pb{N^{-1} (d - 1)^{-1/2} + N^{-1/2} \abs{v_a}} + \varphi^{C_\zeta} N^{-1/2} \pBB{\sum_i^{(ab)} \absB{T_{ia}^{(b)} T_{a \b u}^{(b)} / T_{aa}^{(b)}}^2}^{1/2}
\\
&\;\leq\; \varphi^{C_\zeta} \pb{N^{-1} (d - 1)^{-1} + N^{-1/2} (d - 1)^{-1/2} \abs{v_a}}
\end{align*}
\hp{\zeta}. Thus we may estimate the third term of \eqref{estimate of remaining m3 terms} by
\begin{equation*}
\nu N^{-1} \absB{\E \qB{f'(x_R + q) \cal R_{\b u a} \cal R_{b \b u}}} \;\leq\; \nu N^{-1} \absBB{\E \qBB{f'(x_R^{(a)} + q) \cal R_{\b u a} \sum_i^{(ab)} h_{b i} R^{(ab)}_{i \b u} }} + C \varphi^{-1} \wh {\cal E}_{ab} \;=\; C \varphi^{-1} \wh {\cal E}_{ab}\,,
\end{equation*}
where in the second step we again used that $\E_a \cal R_{\b u a} = 0$. This concludes the proof of \eqref{res estimate step 6}, and hence of \eqref{comparison for T}.

The proof of \eqref{comparison for S} is almost identical to the proof of \eqref{comparison for T}, except that $\E \abs{V_{ab}}^2 V_{ab} = 0$, so that the left-hand side of the analogue of \eqref{res estimate step 6} vanishes. Note that, by definition, $A_{ab}$ depends only on $R$ and on the first two moments of $W_{ab}$, which coincide with those of $V_{ab}$. Hence $A_{ab}$ is the same in \eqref{comparison for T} and \eqref{comparison for S}. This concludes the proof.
\end{proof}

\providecommand{\bysame}{\leavevmode\hbox to3em{\hrulefill}\thinspace}
\providecommand{\MR}{\relax\ifhmode\unskip\space\fi MR }
\providecommand{\MRhref}[2]{%
  \href{http://www.ams.org/mathscinet-getitem?mr=#1}{#2}
}
\providecommand{\href}[2]{#2}


\begin{thebibliography}{10}

\bibitem{BY2}
Z.D. Bai and J.F. Yao, \emph{Limit theorems for sample eigenvalues in a
  generalized spiked population model}, Preprint arXiv:0806.1141.

\bibitem{BY1}
\bysame, \emph{Central limit theorems for eigenvalues in a spiked population
  model}, Ann. Inst. H. Poincar{\'e} (B) \textbf{44} (2008), 447--474.

\bibitem{BBP}
J.~Baik, G.~Ben~Arous, and S.~P{\'e}ch{\'e}, \emph{Phase transition of the
  largest eigenvalue for nonnull complex sample covariance matrices}, Ann.\
  Prob. \textbf{33} (2005), 1643--1697.

\bibitem{BS}
J.~Baik and J.W. Silverstein, \emph{Eigenvalues of large sample covariance
  matrices of spiked population models}, J. Multivar. Anal. \textbf{97} (2006),
  1382--1408.

\bibitem{BGGM1}
F.~Benaych-Georges, A.~Guionnet, and M.~Ma{\"i}da, \emph{Fluctuations of the
  extreme eigenvalues of finite rank deformations of random matrices}, Preprint
  arXiv:1009.0145.

\bibitem{BGGM2}
\bysame, \emph{Large deviations of the extreme eigenvalues of random
  deformations of matrices}, Prob.\ Theor.\ Rel.\ Fields (2010), 1--49.

\bibitem{BGN}
F.~Benaych-Georges and R.R. Nadakuditi, \emph{The eigenvalues and eigenvectors
  of finite, low rank perturbations of large random matrices}, Adv. Math.
  \textbf{227} (2011), 494--521.

\bibitem{Bthesis}
A.~Bloemendal, \emph{Finite rank perturbations of random matrices and their
  continuum limits}, Ph.D. thesis, University of Toronto, 2011.

\bibitem{BV1}
A.~Bloemendal and B.~Vir{\'a}g, \emph{Limits of spiked random matrices {I}},
  Preprint arXiv:1011.1877.

\bibitem{BV2}
\bysame, \emph{Limits of spiked random matrices {II}}, Preprint
  arXiv:1109.3704.

\bibitem{CDMF2}
M.~Capitaine, C.~Donati-Martin, and D.~F{\'e}ral, \emph{Central limit theorems
  for eigenvalues of deformations of {W}igner matrices}, Preprint
  arXiv:0903.4740.

\bibitem{CDMF1}
\bysame, \emph{The largest eigenvalues of finite rank deformation of large
  {W}igner matrices: convergence and nonuniversality of the fluctuations},
  Ann.\ Prob. \textbf{37} (2009), 1--47.

\bibitem{CDMF3}
M.~Capitaine, C.~Donati-Martin, D.~F{\'e}ral, and M.~F{\'e}vrier, \emph{Free
  convolution with a semi-circular distribution and eigenvalues of spiked
  deformations of {W}igner matrices}, Preprint arXiv:1006.3684.

\bibitem{Chat}
S.~Chatterjee, \emph{A generalization of the {L}indeberg principle}, Ann.\
  Prob. \textbf{34} (2006), 2061--2076.

\bibitem{EKYY1}
L.~Erd{\H{o}}s, A.~Knowles, H.T. Yau, and J.~Yin, \emph{Spectral statistics of
  {E}rd{\H{o}}s-{R}\'enyi graphs {I}: Local semicircle law}, to appear in Ann.
  Prob. Preprint arXiv:1103.1919.

\bibitem{EKYY2}
\bysame, \emph{Spectral statistics of {E}rd{\H{o}}s-{R}\'enyi graphs {II}:
  Eigenvalue spacing and the extreme eigenvalues}, to appear in Comm. Math.
  Phys. Preprint arXiv:1103.3869.

\bibitem{ESY6}
L.~Erd{\H{o}}s, S.~P{\'e}ch{\'e}, J.A. Ramirez, B.~Schlein, and H.T. Yau,
  \emph{Bulk universality for {W}igner matrices}, Comm.\ Pure Appl.\ Math.\
  \textbf{63} (2010), 895--925.

\bibitem{ESY7}
L.~Erd{\H{o}}s, J.~Ramirez, B.~Schlein, T.~Tao, V.~Vu, and H.T. Yau, \emph{Bulk
  universality for {W}igner hermitian matrices with subexponential decay},
  Math. Res. Lett. \textbf{17} (2010), 667--674.

\bibitem{ESY5}
L.~Erd{\H{o}}s, J.~Ramirez, B.~Schlein, and H.T. Yau, \emph{Universality of
  sine-kernel for {W}igner matrices with a small {G}aussian perturbation},
  Electr.\ J.\ Prob. \textbf{15} (2010), 526--604.

\bibitem{ESY2}
L.~Erd{\H{o}}s, B.~Schlein, and H.T. Yau, \emph{Local semicircle law and
  complete delocalization for {W}igner random matrices}, Comm.\ Math.\ Phys.\
  \textbf{287} (2009), 641--655.

\bibitem{ESY1}
\bysame, \emph{Semicircle law on short scales and delocalization of
  eigenvectors for {W}igner random matrices}, Ann.\ Prob. \textbf{37} (2009),
  815--852.

\bibitem{ESY3}
\bysame, \emph{{W}egner estimate and level repulsion for {W}igner random
  matrices}, Int. Math. Res. Not. \textbf{2010} (2010), 436--479.

\bibitem{ESY4}
\bysame, \emph{Universality of random matrices and local relaxation flow},
  Invent. Math. \textbf{185} (2011), no.~1, 75--119.

\bibitem{ESYY}
L.~Erd{\H{o}}s, B.~Schlein, H.T. Yau, and J.~Yin, \emph{The local relaxation
  flow approach to universality of the local statistics of random matrices},
  Ann.\ Inst.\ Henri Poincar{\'e} (B) \textbf{48} (2012), 1--46.

\bibitem{EYY1}
L.~Erd{\H{o}}s, H.T. Yau, and J.~Yin, \emph{Bulk universality for generalized
  {W}igner matrices}, Preprint arXiv:1001.3453.

\bibitem{EYY3}
\bysame, \emph{Rigidity of eigenvalues of generalized {W}igner matrices}, to
  appear in Adv. Math. Preprint arXiv:1007.4652.

\bibitem{EYY2}
\bysame, \emph{Universality for generalized {W}igner matrices with {B}ernoulli
  distribution}, J.\ Combinatorics \textbf{1}, no.~2, 15--85.

\bibitem{FP}
D.~F{\'e}ral and S.~P{\'e}ch{\'e}, \emph{The largest eigenvalue of rank one
  deformation of large {W}igner matrices}, Comm.\ Math.\ Phys.\ \textbf{272}
  (2007), 185--228.

\bibitem{FKoml}
Z.~F{\"u}redi and J.~Koml{\'o}s, \emph{The eigenvalues of random symmetric
  matrices}, Combinatorica \textbf{1} (1981), 233--241.

\bibitem{KY1}
A.~Knowles and J.~Yin, \emph{Eigenvector distribution of {W}igner matrices}, to
  appear in Prob. Theor. Rel. Fields. Preprint arXiv:1102.0057.

\bibitem{Pec}
S.~P{\'e}ch{\'e}, \emph{The largest eigenvalue of small rank perturbations of
  {H}ermitian random matrices}, Prob.\ Theor.\ Rel.\ Fields \textbf{134}
  (2006), 127--173.

\bibitem{SoshPert}
A.~Pizzo, D.~Renfrew, and A.~Soshnikov, \emph{On finite rank deformations of
  {W}igner matrices}, to appear in Ann.\ Inst.\ Henri Poincar{\'e} (B).
  Preprint arXiv:1103.3731.

\bibitem{TV2}
T.~Tao and V.~Vu, \emph{Random matrices: {U}niversality of local eigenvalue
  statistics up to the edge}, Comm.\ Math.\ Phys.\ \textbf{298} (2010),
  549--572.

\bibitem{TV1}
\bysame, \emph{Random matrices: {U}niversality of local eigenvalue statistics},
  Acta Math. \textbf{206} (2011), 1--78.

\bibitem{Wig}
E.P. Wigner, \emph{Characteristic vectors of bordered matrices with infinite
  dimensions}, Ann. Math. \textbf{62} (1955), 548--564.

\end{thebibliography}
\end{document}